\documentclass[a4paper]{amsart}

\newcommand{\nc}{\newcommand}
\nc{\G}{{\Gamma}}
\nc{\BC}{{\mathbb C}}
\nc{\BQ}{{\mathbb Q}}
\nc{\BR}{{\mathbb R}}
\nc{\BZ}{{\mathbb Z}}
\nc{\BP}{{\mathbb P}}
\nc{\BN}{{\mathbb N}}
\nc{\BM}{{\mathbb M}}
\nc{\fH}{{\mathfrak{H}}}
\nc{\vp}{{\varepsilon}}\nc{\dpar}{{\partial}}\nc{\al}{{\alpha}}
\nc{\PSL}{{\mbox{PSL}_2(\BR)}}
\nc{\PS}{{\mbox{PSL}_2(\BZ)}}
\nc{\SL}{{\mbox{SL}_2(\BZ)}}

\def\Re{{\rm Re\,}}
\def\Im{{\rm Im\,}}
\def\I{{\mathrm{i}}}

\newtheorem{thm}{Theorem}[section]
\newtheorem{cor}[thm]{Corollary}
\newtheorem{lem}[thm]{Lemma}
\newtheorem{prop}[thm]{Proposition}
\theoremstyle{definition}

\theoremstyle{remark}
\newtheorem{rem}{Remark}[section]
\newtheorem{ex}{Example}[section]

\numberwithin{equation}{section}

\begin{document}

\title{Critical points of Green's function and Geometric Function Theory}

\author{Bj\"orn Gustafsson}
\address{B. Gustafsson\\
Deparment of Mathematics, KTH, 100 44 Stockholm, Sweden.}

\email{gbjorn@@kth.se }

\author{Ahmed Sebbar}
\address{A. Sebbar\\
Universit\'e Bordeaux I\\
Institut de Math\'ematiques, 33405 Talence, France}

\email{sebbar@@math.u-bordeaux1.fr}

\date{December 7, 2009}

\begin{abstract} We study questions related to critical points of the
Green's function of a bounded multiply connected domain in
the complex plane. The motion of critical points, their limiting
positions as the pole approaches the boundary and the differential geometry
of the level lines of the Green's function are main themes in the paper.
A unifying role is played by
various affine and projective connections and corresponding M\"obius invariant
differential operators. In the doubly connected case the three
Eisenstein series $E_2$, $E_4$, $E_6$ are used.
A specific result is that a doubly connected domain is the disjoint union
of the set of critical points of the Green's function, the set of zeros of
the Bergman
kernel and the separating boundary limit positions for these.
At the end we consider the projective properties of the prepotential
associated to a second order differential operator depending canonically on
the domain.
\end{abstract}

\maketitle

\subjclass{11F03, 30C40, 30F30, 34B27, 53B10, 53B15}

\keywords{Critical point, Green's function, Neumann function,
Bergman kernel, Schiffer kernel, Schottky-Klein prime form,
Schottky double, weighted Bergman space,
Poincar\'{e} metric, Martin boundary, projective structure,
projective connection, affine connection, prepotential,
Eisenstein series}

\thanks{The first author is supported by the Swedish Research Council VR.}

\tableofcontents


\section{Introduction}
The results of this paper have their origin in attempts to
understand trajectories of critical points in a planar multiply
connected domain. We start with studying, for a given bounded
multiply connected domain $\Omega$ in the complex plane, the motion
of critical points of the ordinary Green's function $G(z, \zeta)$ of
$\Omega$ as the pole $\zeta$ moves around. If $\Omega$ has
connectivity ${\texttt{g}}+1$ and the critical points are denoted
$z_1(\zeta),\dots,z_{\texttt{g}}(\zeta)$, we show first of all that
the $z_{j}(\zeta)$ stay within a compact subset $K$ of $\Omega$
(this result has also recently been obtain by A.~Solynin
\cite{Solynin}) and secondly that the limiting positions of the
$z_{j}(\zeta)$ as $\zeta$ approaches the boundary coincide with the
corresponding limiting positions of the zeros of the Bergman kernel.
In the doubly connected case (${\texttt{g}}=1$) it even turns out
that the domain is the disjoint union of the set of critical points
of the Green's function and the set of zeros of the Bergman kernel,
plus the common boundary of these two sets.

The method developed to prove the existence of the compact set $K$ is
remarkably related to (a new type of) Martin compactification and uses
many properties of the Bergman, Schiffer and Poisson kernels. All
these functionals depend in one sense or another on critical points of
the Dirichlet Green's function, and they moreover extend in a natural
way to sections of suitable bundles of the Schottky double of the domain,
which is a compact Riemann surface of genus ${\texttt{g}}$. An important role
here is played by the Schwarz function, which can be interpreted as being the coordinate
transition function between the front side and the back side of the Schottky double.

It is natural in this context to study the relationship between the hyperbolic
metric and level lines of the Green's function.
We show, for example, that the average, taken on a level line of the Green's
function, of the Green's potential $G^\mu$ of any
compactly supported measure $\mu$ is proportional to the total mass of $\mu$.
We also study level lines of the Green's function, and other harmonic functions,
from the point of view of hamiltonian mechanics and differential geometry. One main
observation then is
that such level lines are trajectories, as well as Hamilton-Jacobi geodesics, for
systems with the squared modulus of the gradient of the harmonic function as
hamiltonian.

Various other topics touched on are estimates of the Taylor
coefficients of the regular part of the complex Green's function in
terms of the distance to the boundary, and relationships between
these coefficients and the Poincar\'e metric.

A major part of the paper is devoted to a detailed study of the doubly connected case,
modeled by the annulus. In this situation all calculations can be done explicitly
because we have at our disposal the theory of modular forms (theta functions) and elliptic functions.
For example the dichotomy result mentioned in the beginning is proved in this section.
The Ramanujan formula for the derivatives of the basic Eisenstein
series $E_2$, $E_4$, $E_6$ are fundamental.

Some of the mentioned Taylor coefficients and Eisenstein series
transform under coordinate changes as different kinds of
connections. In the second half of the paper we study affine and
projective connections in quite some depth. In general, projective
connections play a unifying role in the paper, in fact almost all of
our work turns around projective connections, and to some extent
affine connections. One example is that they generate M\"obius
invariant differential operators, called Bol operators and denoted
$\Lambda_m$, for which boundary integral formulas, of Stokes type
but for higher order derivatives, can be proved. These are useful
for studying weighted Bergman spaces in general multiply connected
domains. There is one such Bergman space for each half-integer, and
the elements of the space should be thought of as differentials of
this order (actually integrals if the order is negative). The Bol
operator then is an isometry $\Lambda_m: A_m(\Omega)\to B_m(\Omega)$
from a Bergman type space $A_m(\Omega)$ of differentials of order
$\frac{1-m}{2}$ to a (weighted) Bergman space $B_m(\Omega)$
consisting of differentials of order $\frac{1+m}{2}$. Also
reproducing kernels for these, and some other, spaces are studied.

Along with connections, several questions related to curvature of, for example,
level lines
are discussed. The Study formula finds here its natural meaning.
A further aspect is that the operators $\Lambda_m$ all turn out to be symmetric
powers of a single one,
namely $\Lambda_2$.

In the final section of the paper we try to clarify in our setting the meaning
of a certain prepotential
for second order differential operators which has arisen in some recent physics papers.


\section{Generalities on function theory of
finitely connected plane domains}\label{sec:general}

\subsection{The Green's function and the Schottky double}

Let $\Omega\subset {\mathbb{C}}$ be a bounded domain of
finite connectivity, each boundary component being nondegenerate
(i.e., consisting of more than one point).   The oriented boundary
(having $\Omega$ on its left hand side) is denoted $\displaystyle
\partial \Omega=\Gamma= \Gamma_0 + \Gamma_1+ \cdots+ \Gamma_{{\texttt{g}}}$,
where $ \Gamma_0 $ is the outer  component. We shall in most of the
paper discuss only conformally invariant questions, and then we may
assume that each boundary component $\Gamma_j$ is a smooth analytic
curve. Alternatively, one may think of $\Omega$ as a plane bordered
Riemann surface.

Much of the functions theory on $\Omega$ is conveniently described
in terms of the Schottky double $\hat{\Omega}$ of $\Omega$. This is
the compact Riemann surface of genus ${\texttt{g}}$ obtained, when
$\Omega$ has analytic boundary, by welding $\Omega$ along
$\partial\Omega$, with an identical copy $\tilde\Omega$. Thus, as a
point set $\hat\Omega=\Omega \cup\partial\Omega\cup\tilde\Omega$.
The ``backside" $\tilde \Omega$ is provided with the opposite
conformal structure. This means that if $\tilde{z}\in\tilde\Omega$
denotes the point opposite to $z\in\Omega$ then the map
$\tilde{z}\mapsto \bar{z}$ is a holomorphic coordinate on
$\tilde\Omega$. The construction of the Schottky double generalizes
to any bordered Riemann surface and the result is always a symmetric
Riemann surface, i.e., a Riemann surface provided with an
antiholomorphic involution. In the Schottky double case this is  the
map $J:\Hat{\Omega}\to\Hat{\Omega}$ which exchanges $z$ and
$\tilde{z}$ and which keeps $\partial\Omega$ pointwise fixed.

The Schottky double of a plane domain $\Omega$ has a holomorphic
atlas consisting of only two charts: the corresponding coordinate
functions are the identity map $\phi_1:z\mapsto z$ on $\Omega$ and
the map $\phi_2:\tilde z\mapsto \bar z$ on $\tilde\Omega$. When
$\partial\Omega$ is analytic, as is henceforth assumed, both these
maps extend analytically across $\partial\Omega$ in $\hat{\Omega}$,
hence their domains of definitions overlap and the union covers all
$\hat{\Omega}$. Let
\begin{equation}\label{eq:Schwarztransition}
S=\phi_2 \circ \phi_1^{-1}
\end{equation}
be the coordinate transition function. It is analytic and defined in
a neighbourhood of $\partial\Omega$ in ${\mathbb{C}}$, and on
$\partial\Omega$ it satisfies
\begin{equation}\label{Schwarz}
S(z)=\bar{z} \quad (z\in \partial\Omega).
\end{equation}
Thus it is the {\it Schwarz function} \cite{Davis}, \cite{Davis2},
\cite{Shapiro} of $\partial\Omega$.

Differentiating (\ref{Schwarz}) gives
\begin{equation}\label{Sprime}
d\bar{z}=S'(z)dz \quad {\rm along\,\,} \partial \Omega.
\end{equation}
With $s$ an arc-length parameter along $\partial\Omega$ such that
$\Omega$ lies to the left as $s$ increases,
\begin{equation}\label{eq:T}
T(z)=\frac{dz}{ds}
\end{equation}
is the oriented unit tangent vector on $\partial\Omega$. By
(\ref{Sprime}) and since $|T(z)|=1$,
$$
S'(z)=\frac{1}{T(z)^2}, \quad z\in\partial\Omega.
$$
It follows that $1/T(z)$ extends analytically to a neighbourhood of
$\partial\Omega$ and that it gives a selection of a square-root of
$S'(z)$. This is an important observation because it means that on
the Schottky double of any plane domain there is a canonical choice
of square-root of the canonical bundle, i.e., there is canonical
meaning of the concept of a differential of order one-half, and
hence of a differential of any half-integer order.

A function $f$ on $\hat\Omega$ is most conveniently described as a
pair of functions $f_1$, $f_2$ on $\Omega$, continuously extendable
to $\partial\Omega$, such that
$$
f_1(z)=\overline{f_2(z)} \quad (z\in\partial\Omega).
$$
The formal relations to $f$ in terms of the coordinates functions
$\phi_1$ and $\phi_2$ above are
$$
\begin{cases}
f_1 =f\circ \phi_1^{-1},  \\
f_2 =c\circ f\circ \phi_2^{-1}\circ c,
\end{cases}
$$
where $c$ denotes complex conjugation. It follows, for example, that
$f$ is meromorphic if and only if $f_1$ and $f_2$ are meromorphic.
With the same kind of identifications, a differential of order $m$
on $\hat\Omega$ is represented by a pair of functions $f_1$, $f_2$
on $\Omega$ such that
$$
f_1(z)dz^m=\overline{f_2(z)dz^m}\quad {\rm along\,\,} \partial
\Omega.
$$
This is to be interpreted as $f_1(z)T(z)^m=\overline{f_2(z)T(z)^m}$,
or
\begin{equation}\label{eq:halforderdiff}
f_1(z)=\overline{f_2(z)T(z)^{2m}} \quad (z\in\partial\Omega).
\end{equation}
Clearly this makes unambiguous sense for any
$m\in\frac{1}{2}{\mathbb{Z}}$.

The Green's function $G(z,\zeta)$ of $\Omega$ is, as a function of
$z$ for fixed $\zeta\in\Omega$, defined by the properties
$$
G(z,\zeta)=-\log |z-\zeta| + {\rm harmonic} \quad z\in \Omega,
$$
$$
G(z,\zeta)=0, \quad z\in\partial\Omega.
$$
It is symmetric in $z$ and $\zeta$, $G(z,\zeta)=G(\zeta,z)$, and it
extends to the Schottky double as an``odd" function in each
variable, for example, $G(J(z),\zeta)=-G(z,\zeta)$.

The above extension of the Green's function makes it a special case of
a fundamental potential which exists on any compact Riemann surface,
and which is a suitable starting point for discussing the classical
function and differentials: for any three distinct points $a, b, w$
on a compact Riemann surface $M$ there exists a unique function of
the form
\begin{equation}\label{eq:V}
V(z)= V(z,w;a,b)= -\log|z-a|+ \log|z-b|+ {\rm harmonic},
\end{equation}
normalized by $V(w,w;a,b)= 0$. It has the symmetries
\begin{equation}\label{eq:symm}
V(z,w;a,b)=V(a,b;z,w)=-V(z,w;b,a)
\end{equation}
and the transitivity property
\begin{equation}\label{eq:trans}
V(z,w;a,b)+V(z,w;b,c)=V(z,w;a,c).
\end{equation}
See below for explanations, and also \cite{Schiffer-Spencer}, Ch.~4.
If The Riemann surface is symmetric, with involution $J$, then
\begin{equation}\label{eq:VVJ}
V(z,w;a,b)=V(J(z),J(w);J(a),J(b)).
\end{equation}

\begin{ex}\label{ex:cross-ratio}
In the case of the Riemann sphere, $M={\mathbb{P}}$, we have
$$
V(z,w;a,b)=-\log|(z:w:a:b)| =-\log
\big|\frac{(z-a)(w-b)}{(z-b)(w-a)}\big|,
$$
$(z:w:a:b)$ denoting the classical cross-ratio.
\end{ex}

With $M=\Hat{\Omega}$, the Green function is given in terms of $V$
by
\begin{equation}\label{eq:GVJ}
G(z,\zeta)=\frac{1}{2}V(z,J(z); \zeta, J({\zeta}))= V(z,w; \zeta,
J({\zeta})),
\end{equation}
where $w$ is an arbitrary point on $\partial\Omega$ and where the
second equality follows from (\ref{eq:symm}), (\ref{eq:trans}),
(\ref{eq:VVJ}). Cf. also \cite{Fay}, p.~125f.

The existence of $V(z,w;a,b)$ in general is immediate from classical
potential theory, see e.g. \cite{Schiffer-Spencer},
\cite{Farkas-Kra} or, more generally, from
Hodge theory. In fact, $V$ solves the Poisson equation
$-d^*dV=2\pi(\delta_a-\delta_b)$ on $M$, where the star is the Hodge
star and the Dirac measures in the right member are to be considered
as $2$-form currents; the solution exists because $\int_M
2\pi(\delta_a-\delta_b)=0$ and it is unique up to an additive
constant.

From $V(z,w;a,b)$ much of the classical function theory on $M$ can
be built up. For example,
\begin{equation}\label{eq:upsilonV}
\upsilon_{a-b}(z)=-2\frac{\partial V (z,w;a,b)}{\partial z}dz
=-dV(z)-\I ^*dV(z)
\end{equation}
$$
= \frac{dz}{z-a}-\frac{dz}{z-b}+{\rm analytic}
$$
is the unique abelian differential of the third kind with poles of
residues $\pm 1$ at $z=a,b$ and having purely imaginary periods. The
subscript in $\upsilon_{a-b}$ should be thought of as a divisor, and
the definition extends to any divisor of degree zero. Conversely, we
retrieve $V$ from $\upsilon_{a-b}$ by
\begin{equation}\label{eq:Vupsilon}
V(z,w;a,b)=-\Re \int_w^z \upsilon_{a-b},
\end{equation}
from which the symmetries (\ref{eq:symm}) and transitivity property
(\ref{eq:trans}) of $V$ follow, using Riemann's bilinear relations
for one of the symmetries.

From $V(z,w;a,b)$ one can also construct the
$2{\texttt{g}}$-dimensional space of harmonic differentials on $M$ by
considering the conjugate periods. By (\ref{eq:Vupsilon}), the
harmonic differentials  $dV=-\Re\upsilon_{a-b}$ and
$^*dV=-\Im\upsilon_{a-b}$ do not depend on $w$. The first one is
exact, while the second has certain periods, which depend on $a$,
$b$: if $\gamma$ is a closed curve on $M$ then the function
$$
\phi_\gamma(a,b) =\frac{1}{2\pi}\int_\gamma  \,^*dV(\cdot, w;a,b)
$$
is, away from $\gamma$, harmonic in $a$ and $b$  and makes a unit
jump as $a$ (or $b$) crosses $\gamma$. It follows that, keeping $b$
fixed, $d\phi_\gamma (\cdot,b)$ extends to a harmonic differential
on $M$ with periods given by
$$
\int_\sigma d\phi_\gamma (\cdot,b) =\sigma \times \gamma,
$$
where $\sigma \times \gamma$ denotes the intersection number of
$\sigma$ and $\gamma$. Choosing $\gamma$ among the curves in a
canonical homology basis gives the standard harmonic differentials.

As a further aspect, the right member in (\ref{eq:Vupsilon}) can be
written as a Dirichlet integral:
\begin{equation}\label{eq:VVV}
V(z,w;a,b)=-\frac{1}{2\pi}\int_M dV(\cdot, c;z,w)\wedge
^*dV(\cdot,c;a,b)
\end{equation}
$$
=\frac{1}{4\pi}\Im\int_M \upsilon_{z-w} \wedge\bar{\upsilon}_{a-b},
$$
where $c\in M$ is an arbitrary point. Thus $V$ reproduces itself in
a certain sense, and the equation also expresses that $V(z,w;a,b)$,
besides being the potential of the charge distribution
$\delta_a-\delta_b$ when considered as a function of merely $z$,
also is the mutual energy of the two charge distributions
$\delta_a-\delta_b$ and $\delta_z-\delta_w$. The first equality in
(\ref{eq:VVV}) follows by partial integration, and the second from
the definition of $\upsilon_{a-b}$. For the Green's function the
corresponding equation is
$$
G(z,\zeta)= \frac{1}{2\pi}\int_\Omega dG(\cdot, z)\wedge
^*dG(\cdot,\zeta).
$$

We recall that $\upsilon_{a-b}$ has $2{\texttt{g}}$ zeros (since it
has $2$ poles). With $M=\Hat{\Omega}$ and $a=\zeta$, $b=J({\zeta})$,
half of the zeros are on $\Omega$, and by (\ref{eq:GVJ}),
(\ref{eq:upsilonV}) these are exactly the critical points of
$G(z,\zeta)$ (the points where the gradient vanishes). Hence the
Green's function has exactly ${\texttt{g}}$ critical points on
$\Omega$.


\subsection{The Schottky-Klein prime function}

The harmonic theory on a compact Riemann surface is simple and
intuitive, as indicated above. For the holomorphic theory one
usually prefers to work with abelian differentials of the third kind
normalized, not as $\upsilon_{a-b}$, but so that, in terms of a
canonical homology basis, half of the periods vanish. Following
\cite{Yamada} and \cite{Krichever-Marshakov-Zabrodin} we shall, in
the case of a Schottky double of a planar domain $\Omega$, choose a
canonical homology basis $\displaystyle \{
\alpha_1,\dots,\alpha_{{\texttt{g}}}, \,\beta_1,\dots,\beta_{{\texttt{g}}}\}$
such that $\alpha_j$ goes from $\Gamma_0$ to $\Gamma_j$ on $\Omega$
and back to $\Gamma_0$ along the same track on the back-side
$\Tilde{\Omega}$, and such that $\beta_j=\Gamma_j$ ($j=1,\dots,
{\texttt{g}}$). Thus the basis is symmetric with respect to the
involution $J$, more precisely $J(\alpha_j)= -\alpha_j$,
$J(\beta_j)= \beta_j$.

We denote by $\omega_{a-b}$ the abelian differential of the third
kind with the same singularities as $\upsilon_{a-b}$ but normalized
so that the $\alpha_j$-periods vanish:
\begin{equation}\label{eq:normomega}
\int_{\alpha_j} \omega_{a-b}=0,\quad j=1,\dots,{\texttt{g}}.
\end{equation}
This makes sense only for $a,b\notin\alpha_j$, and with the
$\alpha_j$ being fixed curves. Hence $\omega_{a-b}$ is less
canonical than $\upsilon_{a-b}$ but it has the advantage of
depending analytically on $a$, $b$, whereas for $\upsilon_{a-b}$ the
dependence is only harmonic in general.

\begin{rem} As to notation, the differential we denote by $\omega_{a-b}$
is in Fay \cite{Fay}, Schiffer-Spencer \cite{Schiffer-Spencer},
Farkas-Kra \cite{Farkas-Kra} denoted, respectively, $\omega_{a-b}$,
$-d\omega_{ab}$, $\tau_{ab}$. For our $\upsilon_{a-b}$ the
corresponding list is: $\Omega_{a-b}$, $-d\Omega_{ab}$,
$\omega_{ab}$. Also, Fay \cite{Fay}, and Hejhal \cite{Hejhal} use a
homology basis with switched roles between the $\alpha_j$ and
$\beta_j$ curves, but, as noticed by A.~Yamada \cite{Yamada}, the
present choice has certain advantages (see
Lemma~\ref{lem:omegaupsilon} below).
\end{rem}

The integral $\int_w^z \omega_{a-b}$, with unspecified path of
integration, is locally holomorphic in all variables, but
multivalued. To cope with the $2\pi\I$ indeterminacy coming from the
poles one may form the exponential. In the case of the Riemann
sphere this simply gives the cross-ratio between $z$, $w$, $a$, $b$:
$$
\exp \int_w^z \omega_{a-b}=(z:w:a:b).
$$
For Riemann surfaces of genus ${\texttt{g}}>0$, the exponential
remains multivalued, with multiplicative periods (cf.
\cite{Farkas-Kra}), Ch.III.9). However, one can still write it as a
kind of cross-ratio if one is willing to accept further
multivaluedness:
\begin{equation}\label{eq:expEEEE}
\exp \int_w^z \omega_{a-b}=\frac{E(z,a)E(w,b)}{E(z,b)E(w,a)}.
\end{equation}
Here $E(z,\zeta)$ is the Schottky-Klein prime function (prime form),
which should be regarded as a differential of order $-\frac{1}{2}$
in each of the variables, but as such still has multiplicative
periods. Near the diagonal it behaves like $z-\zeta$. The exact
definition of $E(z,\zeta)$ is usually given in terms of theta
functions on the Jacobi variety, see \cite{Fay}, \cite{Hejhal}, \cite{Mumford}
and below. For the Schottky doubles of a plane region there is also a
representation in terms of the Schottky uniformization of the
domain, see for example \cite{Crowdy2}.

It should be remarked that, in the Schottky double case, the kind of
differential of order $-\frac{1}{2}$ referred to for $E(z,\zeta)$
has no representation in the form (\ref{eq:halforderdiff}) (with
$m=-1$). The bundle of half-order differentials defined by
(\ref{eq:halforderdiff}) is ``even'' and does not allow any global
holomorphic section, whereas the one needed for $E(z,\zeta)$ should
be ``odd'', which does allow for a holomorphic section. The
definition of such a bundle requires a finer atlas than the one
consisting of only $\Omega$ and $\tilde{\Omega}$ for its
representation. Of the $2^{2{\texttt{g}}}$ bundles of half-order
differentials, $2^{{\texttt{g}}-1}(2^{{\texttt{g}}}-1)$ are odd and
$2^{{\texttt{g}}-1}(2^{{\texttt{g}}}+1)$ are even, see \cite{Hejhal}.

Now, quite remarkably, in case the Riemann surface is the Schottky
double of a plane domain the two abelian differentials
$\omega_{a-b}$ and $\upsilon_{a-b}$ coincide when $a$ and $b$ are
symmetrically opposite points:

\begin{lem}\label{lem:omegaupsilon}
For $M=\Hat{\Omega}$ and with canonical homology basis chosen as
above,
$$
\omega_{\zeta-J(\zeta)}=\upsilon_{\zeta-J(\zeta)}.
$$
\end{lem}

\begin{rem}
The lemma does not hold if the homology basis is chosen in a
different way, for example with switched roles between the
$\alpha_j$ and $\beta_j$ curves. Note also that even though
$\omega_{a-b}$ makes a jump by $2\pi\I$ as $a$ or $b$ crosses
$\alpha_j$, there is no such jump for $\omega_{\zeta-J(\zeta)}$
because $\zeta$ and $J(\zeta)$ cross $\alpha_j$ simultaneously and
the two contributions cancel.
\end{rem}

\begin{proof}
One simply has to notice that $\upsilon_{\zeta-J(\zeta)}$ satisfies
the normalization (\ref{eq:normomega}) of $\omega_{\zeta-J(\zeta)}$.
Expressed in terms of $V$, since $dV$ already is exact, the
statement to be proven becomes, for any $w\in\partial\Omega$,
$$
\int_{\alpha_j} \,^*dV(\cdot, w;\zeta,J(\zeta))=0, \quad
(j=1,\dots,{\texttt{g}}).
$$
That these periods vanish follows from the symmetry properties
(\ref{eq:VVV}), (\ref{eq:VVJ}) of $V$ with respect to $J$.
\end{proof}

For a general Riemann surface,  the differential $\omega_{a-b}$ can
be recovered from the prime form as a logarithmic derivative of
(\ref{eq:expEEEE}):
\begin{equation}\label{eq:omegadlogE}
\omega_{a-b}(z)= d\log \frac{E(z,a)}{E(z,b)}.
\end{equation}
From this one sees clearly that $E(z,\zeta)$ depends on the homology
basis, since $\omega_{a-b}$ does. The Green's function of a planar
domain $\Omega$ is most directly related to $\upsilon_{a-b}$ on
$\Hat{\Omega}$ as in (\ref{eq:GVJ}), (\ref{eq:upsilonV}), but in
view of Lemma~\ref{lem:omegaupsilon} one can equally well use
$\omega_{a-b}$. Therefore, (\ref{eq:omegadlogE}) gives the following
expression of the Green's function of $\Omega$ in terms of the prime
form on $M=\Hat{\Omega}$.
$$
G(z,\zeta)=-\log|\frac{{E}(z,\zeta)}{{E}(z,J(\zeta))}| \quad (z,
\zeta\in \Omega).
$$
See \cite{Yamada} for further details, and also \cite{Crowdy2},
\cite{Krichever-Marshakov-Zabrodin} for other aspects.

We introduce next the harmonic measures $u_j$,
$j=1,\dots,{\texttt{g}}$, i.e., the harmonic functions in $\Omega$
defined by having the boundary values $u_j=\delta_{kj}$ on
$\Gamma_k$. It is easy to see that their differentials $du_j$ extend
to the Schottky double as everywhere harmonic differentials with
$$
\int_{\alpha_k} du_j = 2\delta_{kj},\quad \int_{\alpha_k} \,^*du_j
=0.
$$
Here the second equation is a consequence of the symmetry of the
extended differential under $J$. Thus, the
$d{\mathcal{U}}_j=\frac{1}{2}(du_j+\I ^*du_j)$
($j=1,\dots,{\texttt{g}}$) constitute the canonical basis of abelian
differentials of the first kind (everywhere holomorphic
differentials), the period matrix with respect to the $\alpha_k$
curves being the identity matrix. For the $\beta_k$-periods we have
$$
\int_{\beta_k} du_j = 0, \quad \int_{\beta_k} \,^*du_j
=\int_{\partial\Omega}u_k \,^*du_j =\int_{\Omega}du_k
\wedge\,^*du_j,
$$
where the latter make up a positive definite matrix. On setting
$$
\tau_{kj}= \frac{1}{2}\int_{\beta_k} du_j +\I\,^*du_j
$$
we thus have $\Re \tau_{kj}=0$ and that the matrix $(\Im \tau_{kj})$
is positive definite. With ${\mathcal{U}}_j(z)=\int^z
\frac{1}{2}(du_j+\I ^*du_j) =\int^z \frac{\partial u_j}{\partial
z}dz$ denoting the corresponding abelian integrals (multivalued),
the map
$$
z\mapsto {\mathcal{U}}(z) =({\mathcal{U}}_1(z),\dots
{\mathcal{U}}_{{\texttt{g}}}(z))
$$
defines, up to a shift, the Abel map into the Jacobi variety
${\mathbb{C}}^{{\texttt{g}}}/({\mathbb{Z}}^{\texttt{g}}+\tau{\mathbb{Z}}^{\texttt{g}})$,
$\tau=(\tau_{kj})$.

At this point we can make the definition of the Schottky-Klein prime
function slightly more precise. The first order theta function with
(half-integer) characteristics
$
\begin{bmatrix}
\delta\\
\epsilon
\end{bmatrix},
$
where $\delta, \epsilon \in \frac{1}{2}{\mathbb{Z}}^{\texttt{g}}$ are
row vectors, is defined by
\begin{equation}\label{eq:theta}
\vartheta
\begin{bmatrix}
\delta\\
\epsilon
\end{bmatrix}(w)
=\vartheta
\begin{bmatrix}
\delta\\
\epsilon
\end{bmatrix}(w;\tau)
= \sum_{m\in \BZ^{\texttt{g}}}
\exp\left[{\I\pi(m+\delta)\tau(m+\delta)^{t}+ 2\pi\I(w+ \epsilon)(m+
\delta)^{t}}\right]
\end{equation}
for $w=(w_1,\dots,w_{\texttt{g}})\in {\mathbb{C}}^{\texttt{g}}$. The
superscript $t$ denotes transposition of a matrix. To define the
prime form the characteristics should first of all be chosen to be
odd, i.e., so that $4\delta\epsilon^t$ is an odd number (e.g.,
$\delta=(\frac{1}{2}, 0\dots,0)$, $\epsilon=(\frac{1}{2},
0\dots,0)$). For simplicity, set $\vartheta_*=\vartheta
\begin{bmatrix}
\delta\\
\epsilon
\end{bmatrix}$
whenever such a choice of $\delta$, $\epsilon$ has been made, and
introduce
$$
c_j=\frac{\partial \vartheta_*}{\partial w_j}(0).
$$
The $\delta$ and $\epsilon$ should in addition be chosen to be
non-singular, i.e., so that not all of the constants $c_j$ vanish.
This can always be done, see \cite{Fay}. Then the Schottky-Klein
prime form is, when considered as defined on $\Hat{\Omega}$,
$$
E(z,\zeta)= \frac{\vartheta_* ({\mathcal{U}}(z)-{\mathcal{U}}(\zeta))}
{\sqrt{\sum_{j=1}^{{\texttt{g}}}c_j d{\mathcal{U}}_j(z)}
\sqrt{\sum_{j=1}^{{\texttt{g}}}c_j d{\mathcal{U}}_j(\zeta)}}.
$$
See \cite{Krichever-Marshakov-Zabrodin} (Appendix) for further
details using the present point of view.


\section{Critical points of the Green's function}

\subsection{The Bergman and Poisson kernels}

The Bergman and Schiffer kernels \cite{Bergman},
\cite{Schiffer-Spencer}, $K(z, \zeta)$ and $L(z, \zeta)$
respectively, will be discussed in some detail later on. For the
moment we just recall their representations in terms of the Green's
function and the potential $V(z,w;a,b)$ in (\ref{eq:V}):
\begin{equation}\label{eq:K}
K(z, \zeta) = -\frac{2}{\pi}\frac{\partial^2 G}{\partial z \partial
\bar\zeta}(z,\zeta) =\, \, \frac{1}{\pi}\frac{\partial^2 V}{\partial
z\partial \bar \zeta}(z,J(z); \zeta,J(\zeta)),
\end{equation}
\begin{equation}\label{eq:L}
 L(z, \zeta)= -\frac{2}{\pi}\frac{\partial ^2
G}{\partial z \partial \zeta}(z,\zeta) =\, \,
\frac{1}{\pi}\frac{\partial^2 V}{\partial z
\partial \zeta}(z,J(z); \zeta ,J(\zeta)).
\end{equation}
By the symmetry of $G(z,\zeta)$ (or $V$) we have
$L(z,\zeta)=L(\zeta, z)$ and $K(z,\zeta)=\overline{K(\zeta,z)}$.

Since $G(z,\zeta)=0$ for $z\in\partial\Omega$, $\zeta\in \Omega$,
\begin{equation}\label{eq:KL}
L(z,\zeta)dz+\overline{K(z,\zeta)dz}=0
\end{equation}
along $\partial\Omega$ (with respect to $z$). This means that, for
any fixed $\zeta$, the pair $K(z,\zeta)dz$, $L(z,\zeta)dz$ combines
into a meromorphic differential on $\Hat\Omega$. It has a pole of
order two residing in $L(z,\zeta)$:
\begin{equation}\label{eq:l}
L(z, \zeta)= \frac{1}{\pi}\frac{1}{(z- \zeta)^2}- \ell (z, \zeta),
\end{equation}
where $\ell (z,\zeta)$ (the ``$\ell$-kernel'') is regular in both
variables. Consequently, $K(z,\zeta)dz$, $L(z,\zeta)dz$ have
altogether $2{\texttt{g}}$ zeros. For $\zeta\in\partial\Omega$ the
zeros are, by (\ref{eq:KL}) with switched roles between $z$ and
$\zeta$, equally shared: ${\texttt{g}}$ zeros for $K(z,\zeta)dz$ and
$L(z,\zeta)dz$ at the same points, and no zeros on $\partial\Omega$
as we shall see in the next section.  A further consequence of
(\ref{eq:KL}) is that when $\partial\Omega$ is analytic, the kernels
$K(z,\zeta)$, $L(z,\zeta)$ have analytic extensions across
$\partial\Omega$ in ${\mathbb{C}}$.

Using the symmetries of $K(z,\zeta)$ and $L(z,\zeta)$, and
(\ref{eq:KL}) one finds that the double differentials
$K(z,\zeta)dzd\bar{\zeta}$, $L(z,\zeta)dzd{\zeta}$ are ``real'' on
the boundary:
$$
L(z,\zeta)dzd{\zeta}\in {\mathbb{R}},\quad K(z,\zeta)dzd\bar{\zeta}\in
{\mathbb{R}}\quad (z,\zeta \in\partial\Omega).
$$
The precise meaning of for example the first equation is that
$L(z,\zeta)T(z)T(\zeta)\in {\mathbb{R}}$, where $T(z)$ is the tangent
vector at $z\in\partial\Omega$ (see (\ref{eq:T})).

The Poisson kernel of $\Omega$ is the normal derivative of the Green
function when one of the variables is on the boundary:
$$
P(z, \zeta)= -\frac{1}{2\pi}\frac{\partial G}{
\partial n_z}(z,\zeta),\quad z\in \partial\Omega,\,\zeta\in \Omega.
$$
Here $\frac{ \partial }{\partial n_z}$ denotes the outward normal
derivative at $z\in\partial\Omega$. With $ds=ds_z$ denoting
arc-length with respect to $z$, the definition may also be written
in differential form in several ways, for example (with $d=d_z$)
$$
P(z, \zeta)ds_z =-\frac{1}{2\pi} {}^*d
G(z,\zeta)=-\frac{1}{\pi}\Im\frac{\partial G}{\partial z}(z,\zeta)dz
$$
$$
=-\frac{1}{2\pi}\Im\frac{\partial V}{\partial z}(z,J(z);\zeta,
J({\zeta}))dz =-\frac{1}{2\pi}\Im\upsilon_{\zeta-J({\zeta})}(z).
$$


\subsection{Critical points}

In this subsection and the next we show that, as the pole moves
around, the set of critical points of the Green's function of
$\Omega$ stay within a compact subset of $\Omega$, and that the
limiting positions, as the pole tends to $\partial\Omega$, coincide
with the zero set of the Bergman kernel in this boundary limit. The
latter result was the main motivation for the present work. Related 
results have been obtained by S.~Bell \cite{Bell} (section~30) and
A.~Solynin \cite{Solynin}.

For fixed $z \in\partial \Omega$, the map $\zeta\mapsto  P(z,
\zeta)$ is harmonic and strictly positive in $\Omega$ and vanishes
on $\partial \Omega \setminus \lbrace z\rbrace$. It extends, when
$\partial\Omega$ is analytic, to a harmonic function in a
neighborhood (in ${\mathbb{C}}$) of $\partial \Omega$ except for a
pole at $\zeta=z$. Since $P(z,\cdot)$ attains its minimum in $\bar
\Omega \setminus \lbrace z\rbrace$ on $\partial \Omega$, the Hopf
maximum principle \cite{Protter-Weinberger} shows that
$\displaystyle\frac{\partial P}{
\partial n_\zeta}(z,\zeta)< 0$ for all $\zeta \in \partial \Omega \setminus
\lbrace z\rbrace$. Thus
\begin{equation}\label{1}
\frac{\partial^2 G}{\partial n_z \partial  n_\zeta}(z,\zeta)>
0,\quad z, \zeta \in \partial \Omega,\quad z \neq \zeta.
\end{equation}

The fact that this double normal derivative is nonnegative can be
more directly understood by interpreting the left member of
(\ref{1}) as a double difference quotient: with differences
$\triangle z$, $\triangle\zeta$ pointing in the normal direction
into $\Omega$ from $z,\zeta\in\partial\Omega$, respectively, we have
$$
\frac{\partial^2 G(z, \zeta)}{\partial n_z
\partial n_\zeta}=
$$
$$
=\lim_{|\triangle z|, |\triangle \zeta| \to 0}\frac{ G(z+\triangle
z, \zeta+\triangle\zeta)- G(z+\triangle z, \zeta)- G(z,
\zeta+\triangle\zeta) + G(z, \zeta)}{|\triangle z| |\triangle \zeta|
}.
$$
This is obviously nonnegative because all terms in the numerator
vanish except the first: $G(z+\triangle z, \zeta+\triangle\zeta)>0$.

Complementary to (\ref{1}) we have
$$
\frac{\partial^2 G}{\partial s_\zeta \partial  s_z}(z,\zeta)
=\frac{\partial^2 G}{\partial s_\zeta \partial  n_z}(z,\zeta)
=\frac{\partial^2 G}{\partial n_\zeta \partial  s_z}(z,\zeta) =0
\quad (z, \zeta \in \partial \Omega,\,\, z \neq \zeta).
$$
It follows that, for $z,\zeta\in\partial\Omega$, $z\ne\zeta$,
$$
\frac{\partial^2 G}{\partial z  \partial \zeta}(z,\zeta) \neq 0,
\quad \frac{\partial^2 G}{\partial z  \partial \bar{\zeta}}(z,\zeta)
\neq 0, \quad
\frac{\partial^2 G}{\partial z  \partial
n_\zeta}(z,\zeta) \neq 0,
$$
because these derivatives contain derivations in the normal
direction in each of the variables. In view of the singularity of
type $\displaystyle {1}/{(z- \zeta)^2}$ at $z= \zeta$ and of the
continuity of $\displaystyle \frac{\partial^2 G}{\partial z
\partial \zeta}$ and $\displaystyle \frac{\partial^2 G}{\partial z
\partial \bar\zeta}$ as mappings from a neighborhood of $\displaystyle
\partial \Omega\times\partial \Omega \subset \BC \times \BC$ to the
Riemann sphere it is clear that the above inequalities persist to
hold for $z=\zeta$, and also that the quantities are bounded away
from zero:
\begin{equation}\label{2}
\big| \frac{\partial^2 G}{\partial z \, \partial \zeta}\big| \geq
c>0, \quad \big|\frac{\partial^2 G}{\partial z\,
\partial\bar\zeta}\big| \geq c > 0
\end{equation}
in a full  neighborhood of $\displaystyle \partial \Omega\times
\partial \Omega$ and for some constant $c> 0$.
Note that (\ref{2}) says that $K(z,\zeta)$, $L(z,\zeta)$ are bounded
away from zero in a neighborhood of
$\partial\Omega\times\partial\Omega$ when $\partial\Omega$ is
analytic. With $z$ in a neighborhood of $\partial\Omega$ and
$\zeta\in\partial\Omega$ we also infer
\begin{equation} \label{3}
\big|\frac{\partial^2 G}{\partial z\, \partial n_\zeta}\big| \geq c
> 0.
\end{equation}

Since $\displaystyle \frac{\partial G}{\partial z}(z, \zeta)= 0$ for
$\displaystyle \zeta\in \partial \Omega$ we obtain, on integrating
(\ref{3}) with respect to $\zeta$ in the normal direction,
$$
{\vert\frac{\partial G}{\partial z}(z, \zeta)\vert
\geq c \,{\rm dist} (\zeta, \partial \Omega)}
$$
for $z, \zeta$ in a neighborhood of $\partial \Omega$ and for some
$c>0$. In particular there exists a compact set $K\subset \Omega$
such that
\begin{equation}\label{eq:grad}
\displaystyle \frac{\partial G}{\partial z}(z, \zeta) \neq 0
\end{equation}
for all $z,\zeta \in \Omega\setminus K$. Now we are ready to
conclude the following, also obtained (using slightly different
arguments) by A.~Solynin \cite{Solynin}.

\begin{thm}\label{thm:critical}
Let $\Omega\subset \BC$ be a bounded finitely connected domain such
that each component of $\BC\setminus \Omega$ consists of a least two
points. Then there exists a compact set $K\subset \Omega$ such that
\begin{eqnarray}\label{4}
\displaystyle{\frac{\partial G}{\partial z}(z, \zeta)}& \neq 0 ,& \,
z\in \Omega\setminus K,\quad\zeta\in \Omega,\\
\displaystyle{\frac{\partial^2 G}{\partial z  \partial
\zeta}(z,\zeta)}&
\neq 0 ,&\, z,\zeta \in \Omega\setminus K,\\
\displaystyle{\frac{\partial^2 G}{\partial z  \partial \bar
\zeta}(z,\zeta)}& \neq 0,&\,  z,\zeta \in \Omega\setminus K.
\end{eqnarray}
\end{thm}

\begin{proof}
The statements are conformally invariant, so it is enough to prove
them when $\Omega$ has smooth analytic boundary. By (\ref{2}),
(\ref{3}), (\ref{eq:grad}) the desired inequalities are valid for
$z,\zeta \in \Omega\setminus K$ and it remains only to prove that,
in the first inequality, we can allow all $\zeta\in \Omega$ by
possibly enlarging the compact set $K$. If this were not true, then
there would exist a sequence $\{(z_n, \zeta_n)\}$ with
$$
{\frac{\partial G}{\partial z}(z_n, \zeta_n)} = 0, \quad z_n
\rightarrow z \in \partial \Omega, \quad \zeta_n \rightarrow \zeta
\in \bar \Omega.
$$
According to (\ref{eq:grad}), $\zeta \in  \Omega$. But then also
$\displaystyle{\frac{\partial G}{\partial z}(z, \zeta)= 0}$, which
however cannot be true because $\displaystyle{\frac{\partial
G}{\partial z}(z, \zeta)}$ is a nonzero constant factor times the
Poisson kernel $P(z, \zeta)$, which is strictly positive for $\zeta
\in \Omega, z \in \partial \Omega$. This contradiction finishes the
proof.
\end{proof}

\begin{rem} With slight modifications the proof also works in
several real variables: if $\Omega\subset{\mathbb{R}}^n$ has real
analytic boundary then there exists a compact set $K\subset \Omega$
such that $\nabla_x G(x,y)\ne 0$ for all $x\in \Omega\setminus K$,
$y\in \Omega$. The crucial observation is that (\ref{1}) persists to
hold, so that $\displaystyle\sum_{i,j=1}^n \frac{\partial^2
G(x,y)}{\partial x_i
\partial y_j}\xi_i \eta_j\ne 0$  for $x,y\in \partial \Omega$
whenever $\xi$, $\eta$ are nontangential vectors at $x$ and $y$
respectively.
\end{rem}

\begin{ex}\label{ex:disk}
The boundary behavior and the nature of the pole are illustrated by
the case of the unit disk: $\Omega={{\mathbb{D}}}=\{
z\in{{\mathbb{C}}}: |z|<1 \}$. In this case $\Hat \Omega =
{\mathbb{P}} = \BC \cup \lbrace\infty \rbrace$ with involution
$J:z\mapsto 1/\bar{z}$, and the Schwarz function is $\displaystyle
S(z)= 1/z$. The fundamental potential was given in
Example~\ref{ex:cross-ratio}, and the Green's function is
$$
G(z,\zeta)=-\log\vert  \frac{z- \zeta}{1-z \bar \zeta}\vert.
$$
It follows that
$$
\frac{\partial G(z,\zeta)}{\partial z}=-\frac{1-\vert
\zeta\vert^2}{2(z-\zeta)(1-z\bar{\zeta})},
$$
$$
\frac{\partial^2 G}{\partial z\,
\partial\zeta}=- \frac{1}{2}\frac{1}{(z- \zeta)^2},\quad
\frac{\partial^2 G}{\partial z\,
\partial\bar\zeta} =- \frac{1}{2} \frac{1}{(1-z \bar\zeta)^2},
$$
$$
P(z, \zeta) = \frac{1- \vert \zeta\vert ^2}{\vert z -\zeta \vert^2}.
$$
As for the estimates (\ref{2}),  (\ref{3}) we note that
$\displaystyle \vert z- \zeta\vert \leq 2,\, \, \vert1- z\bar
\zeta\vert\leq 2$ so that, for $\displaystyle z,\zeta \in
\partial {\mathbb{D}}$,
$$
\big| \frac{\partial^2 G}{\partial z\,
\partial \zeta}\big| \geq \frac{1}{8},
\quad \big| \frac{\partial^2 G}{\partial z\,
\partial\bar \zeta}\big| \geq \frac{1}{8},
\quad \big| \frac{\partial^2 G}{\partial z\,
\partial n_\zeta}\big| \geq \frac{1}{8}.
$$
\end{ex}


\subsection{Critical points in the boundary limit}\label{sec:criticalboundary}

When $\zeta\in \partial \Omega$, $\displaystyle{\frac{\partial
G}{\partial z}(z, \zeta)}= 0$ for all $z\in \Omega$, but  by
``blow-up'' one can still speak of a nontrivial limit of
$\displaystyle{\frac{\partial G}{\partial z}(z, \zeta)}$ as
$\zeta\to\partial \Omega$. In a certain sense
$\displaystyle{\frac{\partial^2 G}{\partial z\partial n_\zeta}(z,
\zeta)}$, and hence the Bergman and Schiffer kernels, represents
this limit, but there is also a representation in terms of a Martin
type construction, which we shall discuss in the next section.
Throughout this subsection we assume that $\partial\Omega$ is
analytic.

Fix a point $a\in \Omega$ such that $\partial_z G(a,\zeta)\ne 0$ for
all $\zeta\in\Omega$ and $K(a,\zeta) \neq 0$ for all
$\zeta\in\partial\Omega$. This is possible (with $a$ close enough to
the boundary) by Theorem~\ref{thm:critical} and what precedes it,
viz. (\ref{2}). Then define
\begin{eqnarray}\label{Martin}
F(z,\zeta)= \frac{\partial_z G (z,\zeta)}{\partial_z
G(a,\zeta)},\quad z,\zeta \in \Omega.
\end{eqnarray}
As a function of $z$, $F(z,\zeta)$ is  meromorphic in $\Omega$ with
a pole at $\zeta$. It is normalized so that $F(a,\zeta)= 1$, which
prevents it from degeneration as $\zeta \rightarrow\partial \Omega$.
In Example~\ref{ex:disk}, for the unit disk, we see this from
$$\displaystyle
F(z,\zeta)= \frac{(a-\zeta)(1-a\bar \zeta)}{(z-\zeta)(1- z\bar
\zeta)} = \frac{(a-\zeta)(a-{1}/{\bar \zeta})}{(z-\zeta)(z-{1}/{\bar
\zeta})}.
$$
As a comparison,
$$
\frac{ K(z,\zeta)}{ K(a,\zeta)}= \frac{(1-a\bar \zeta)^2}{ (1- z\bar
\zeta)^2} = \frac{(a-{1}/{\bar \zeta})^2}{(z-{1}/{\bar \zeta})^2 }.
$$
Since $\zeta=1/\bar{\zeta}$ on the boundary, it follows that
$F(z,\zeta)$ and $\displaystyle\frac{ K(z,\zeta)}{ K(a,\zeta)}$
coincide there.

The following theorem shows that the above is what happens in
general.

\begin{thm}\label{5}
Assume $\partial\Omega$ is analytic and choose $a\in \Omega$ as
above. Then the function $F(z,\zeta)$, originally defined in
$\Omega\times\Omega$, extends continuously to
$\Omega\times\overline\Omega$, and on $\Omega\times\partial\Omega$
it agrees with $\displaystyle\frac{K(z,\zeta)}{K(a,\zeta)}$.
Specifically, if $(z_n,\zeta_n)\in \Omega\times\Omega$, $(z,\zeta)
\in \Omega\times\partial\Omega$ and $(z_n,\zeta_n)\to(z,\zeta)$ as
$n\to \infty$, then
$$
\lim_{n\to \infty} F(z_n, \zeta_n)= \frac{ K(z,\zeta)}{K(a,\zeta)}.
$$
\end{thm}

\begin{proof}
Let $(z_n,\zeta_n)\to(z,\zeta)$ as in the statement and let
$\eta_n\in\partial\Omega$ be closest points on the boundary to the
$\zeta_n$:
$$
|\zeta_n-\eta_n|=d(\zeta_n)=d(\zeta_n,\partial\Omega).
$$
With $T(\eta_n)$ the tangent vector at $\eta_n\in\partial\Omega$, we
have
$$
\frac{\partial G}{\partial z}(z_n, \eta_n)= 0,
$$
$$
\frac{\partial^2 G}{\partial z \partial \zeta}(z_n, \eta_n)\,
T(\eta_n) + \frac{\partial^2 G}{\partial z \partial \bar \zeta}(z_n,
\eta_n) \,\overline{T(\eta_n)} = 0.
$$
Therefore, Taylor expansion with respect to $\zeta$ of
$\frac{\partial G}{\partial z}(z_n, \zeta)$ at $\zeta=\eta_n$ gives
$$
\frac{\partial G}{\partial z}(z_n, \zeta_n) = \frac{\partial^2
G}{\partial z\partial \zeta} (z_n, \eta_n)(\zeta_n- \eta_n)
+\frac{\partial^2 G}{\partial z\partial \bar\zeta} (z_n,
\eta_n)(\bar\zeta_n- \bar\eta_n) + {\mathcal{O}}(\vert \zeta_n-
\eta_n\vert^2)
$$
$$
=2\I\frac{\partial^2 G}{\partial z \partial \zeta}(z_n, \eta_n)\,
T(\eta_n)\Im \frac{\zeta_n- \eta_n}{T(\eta_n)}
+{\mathcal{O}}({d}(\zeta_n)^2)
$$
$$
 =2\I\frac{\partial^2
G}{\partial z \partial \zeta}(z_n, \eta_n)\, T(\eta_n){d}(\zeta_n)
+{\mathcal{O}}({d}(\zeta_n)^2).
$$
Note that $\Im \frac{\zeta_n- \eta_n}{T(\eta_n)}={d}(\zeta_n)$. The
remainder term ${\mathcal{O}}({d}(\zeta_n)^2)$ in principle depends
on $z_n$, but it is uniformly small as $z_n\to z\in\Omega$.

We now obtain
$$
F(z_n,\zeta_n)=\frac{\partial_z G (z,\zeta_n)}{\partial_z
G(a,\zeta_n)} =\frac{2\I\frac{\partial^2 G}{\partial z \partial
\zeta}(z_n, \eta_n)T(\eta_n){d}(\zeta_n)
+{\mathcal{O}}({d}(\zeta_n)^2)}{2\I\frac{\partial^2 G}{\partial z
\partial \zeta}(a, \eta_n)T(\eta_n){d}(\zeta_n)
+{\mathcal{O}}({d}(\zeta_n)^2)}
$$
$$
=\frac{K(z_n,\eta_n)+{\mathcal{O}}({d}(\zeta_n))}{K(a,\eta_n)+{\mathcal{O}}({d}(\zeta_n))}.
$$
Since $K(a,\zeta)$ is bounded away from zero for $\zeta$ close to
$\partial\Omega$ the statements of the theorem follow.
\end{proof}

As a consequence we have
\begin{cor}\label{cor:critical}
The limit set of the set of critical points of $G(z,\zeta)$ as
$\zeta\to\partial\Omega$ is exactly the set of zeros of $K(z,\zeta)$
for $\zeta\in\partial\Omega$. More precisely, for $z\in \Omega,
\zeta\in \partial \Omega$, the following statements are equivalent.
\begin{enumerate}

\item[(i)] $K(z,\zeta)= 0$.

\item[(ii)] $ L(z,\zeta)= 0$.

\item[(iii)]  There exist $(z_n, \zeta_n)\in \Omega\times \Omega$,
$(n=1, 2,\cdots)$ with $ \partial_z G(z_n,\zeta_n)= 0$  such that
$(z_n, \zeta_n)\to (z, \zeta)$  as $n\to \infty$.

\item[(iv)] For each sequence $\{\zeta_n\}_{n\geq 1} \subset \Omega$ such that
$\zeta_n \to \zeta$, there exist $z_n \in \Omega$ with $z_n \to z$
such that $\partial_z G(z_n, \zeta_n)= 0$.

\end{enumerate}
\end{cor}

\begin{proof}

By (\ref{eq:KL}) (with switched roles of $z$ and $\zeta$), (i) and
(ii) are equivalent, and clearly (iv) implies (iii).

If $\partial_z G(z_n,\zeta_n)=0$, i.e., $F(z_n , \zeta_n) =0$, then
it is immediate from Theorem~\ref{5} that $K(z, \zeta)= 0$. Thus
(iv) implies (i). Conversely, assume for instance that $K(z, \zeta)=
0$, $\frac{\partial}{\partial z}K(z, \zeta)\ne 0$  and let $\zeta_n
\in \Omega$, $\zeta_n\to\zeta$. The functions $F(\cdot, \zeta_n)$
and $ K(\cdot, \zeta)$ have equally many poles and zeros (namely $2$
poles and $2{\texttt{g}}$ zeros seen in the Schottky double),
and since $ F(\cdot, \zeta_n)$ tends
to a constant times $K(\cdot, \zeta)$, $F(\cdot, \zeta_n)$ must have
a zero $z_n$ near the zero $z$ of $K(\cdot,\zeta)$. In fact, let
$\displaystyle \gamma=
\partial {\mathbb{D}}(z, \epsilon)$ with $\epsilon > 0$ small enough.
Then, as $n\to\infty$,
$$
\frac{1}{2\pi\I}\int_{\gamma} d \log F(\cdot, \zeta_n) \to
\frac{1}{2\pi\I}\int_{\gamma} d \log K(\cdot, \zeta) = 1,
$$
the last integral being de number of zeros of $K(\cdot, \zeta)$
inside $\gamma$. Hence $F(\cdot, \zeta_n) $ has exactly one zero
$z_n$ inside $\gamma$. Letting $\epsilon$ tend to $0$ as $n$ tends
to $\infty$, we conclude that there are zeros $z_n$ with $ z_n \to
z$. Thus (i) implies (iv), and the proof is complete.
\end{proof}


\begin{rem}
The following example illustrates the fact that critical points of
Green's function are not necessarily simple. Let $\omega=e^{2\pi\I
/3}$, $D_j= {\mathbb{D}}(\omega ^j, \rho)$, $j= 1, 2, 3$, and $0<\rho<
\frac{\sqrt 3}{2}$. The Green's function $G(z)$ of
$\Omega=\BC\setminus (\bar D_1 \cup \bar D_2 \cup \bar D_3)$ with
pole at infinity has two critical points (${\texttt{g}}=2$). Since
$\Omega$ is invariant under $z \to \omega z$, the two critical
points must be a double point located at the origin. By means of a
M\"obius transformation, one obtains an example of a $2$-holed disk
which has a Green's function with a multiple critical point.
\end{rem}


\section{The Martin and gradient boundaries}

\subsection{The Martin compactification}

In this section we shall give the preceding considerations their
right meaning. We adapt the construction of Martin compactification
as presented for instance in I.S.~Gal \cite{Gal}. Let us first
recall a general theorem of Constantinescu-Cornea (see
\cite{Constantinescu-Cornea}, p.~97).

\begin{thm}
Let $ \Omega $ be a non-compact locally compact Hausdorff space and
let $\Phi$ be a family of  continuous functions $\Omega\to
[-\infty,\, +\infty]$. Then there exists a compact topological space
$\Omega^M$, unique up to homeomorphisms, such that
\begin{enumerate}

\item[(i)] $\Omega $ is an open and dense subset of $\Omega^M$.

\item[(ii)] Every $f\in  \Phi$ can be extended to a continuous function
$ f^M$ on $\Omega^M$.

\item[(iii)] The functions $f^M$ separate points on
$\partial^M\Omega:=\Omega^M \setminus \Omega$.

\end{enumerate}
\end{thm}

For example, when  $\Omega$ is a multiply connected domain in the
plane one may, for a fixed $a \in \Omega$, consider the family
$\Phi$ of functions $\zeta \mapsto M(z,\zeta)= \displaystyle{
\frac{G(z,\zeta)}{G(a,\zeta)}}$,  parametrized by $z\in\Omega$ and
with the convention that $M(a,a)=1$. Each function $z \mapsto
M(z,\zeta)$ is continuous, even at $a$. The space $\Omega^M$
obtained by the theorem of Constantinescu-Cornea is, up to a
homeomorphism, independent of $a$. It is the Martin compactification
of $\Omega$, and the Martin boundary is $\partial^M\Omega=\Omega^M
\setminus \Omega$.

As discovered by M.~Brelot and I.~S.~Gal (see \cite{Gal}), the
Martin boundary can also be introduced via uniform structures. For
generalities on such, see \cite{Kelley}. Let $X$ be a nonempty set,
let $\left( Y, \mathcal V \right )$ be a uniform space and let
$\Phi$ be a family of functions $\phi: X \to Y$. Then there is a
weakest uniform structure on $X$ making the functions in $\Phi$
uniformly continuous. It is the uniformity ${\mathcal U}$ generated
by all the $\phi^{-1}(V)= \{(x,y)\in X\times X: \left( \phi(x),
\phi(y)\right) \in {V}\}$, for $V\in\mathcal{V}$, $\phi\in\Phi$, as
a subbase.

If $\mathcal V$ is precompact (totally bounded), so is \,$\mathcal
U$. Now the space $Y= [-\infty,\,+\infty ]$ is compact, hence has a
unique uniform structure, and this is precompact. With $\Omega$ a
multiply connected domain in the complex plane, and $\Phi$ the set
of functions $M(z,\cdot)$ ($z\in\Omega$) above, this gives a uniform
structure the completion of which is the Martin compactification.
The same remark also applies, with $Y= {\mathbb{P}}$, to the family
of functions of $\zeta\mapsto \displaystyle F(z,\zeta)= \partial_z G
(z,\zeta) /\partial_z G(a,\zeta)$ for $z \in \Omega$, $F(a, a)= 1$.
This gives a uniform structure $\mathcal U_G$, which we call the
{\em gradient structure}. Since ${\mathbb{P}}$ is compact we have

\begin{thm}\label{G1}
The gradient structure $\mathcal U _G$ is precompact.
\end{thm}

To make the link with the Constantinescu-Cornea theorem we recall
that as soon as a uniform structure is introduced on a set $X$,
there automatically arises a complete uniform space $(\bar X, \bar{
\mathcal U})$ and an injection map $f: X\to \bar X$ such that $f(X)$
is dense in $\bar X$ and ${\mathcal U}=  f^{-1} \left(\bar {\mathcal
U}\right )$. The triple $\left( f, \bar X, \bar{\mathcal U} \right)$
is the completion of $X$ with respect to $\mathcal U$. Moreover the
space $\bar X$ is compact if and only if the uniform structure
$\mathcal U$ is precompact. In our setting, with a multiply
connected domain $\Omega\subset{\mathbb{C}}$, we can therefore
formulate Theorem~\ref{G1} in a more precise way as follows.

\begin{thm}
The completion of the space $\Omega$ with respect to its gradient
structure $\mathcal U _G$ yields a compact space $\Omega^{\mathcal
G}=\Omega\cup \partial^{\mathcal G} \Omega$, the gradient
compactification of $\Omega$.
\end{thm}

In analogy with the Martin compactification, we call the set
$\partial^{\mathcal G} \Omega=\Omega^{\mathcal G}\setminus \Omega$
the {\em gradient boundary}. In Theorem~\ref{5} we showed that for
multiply connected domains with analytic boundary, the gradient
boundary is identical with the Euclidian one.


\subsection{An estimate of P. Levy}

Thinking of the Martin boundary as a local concept we may, in view
of Example~\ref{ex:disk} and beginning of
subsection~\ref{sec:criticalboundary}, ask for local approximation
of the Green's function of a general domain with sufficiently smooth
boundary by the Green's function of the upper half-plane. Sharp
estimates in this direction have been obtained by P.~Levy
\cite{Levy}. Below we review Levy's method.

Let $z \in \Omega$, ${d}(z)= {\rm dist\,} (z,\partial\Omega)$, $c$ a
point on $\partial \Omega$  on distance ${d}(z)$ from $z$ and let
$z'$ be the symmetric point of $z$ with respect  to the boundary
$\partial \Omega$, so that $\displaystyle c= \frac{1}{2}(z+ z') $.
We assume that $z' \notin \bar \Omega$, which is the case for
instance if $\partial \Omega$ satisfies an exterior ball condition
(which we henceforth assume) and ${d}(z)$ is sufficiently small. Let
$R$ the radius of the largest circle in $\Omega$ which is tangent to
$\partial \Omega$ at $c$ and $R'$ the radius of the largest circle
outside $\bar \Omega$ which is tangent to $\partial \Omega$ at $c$.
We denote by $a$, $a'$ the centers of the respective circles,
assuming that $R'<\infty$; the case $R'=\infty$ can be easily
covered by a limit argument. The points $a,z,c,z',a'$ lie along a
straight line.

\begin{thm} In the above notation,
$$
\log \left(1- \frac{2d}{2R' +d} \right)< \log |\frac{z'- \zeta}{z-
\zeta}|- G(z, \zeta) <\log \left(1+ \frac{2d}{2R  -d} \right)
$$
for every $\zeta \in \Omega$.
\end{thm}

\begin{proof}
The function $\zeta \mapsto \log |\frac{z'- \zeta}{z- \zeta}|- G(z,
\zeta) $ is harmonic in $\Omega$ and equals $\displaystyle \log
|\frac{z'- \zeta}{z- \zeta}|$ on $\partial \Omega$. The level lines
of the latter function are circles, namely Apollonius circles with
respect to the points $z$ and $z'$. Recall that the Apollonius
circles with centers $z$ and $z'$ are $\Gamma_k= \{\zeta \in \BC:
|z'-\zeta|= k|z-\zeta|\}$ for $k>0$.

We conclude from the above that we get upper and lower bounds for
$\displaystyle \log |\frac{z'- \zeta}{z- \zeta}|$ on $\partial
\Omega$ by considering one Apollonius circle which lies entirely
inside $\Omega$ and one which lies entirely outside $\Omega$.
Optimal choices with respect to the given data are those Apollonius
circles which are tangent to the largest interior and exterior
balls. The points of tangency are denoted by $b$ and $b'$. It is
easily seen that
$$
\log|\frac{z'- b'}{z- b'}|< \log |\frac{z'- \zeta}{z- \zeta}| < \log
|\frac{z'- b}{z- b}|
$$
for $\zeta\in\partial\Omega$. Since
$$
|z'- b|= 2R+d,\quad |z- b|= 2R- d,
$$
$$
|z'- b'|= 2R'- d,\quad |z- b'|= 2R'+d
$$
we get, for $\zeta\in \partial \Omega$,
$$
\log\left(1- \frac{2d}{2R'+d}\right)< \log |\frac{z'- \zeta}{z-
\zeta}|< \log\left(1+ \frac{2d}{2R- d}\right),
$$
and hence
$$
\log\left(1- \frac{2d}{2R'+d}\right)< \log |\frac{z'- \zeta}{z-
\zeta}|- G(z, \zeta)< \log\left(1+ \frac{2d}{2R- d}\right).
$$
Now the assertion of the theorem follows from the maximum principle.
\end{proof}


\section{The Poincar\'{e} metric, Taylor coefficients and level lines}\label{sec:Taylor}

\subsection{The Poincar\'{e} metric on level lines for the Green function}

Here we shall discuss some properties of level lines of Green's
function and in particular their connections with the Poincar\'e
metric. Later on we shall interpret these level lines as geodesics
for a different metric.

Consider a simply connected domain $\Omega$, let $d\sigma_z=
\rho(z) |d z|$ be the Poincar\'e metric and $G(z, \zeta)$  the
Green's function of $\Omega$.

\begin{prop}
For every $a \in \Omega$, the density $\rho (z)$ of the Poincar\'e
metric is
\begin{equation}\label{eq:rhosinhG}
\rho(z)= \frac{\vert \frac{\partial G}{\partial z }(z, a)\vert}{
\sinh  G(z,a)}.
\end{equation}
In particular, on each level line $G(z, a)=c$,
$$
\rho(z)=- \frac{1}{ 2 \sinh c} \frac{\partial G} {\partial n_z}(z,
a),
$$
i.e., $\rho$ is proportional to the harmonic measure with respect to
$a$, or equivalently to the Poisson kernel of the enclosed domain.
\end{prop}

\begin{proof}
The proof is classical and we recall it for the convenience of the
reader. With ${\mathcal G}(z,a)=G(z,a)+\I G^*(z,a)$ an analytic
completion of $G(z,a)$ with respect to $z$, the map $\displaystyle z
\mapsto \zeta = e^{-{\mathcal G}(z,a)} $ sends $\Omega$ onto the
unit disk, for which the Poincar\'e metric is
$$
d\sigma_{\zeta}= \frac{\vert d \zeta \vert}{1- \vert \zeta \vert^2}.
$$
By conformal invariance
$$
d\sigma_z = \rho(z) \vert dz\vert = \frac{\vert d(e^{-{\mathcal
G}(z,a)}) \vert}{1- \vert e^{-{\mathcal G}(z,a)}\vert^2}
$$
$$
= \frac{ e^{-G(z, a)}\, \vert 2 \frac{\partial G}{\partial z}(z,
a)\vert }{1-  e^{-2G(z,a)}} \vert dz \vert =\frac{\vert
\frac{\partial G}{\partial z }(z, a)\vert}{ \sinh \left(
G(z,a)\right)}\vert dz \vert,
$$
which is the desired result.
\end{proof}

\begin{rem}
If $\Omega$ is multiply connected the above expression for the
Poincar\'{e} metric still holds if $G(z,\zeta)$ is interpreted as
the Green's function for the universal covering surface of $\Omega$.
This is not single-valued in $\Omega$, but the combination appearing
in the expression for $\rho$ is single-valued.
\end{rem}

As an application, we reprove a result due to T.~Kubo \cite{Kubo}.
With $\Omega$ simply connected as above, let $a \in \Omega$, let
$\mu$ be a positive measure with compact support in $\Omega$ and
choose $c>0$ so that the support of $\mu$ is contained in
$\displaystyle K= \{ z\in \Omega: G(z, a)\geq c \}$. We define the
Green's potential $G^{\mu}$ of $\mu$ by
$$
G^{\mu} (z)= \int G(z, \zeta)\, d\mu(\zeta).
$$
The functions $G$, $G^{\mu}$ are both harmonic in $\Omega \setminus
K$ and vanish on $\partial \Omega$. Hence, with $d\sigma_z$ the
Poincar\'{e} metric as above,
$$
\int_{\partial K}G^{\mu} (z)d\sigma_z = -\frac{1}{ 2 \sinh c}
\int_{\partial K} G^{\mu} (z) \frac{\partial G}{\partial n}(z,a)ds
$$
$$
=\frac{1}{ 2 \sinh c} \int_{\partial (\Omega\setminus K)} G^{\mu}
(z) \frac{\partial G}{\partial n}(z, a)ds =  \frac{1}{ 2 \sinh c}
\int_{\partial (\Omega\setminus K)}G(z,a) \frac{\partial
G^{\mu}}{\partial n}ds
$$
$$
= - \frac{c}{ 2 \sinh c} \int_{\partial K} \frac{\partial
G^{\mu}}{\partial n}ds =\frac{\pi c}{\sinh c}\mu(K).
$$

Thus, under the above assumptions,
\begin{thm}
The average of the Green potential $G^{\mu}$, with respect to the
Poincar\'e metric, on a level line $G(\cdot, a) = c$ equals a
constant (independent of $\mu$) times the total mass of $\mu$.
\end{thm}

For the Poincar\'e metric in a simply connected domain we have the
estimates (see, for example, \cite{Milnor})
\begin{equation}\label{eq:drhod}
\frac{1}{4{d}(z)}\leq \rho(z)\leq \frac{1}{{d}(z)},
\end{equation}
where the lower bound is a consequence of the Koebe one-quarter
theorem. Compare also (\ref{eq:logdc}) below. By (\ref{eq:rhosinhG})
this gives the following estimate of the distance to the boundary
directly in terms of the Green function:
$$
\frac{\sinh G(z,a)}{4\vert \frac{\partial G}{\partial z}(z, a)\vert}
\leq d(z)\leq \frac{\sinh G(z,a)}{\vert\frac{\partial G}{\partial
z}(z, a)\vert}.
$$
See \cite{Milnor} for possible applications of such estimates to
computer graphics.

Also the Bergman kernel can provide estimates for the distance to
the boundary, even in the multiply connected case.
In fact, setting
$$
K^{(m,n)}(z, \zeta)= \frac{ \partial^{m+n}}{\partial z^m
\partial \bar \zeta ^n}K(z,  \zeta)
$$
we have, according to P.~Davis and H.~Pollak \cite{Davis}, the Cauchy-Hadamard type
formula
\begin{equation} \label{Davis1}
\frac{1}{{d} (z)}= \limsup_{n\to \infty}  \frac{e}{n}
\left(K^{(n,n)}(z,  z)\right)^{\frac{1}{2n}}
= \limsup_{n\to \infty}
\left(\frac{1}{(n!)^2}K^{(n,n)}(z,  z)\right)^{\frac{1}{2n}}
\end{equation}
for any $z\in \Omega$.

In the simply connected case, this gives an interesting formula for
the distance to the boundary. Let
$\phi(\zeta)=\exp(-{\mathcal{G}}(\zeta,z))= a_1 (\zeta-z)+
a_2(\zeta-z)^2+\cdots$, $a_1>0$, be the conformal map from $\Omega$
to the unit disk which takes a given point $z$ to the origin. Then
\begin{equation} \label{Davis2}
\frac{1}{d(z)}= \limsup_{n\to \infty} \left[ \sum_{k= 0}^n (k+1)
\left| \sum_{j = k}^n \frac{(n-j+1)}{j!} a_{n-j+1} \frac{d^j
}{d\zeta^j}\big|_{\zeta=z}
\phi(\zeta)^{k}\right|^2\right]^{\frac{1}{2n}}.
\end{equation}


\subsection{Taylor coefficients}\label{sec:taylor coefficients}

Above we saw how the distance to the boundary controls the
Poincar\'{e} metric. Below we shall see more generally how this
distance controls the Taylor coefficients of the Green's function,
which in the simply connected case embody the Poincar\'{e} metric.
For any multiply connected domain $\Omega$, let $H(z,\zeta)$ be the
regular part of the Green's function, defined by
\begin{equation}\label{eq:GH}
G(z,\zeta)=-\log|z-\zeta|+ H(z, \zeta),
\end{equation}
and let ${\mathcal G}(z,\zeta)$ and ${\mathcal H}(z,\zeta)=H(z,
\zeta)+\I H^*(z, \zeta)$ denote analytic completions of $G(z,\zeta)$
and $H(z,\zeta)$ with respect to $z$, for fixed $\zeta$. Then
${\mathcal G}(z,\zeta)$ is multivalued, but ${\mathcal H}(z,\zeta)$
is a perfectly well-defined analytic function (in $z$), uniquely
determined after the normalization $\Im{\mathcal H}(\zeta,\zeta)=0$,
henceforth assumed. Thus ${\mathcal H}(z,\zeta)$ can be expanded in
a power series around $z=\zeta$:
\begin{equation}\label{eq:Hexpansion}
{\mathcal H}(z,\zeta)=c_0(\zeta) + c_1(\zeta)(z-\zeta) +
c_2(\zeta)(z-\zeta)^2 +\dots,
\end{equation}
where $c_0$ is real by the normalization chosen.

The first few of the coefficients $c_n(\zeta)$ are domain function
which have geometric and physical relevance. The constant term,
$c_0(\zeta)$, is sometimes called the Robin constant and
$e^{-c_0(\zeta)}$ is a kind of capacity (if one allows
$\infty\in\Omega$ then $c_0(\infty)$ is the ordinary logarithmic
capacity of ${\mathbb{C}}\setminus \Omega$), cf. \cite{Sario-Oikawa}.

As follows from (\ref{eq:K}), $c_0(z)$ is related to the Bergman
kernel by
$$
-\Delta c_0(z) =4\pi K(z,z).
$$
Under conformal mappings $c_0(z)$ transforms in such a way that
$$
d\sigma=e^{-c_0(z)}|dz|
$$
is a conformally invariant metric (see further
Section~\ref{sec:connections}). When $\Omega$ is simply connected
this metric coincides with the Poincar\'{e} metric and with the
Bergman metric, normalized to be $d\sigma =\sqrt{\pi K(z,z)}|dz|$,
but for multiply connected domains none of these metrics are the
same. Note also that comparison with (\ref{eq:rhosinhG}) shows that
$$
{c_0(z)}= -\log\frac{\vert \frac{\partial G}{\partial z }(z,
a)\vert}{ \sinh G(z,a)}
$$
for all $z,a\in\Omega$ when $\Omega$ is simply connected. In the
multiply connected case we still have that the limit, as $z\to a$,
of the right member equals $c_0(a)$.

In a simply connected domain any conformally invariant metric has
constant Gaussian curvature, because the curvature transforms as a
scalar and the conformal group acts transitively on a simply
connected domain. For the above metrics the constant curvature is
negative: $\kappa_{\rm Gauss}=-4$. This means that $c_0$ satisfies
the Liouville equation $-\Delta c_0=4e^{-2c_0}$, and for the Bergman
and Schiffer kernels it means that
\begin{equation}\label{curvature7}
\frac{\partial^2 \log K(z,\zeta)}{\partial z \partial \bar\zeta} =
2\pi K(z, \zeta),\quad \frac{\partial^2 \log L(z,\zeta)}{\partial z
\partial \zeta} = -2\pi L(z, \zeta).
\end{equation}
We emphasize that these relations only hold for simply connected
domains. For multiply connected domains there are counterparts of
for example (\ref{curvature7}) involving also the zeros of the
Bergman kernel, see \cite{Schiffer-Hawley}, \cite{Hejhal},
\cite{Hejhal2} and Remark~\ref{rem:span} in the present paper.

The function $c_0(z)$ can be estimated in terms of the distance
$d(z)=d(z,\partial \Omega)$ to the boundary by
\begin{equation}\label{eq:logdc}
\log d(z)\leq c_0(z)\leq \log d(z)+A
\end{equation}
for some constant $A$. The lower bound is an elementary consequence
of the maximum principle combined with monotonicity properties of
$c_0(z)$ with respect to the domain. The upper bound depends on the
nature of the domain. If for example the domain is convex, one can
take $A=\log 2$, and for a general simply connected domain
(\ref{eq:logdc}) is the same as (\ref{eq:drhod}), i.e., $A=\log 4$
works. See \cite{Sario-Oikawa}, \cite{Ahlfors} for further discussions.

For the higher coefficients we have the following estimates, one of
which will be used in Section~\ref{sec:connections on double}.

\begin{lem}\label{lem:estcn}
For $n\geq 1$,
$$
|c_n(\zeta)|\leq\frac{1}{n d(\zeta,\partial\Omega)^n} \quad
(z\in\Omega).
$$
\end{lem}

\begin{proof}
The derivative of the Green's function has the Taylor expansion,
with respect to $z$,
$$
2\frac{\partial G(z,\zeta)}{\partial z} =-\frac{1}{z-\zeta}
+\sum_{n=1}^\infty nc_n(\zeta)(z-\zeta)^{n-1}.
$$
Thus, for any region $D\subset\Omega$ containing $\zeta$,
$$
nc_n(\zeta)= \frac{1}{2\pi\I} \int_{\partial D} 2\frac{\partial
G(z,\zeta)}{\partial z} \frac{dz}{(z-\zeta)^n} = \frac{1}{2\pi\I}
\int_{\partial D}  \frac{dG(z,\zeta)+\I ^*dG(z,\zeta)}{(z-\zeta)^n}
$$
On taking $D=\{z\in\Omega: G(z,\zeta)>\epsilon\}$, where
$\epsilon>0$, we have $dG(\cdot,\zeta)=0$ along $\partial D$, while
$-^*dG(\cdot,\zeta)$ can be considered as a positive measure of
total mass $2\pi$ on $\partial D$. By letting $\epsilon\to 0$ this
gives the desired estimates:
$$
|nc_n(\zeta)|= |\frac{1}{2\pi} \int_{\partial\Omega} \frac{
^*dG(z,\zeta)}{(z-\zeta)^n}|\leq
\frac{1}{d(\zeta,\partial\Omega)^n}.
$$
\end{proof}


\subsection{The Green's function by domain variations}

The coefficient $c_1$ in (\ref{eq:Hexpansion}) can be directly
obtained via the formula
$$
c_1(\zeta)= \frac{\I}{2}\int_{\partial\Omega} (\frac{\partial
G(z,\zeta)}{\partial z})^2 dz =\frac{1}{2\I}\int_{\partial\Omega}
|\frac{\partial G(z,\zeta)}{\partial n}|^2 d\bar{z},
$$
which follows from the residue theorem. Instead of the residue
theorem one may use the fact that $2\bar{c}_1$ is the gradient of
$c_0$, as is easily checked (see (\ref{eq:c1fromc0})), combined with
the Hadamard variational formula,
$$
\delta G_\Omega(z,\zeta) =\frac{1}{2\pi}\int_{\partial\Omega}
\frac{\partial G(\cdot,z)}{\partial n}\frac{\partial
G(\cdot,\zeta)}{\partial n} \delta n \,ds,
$$
which gives the infinitesimal change of the Green function under an
infinitesimal deformation $\delta n$ of $\partial\Omega$ in the
outward normal direction. Since $\delta H_\Omega(z,\zeta)=\delta
G_\Omega(z,\zeta)$, one obtains $\delta c_0(\zeta)$ by choosing
$z=\zeta$ above, and then the gradient of $c_0$ is obtained by
choosing $\delta n$ suitably. See \cite{Flucher}, Lemma~8.4, for
further details.

Here we wish to expand slightly on another use of the Hadamard
formula. To avoid infinitesimals one may replace $\delta$ by
$\frac{\delta}{\delta t}$, where $t$ is a time parameter, so that
$\frac{\delta n}{\delta t}$ means the velocity of $\partial\Omega$
in the normal direction. Of special interest is to take this normal
velocity proportional to the normal derivative of the Green's
function itself, say
$$
\frac{\delta n}{\delta t}=-\frac{\partial G(\cdot,a)}{\partial n},
$$
where $a\in \Omega$ is a fixed point. This is called Laplacian
growth, or Hele-Shaw flow with a point source (see, e.g.,
\cite{Gustafsson-Vasilev} for further information), and if
$\nabla(a)=\frac{\delta}{\delta t}$ denotes the corresponding
derivative, acting on domain functionals, the Hadamard formula gives
$$
\nabla(a)G_\Omega(b,c) =-\frac{1}{2\pi}\int_{\partial\Omega}
\frac{\partial G(\cdot,a)}{\partial n} \frac{\partial
G(\cdot,b)}{\partial n}\frac{\partial G(\cdot,c)}{\partial n} ds,
$$
hence that $\nabla(a)G_\Omega(b,c)$ is totally symmetric in $a$,
$b$, $c$.

This remarkable fact appears in a series of articles by M.~Mineev,
P.~Wiegmann, A.~Zabrodin, I.~Krichever, A.~Marshakov, L.~Takhtajan,
for example \cite{Mineev-Wiegmann},
\cite{Krichever-Marshakov-Zabrodin}, \cite{Takhtajan}, from a more
general perspective to be a consequence of two things. The first is
that the regular part $H(z,\zeta)$ of the Green's function is a double
variational derivative of an energy functional $F=F(\Omega)$, which
can be identified with the logarithm of a certain $\tau$-function
\cite{Kostov et al}. Precisely, $H_\Omega(a,b)=\nabla(a)\nabla(b)
F(\Omega)$, hence
$$
G_\Omega(a,b)=-\log|a-b|+\nabla(a)\nabla(b) F(\Omega),
$$
where
$$
F(\Omega)=\frac{1}{8\pi^2}\int_{{{\mathbb{D}}}(0,R)\setminus
\Omega}\int_{{{\mathbb{D}}}(0,R)\setminus \Omega} \log|z-\zeta| dA(z)
dA(\zeta),
$$
$dA$ denoting area measure and where $R$ is large enough, so that
$\Omega\subset{{\mathbb{D}}}(0,R)$. The second ingredient is that the
Dirchlet problem is ``integrable'' in the sense that
$$
\nabla(a)\nabla(b)=\nabla(b)\nabla(a).
$$
Clearly these two facts embody the total symmetry of $\nabla(a)G_\Omega(b,c)$.


\subsection{Level lines of harmonic functions as geodesics and
trajectories}\label{sec:levellines}

Here we shall interpret level lines of harmonic functions as
geodesics in riemannian manifolds and trajectories of hamiltonian
systems. Curvature of level lines and geodesics will be discussed in
Section~\ref{sec:connections on double}.

\begin{prop}\label{prop:geodesic}
Let $u$ be harmonic in some domain, $u^*$ a harmonic conjugate of
$u$ and let $\varphi >0$ be any smooth function in one real
variable. Then, away from critical points of $u$, the level lines of
$u$ are geodesics for the metric
$$
d\sigma =\varphi (u^*)|\nabla u||dz|.
$$
\end{prop}

\begin{proof} Let $\Phi$ be a primitive function of $\varphi$ and let
$\gamma$ be a level line of $u$, with $\nabla u\ne 0$ on $\gamma$.
By the  Cauchy-Riemann equations, this level line is simultaneously
an integral curve of $\nabla u^*$, and $|\nabla u|=|\nabla u^*|$.
Thus, if $z_0$ and $z_1$ denote the end points of $\gamma$ ordered
so that $u^*(z_0)<u^*(z_1)$,
$$
\int_\gamma d\sigma =\int_\gamma \Phi'(u^*)|\nabla u^*||dz|
=\int_\gamma \Phi'(u^*)du^*=\Phi(u^*(z_1))-\Phi(u^*(z_0)).
$$
Since, along a curve in general, $|\nabla u^*||dz|\geq du^*$,
integration of $d\sigma$ along any curve from $z_0$ to $z_1$ will
give a value $\geq \Phi(u^*(z_1))-\Phi(u^*(z_0))$. Thus $\gamma$ is
a geodesic.
\end{proof}

As a simple remark on trajectories, the level lines of any
(smooth) function $u$ in a domain $\Omega\subset{\mathbb{C}}$ are
trajectories of the hamiltonian system with phase space $\Omega$,
symplectic form $\omega=dx\wedge dy$ and hamiltonian function $u$
(see \cite{Arnold} for the terminology). The Hamilton equations then
are $\frac{d x}{d t}= \frac{\partial u}{\partial y}$, $\frac{d y}{d
t}=-\frac{\partial u}{\partial x}$ or, in complex form,
\begin{equation}\label{eq:hamiltonian}
\frac{dz}{dt}=-2\I\frac{\partial u}{\partial \bar z}.
\end{equation}
Such a kind of hamiltonian system describes for example the motion
of a point vortex in a plane domain, see \cite{Lin}, \cite{Newton}
and Section~\ref{sec:connections on double} below.

The above equations guarantee that the motion is along the level
lines of $u$. Assume now that $u$ is harmonic. Then
(\ref{eq:hamiltonian}) gives also
\begin{equation}\label{eq:Newton}
\frac{d^2 z}{dt^2}=-2\I\frac{\partial^2 u}{\partial \bar
z^2}\frac{d\bar z}{dt} =4\frac{\partial^2 u}{\partial \bar
z^2}\frac{\partial u}{\partial z} =-2\frac{\partial V}{\partial
\bar{z}},
\end{equation}
where
$$
V =-\frac{1}{2}|\nabla u|^2=-2\frac{\partial u}{\partial
z}\frac{\partial u}{\partial \bar z}.
$$
Notice that $V$ is real-valued and that the right-hand side of
(\ref{eq:Newton}) is minus the gradient of $V$. Therefore we can
think of (\ref{eq:Newton}) as an ordinary newtonian system for the
motion of a unit point mass in the potential $V$. Thus

\begin{prop}\label{prop:Newton}
The level lines of any harmonic function $u$ are, away from critical
points, trajectories for the newtonian system with potential energy
$-\frac{1}{2}|\nabla\,u|^2$.
\end{prop}

\begin{rem}
The proposition generalizes to the case that $\Delta u = c$,  $c$
constant, with $V$ changed to
$$
V =-\frac{1}{2}|\nabla u|^2 +cu.
$$
\end{rem}

If we want to put Proposition~\ref{prop:Newton} into a hamiltonian
formulation, the domain $\Omega$ takes the role of configuration
space, while phase space is $\Omega\times {\mathbb{C}}$. The kinetic
energy is $T=\frac{1}{2}|\frac{dz}{dt}|^2$ and the hamiltonian
$$
H=T+V,
$$
to be considered as a function of position $q=z$ and momentum
$p=\frac{dz}{dt}$. The Hamilton equations, written in complex
notation, are
\begin{equation}\label{eq:complexHamilton}
\frac{d q}{d t}= 2 \frac{\partial H}{\partial \bar{p}}, \quad
\frac{d p}{d t}=-2\frac{\partial H}{\partial \bar{q}}.
\end{equation}

Along a trajectory $t\mapsto (p(t),q(t))$ in phase space, $H(p,q)$
is constant, say $H=E$, where $E$ is the total energy. Now, the
principle of least action in the form of Maurpertuis, Euler,
Lagrange and Jacobi (see \cite{Arnold}, \cite{Frankel}), states that
the trace in configuration space of a hamiltonian trajectory of
constant energy $E$ is a geodesic for the Jacobi metric,
$$
d\rho = \sqrt{T} ds =\sqrt{E-V(z)}ds.
$$
Here $ds$ denotes the ordinary euclidean metric in $\Omega$, hence
arc-length along trajectories; in a more general context it would be
the metric in configuration space induced by the kinetic energy.

In our case $V =-\frac{1}{2}|\nabla u|^2$, and since we derived the
motion from (\ref{eq:hamiltonian}), $T=\frac{1}{2}|\nabla u|^2$.
Thus $E=0$ and the Jacobi metric becomes
$$
d\rho =\frac{1}{\sqrt{2}}|\nabla u|\,|dz|,
$$
i.e., the least action principle becomes an instance of
Proposition~\ref{prop:geodesic}.

The hamiltonian system (\ref{eq:complexHamilton}) or
(\ref{eq:Newton}) has of course many more trajectories (in
configuration space) than the level lines of $u$. In fact, through
any point $z$ there is one trajectory starting out with any
prescribed speed $\frac{dz}{dt}$. (Even if we ignore the
parametrization, the modulus $|\frac{dz}{dt}|$ of the speed affects
the trajectory as a point set). Some of these other trajectories
(one in each direction) can covered by the above analysis by mixing
$u$ with its harmonic conjugate, as follows.

Let $u^*$ be a harmonic conjugate of $u$ and let
$$
w=f(z)=u+\I u^*,
$$
so that $f$ is conformal away from critical points of $u$. Then, for
any $\theta\in{\mathbb{R}}$, $|\nabla u|=|2\frac{\partial u}{\partial
z}|=|f'(z)|= |e^{\I\theta}f'(z)|=|\nabla(\cos \theta\,u
-\sin\theta\,u^* )|$, so that $V$ remains unchanged if $u$ is
replaced by $\cos \theta\,u -\sin\theta\,u^*$. Thus the level lines
of all such functions are also trajectories, and it is easily seen
that they represent all trajectories with total energy $E=0$. In
terms of $f(z)$ we can express the conclusion as follows.

\begin{prop}
The trace in configuration space $\Omega$ of the trajectories on the
energy surface $E=0$ of the hamiltonian system
(\ref{eq:complexHamilton}) or (\ref{eq:Newton}) are exactly the
inverse images under the conformal map $w=f(z)=u+\I u^*$ of the
straight lines in the $w$-plane.
\end{prop}

As an application of the above we may take $u$ to be the Green's
function of a domain $\Omega\subset{\mathbb{C}}$: $u(z)=G(z,a)$, $a\in
\Omega$ fixed. With
$\Omega'=\Omega\setminus\{z_1(a),\dots,z_{{\texttt{g}}}(a)\}$, where
$z_1(a),\dots,z_{{\texttt{g}}}(a)$ are the critical points of
$G(z,a)$, we conclude from Proposition~\ref{prop:geodesic} that that
level sets of $G(z,a)$ are geodesics for
$$
d\sigma_1= |\nabla G(z, a)| |d z|,
$$
and
$$
d\sigma_2= \frac{|\nabla G(z, a)|}{G^*(z,a)}  |d z|.
$$
And, by changing the roles of $G$ and $G^*$ we see that the level
lines of the harmonic conjugate $G^*(z,a)$ are geodesics for
$$
d\sigma_3= \frac{|\nabla G(z, a)|}{G(z,a)}  |d z|
$$
(and for $d\sigma_1$). In the image region under
$f(z)=G(z,a)+\I G^*(z,a)$, the above geodesics for
$d\sigma_1$ correspond to geodesics of the euclidean metric and the
geodesics for $d\sigma_3$ correspond to geodesics for the
Poincar\'{e} metric in the right half-plane.


\section{Doubly connected domains}

\subsection{The general domain functions for an annulus}

This example is fundamental, it reveals in many respect the
essential ideas. Our exposition is based on ideas which  have been
elaborated in \cite{SF}, and previously in \cite {Maria}. Since we
shall only discuss conformally invariant questions it will be enough
work with annuli. We shall use the notation
$$
A _{a, b}= \{ z\in \BC: a< \vert z\vert < b \},
$$
$0<a<b$, for an annulus centered at the origin. The conformal type
is determined by the quotient $b/a$ (the modulus), so we can fix
either $a$ or $b$. Note also the conformal symmetries of $A_{a,b}$:
$z\mapsto e^{\I\theta}z$ for any $\theta\in{\mathbb{R}}$ and $z\mapsto
ab/z$, the latter being the conformal reflection about the symmetry
line $|z|=\sqrt{ab}$, which exchanges the outer and inner boundary
components.

Choosing the annulus to be $A_{1,R}$, where $R>1$, we can represent
the main domain functions in terms of elliptic functions. The
multivalued function $z\mapsto t=\log z$ lifts the annulus to the
strip $\{t\in{\mathbb{C}}: 0<\Re t< \log R\}$, which then represents
the universal covering surface of $A_{1,R}$. The lift map extends to
the Schottky double of the annulus, which is a torus, and takes it
onto the universal covering surface of the torus, namely
${\mathbb{C}}$. The covering transformations on ${\mathbb{C}}$ are
generated by $t\mapsto t+2\log R$ and $t\mapsto t+ 2\pi\I$, hence we
have a period lattice with half-periods
\begin{equation}\label{eq:periods}
\begin{cases}
\omega_1=\log R,\\
\omega_2=\I\pi.
\end{cases}
\end{equation}

One fundamental domain, symmetric around the origin, is
$$
F_0=\{t\in{{\mathbb{C}}}: -\log R<\Re t< \log R,\,-\pi<\Im t<\pi\}.
$$
Here the right half of this corresponds to the annulus $A_{1,R}$
itself, while the left half corresponds to the copy of the annulus
which makes up the back-side of the Schottky double. The
anti-holomorphic involution is the reflection in the imaginary axis:
$J(t)=-\bar{t}$. A homology basis, as chosen in
Section~\ref{sec:general}, can be taken to be $\alpha_1=[-\log R,
\log R]$, $\beta_1=[-\I\pi,\I\pi]$, as point sets. The orientations
of $\alpha_1$ and $\beta_1$ will however be opposite to the ordinary
orientations of the real and imaginary axes.

Sometimes it is advantageous to work in a fundamental domain whose
boundary is made up of preimages of the curves in the homology basis
chosen. Such a fundamental domain is
\begin{equation}\label{eq:F1}
F_1=\{t\in{{\mathbb{C}}}: 0<\Re t< 2\log R,\,0<\Im t<2\pi\}.
\end{equation}
It is often convenient to scale the period lattice so that
it is generated by $1$ and $\tau$, with $\Im\tau>0$, in place of
$2\omega_1$, $2\omega_2$. Then $\tau= \omega_2/\omega_2$, and in our
case we have
\begin{equation}\label{eq:tau}
\tau=\frac{\I\pi}{\log R}.
\end{equation}

We recall the standard elliptic functions associated to the given
period lattice. The Weierstrass $\wp$-function is
$$
\displaystyle{\wp(z)=\wp(z; 2\omega_1,\,2\omega_2) = \frac{1}{z^2} +
\sum_{\omega\ne 0} \left( \frac{1}{(z+ \omega)^2 }-
\frac{1}{\omega^2} \right)},
$$
and the $\zeta$- and $\sigma$-functions are
$$
\displaystyle{\zeta(z) = \frac{1}{z} + \sum_{\omega\ne 0} \left(
\frac{1}{z+ \omega }+ \frac{1}{\omega} + \frac{z}{\omega^2}\right)},
$$
$$
\displaystyle{\sigma(z)=z\prod_{\omega\ne 0} \left(1-
\frac{z}{\omega }\right)
\exp\left(\frac{z}{\omega}+\frac{z^2}{2\omega^2} \right)},
$$
where in all cases $\omega=2m_1\omega_1+ 2m_2\omega_2$ and
summations and products are taken over all $(m_1,m_2)\in \BZ \times
\BZ \setminus (0,0)$.

The above functions are related by
$$
\wp (z)=-\zeta'(z), \quad \zeta(z)=\frac{\sigma'(z)}{\sigma(z)},
$$
$\wp(z)$ is doubly periodic, $\zeta(z)$ acquires constants along the
periods,
$$
\zeta(z+2\omega_j)=\zeta(z)+2\eta_j,
$$
and for $\sigma(z)$  we have
$$
\sigma(z+2\omega_j)=-\sigma(z)e^{2\eta_j (z+\omega_j)}
$$
($j=1,2$).
The constants $\eta_j$ are given by $\eta_j=\zeta(\omega_j)$ and
they satisfy the Legendre relation
$$
\eta_1\omega_2-\eta_2\omega_1=\frac{\I\pi}{2}.
$$
In our case $\eta_1$ is real and positive, $\eta_2$ is purely
imaginary and the Legendre relation becomes
$$
\frac{\eta_1}{\log R}-\frac{\eta_2}{\I\pi}=\frac{1}{2\log R}.
$$
We record also the differential equation satisfied by $\wp(z)$:
$$
\wp'(z)^2 = 4 \wp^3 -g_2 \wp (z) -g_3,
$$
where
\begin{equation}\label{eq:g2g3}
g_2= 60\sum_{\omega\ne 0}\frac{1}{\omega^4},\quad g_3=
140\sum_{\omega\ne 0}\frac{1}{\omega^6}.
\end{equation}

As a first step we shall make some of the functions and
differentials appearing in Section~\ref{sec:general} explicit in the
annulus case, and then (in a later subsection) we shall study the
critical points of the Green's function in some detail. When we work
directly in the annulus, and with the Schottky double of it realized
in the same plane by reflection, we shall denote points by letters,
$z$, $w$, $a$, $b$ and similar. When we work on the universal
covering surface ${\mathbb{C}}$, with its period lattice, we shall
denote the corresponding points $t$, $s$, $u$, $v$. Thus $t=\log z$,
$s=\log w$, $u=\log a$, $v= \log b$.

The two versions of the abelian differentials of the third kind
discussed in Section~\ref{sec:general} are, on the universal
covering surface, given by
\begin{equation}\label{eq:upsilonzeta}
\upsilon_{u-v}(t)=(\zeta(t-u)-\zeta(t-v)+A)dt,
\end{equation}
\begin{equation}\label{eq:omegazeta}
\omega_{u-v}(t)=(\zeta(t-u)-\zeta(t-v)+B)dt,
\end{equation}
where $A$ and $B$ are constants (depending on $u,v$) chosen so that
$\upsilon_{u-v}$, $\omega_{u-v}$ get the desired periods. This means
that
$$
\int_{-\log R}^{\log R}(\zeta(t-u)-\zeta(t-v)+A)dt, \quad
\int_{-\pi\I}^{\pi\I}(\zeta(t-u)-\zeta(t-v)+A)dt
$$
are both purely imaginary, and
$$
\int_{-\log R}^{\log R}(\zeta(t-u)-\zeta(t-v)+B)dt=0.
$$
By straightforward calculations this
gives
\begin{equation}\label{eq:A}
A=\frac{\eta_1(u-v)}{\log R}+\frac{\Im(u-v)}{2\I\log R},
\end{equation}
\begin{equation}\label{eq:B}
B=\frac{\eta_1(u-v)+\pi\I m}{\log R},
\end{equation}
where the integer $m$ depends on the location of $u$, $v$ relative
to the preimage of $\alpha_1$ in the period lattice. With $u,v\in
F_0$, then $m=0$ if $u$ and $v$ are in the same component of
$F_0\setminus [-\log R,\log R]$, which is the case, for example, if
$\Im(u-v)=0$. Thus we see, in accordance with
Lemma~\ref{lem:omegaupsilon}, that $A=B$ when $v=J(u)$. One can also
achieve $m=0$ in more general situations by working in the
fundamental domain $F_1$, where $\alpha_1$ is represented by $[0,
2\log R]$, hence is part of the boundary.

Next we compute $V(t,s;u,v)$ by integrating $\upsilon_{u-v}$, see
(\ref{eq:Vupsilon}). The calculation is straightforward and the
result is
$$
V(t,s;u,v)=-\log\big|\frac{\sigma(t-u)\sigma(s-v)}{\sigma(t-v)\sigma(s-u)}\big|
$$
$$
-\frac{\eta_1\Re((t-s)(u-v))}{\log R} -\frac{\Im(t-s)\Im(u-v)}{2\log
R}.
$$
Pulling this back to the plane of the annulus, i.e., on substituting
$t= \log z$ etc., we get
$$
V(z,w;a,b)=-\log\big|\frac{\sigma(\log\frac{z}{a})\sigma(\log\frac{w}{b})}
{\sigma(\log\frac{z}{b})\sigma(\log\frac{w}{a})}\big|
-\frac{\eta_1\Re(\log\frac{z}{w}\log\frac{a}{b})}{\log R}
-\frac{\arg\frac{z}{w}\arg\frac{a}{b}}{2\log R}.
$$
We record also the final expressions for $\upsilon_{a-b}$ and
$\omega_{a-b}$ in the plane of the annulus:
$$
\upsilon_{a-b}=(\zeta(\log\frac{z}{a})-\zeta(\log\frac{z}{b})
+\frac{\eta_1\log\frac{a}{b}}{\log R}+\frac{\arg\frac{a}{b}}{2\I\log
R})\frac{dz}{z},
$$
$$
\omega_{a-b}=(\zeta(\log\frac{z}{a})-\zeta(\log\frac{z}{b})
+\frac{\eta_1\log\frac{a}{b}+\pi\I m}{\log R})\frac{dz}{z}.
$$

The Green's function is by (\ref{eq:GVJ}) just a special case of
$V$. Choosing $s=J(t)=-\bar t$, $v=J(u)=-\bar u$ and using the first
alternative in (\ref{eq:GVJ}) gives, on the universal covering
surface,
$$
G(t,u)=-\frac{1}{2}\log\big|\frac{\sigma(t-u)\sigma(-\bar{t}+\bar{u})}
{\sigma(t+\bar{u})\sigma(-\bar{t}-u}\big| -\frac{\eta_1\Re t\,\Re
u}{\log R}.
$$
Hence, in the annulus ($z,a\in A_{1,R}$),
$$
G(z,a)=-\frac{1}{2}\log\big|\frac{\sigma(\log\frac{z}{a})
\sigma(\log\frac{\bar{a}}{\bar{z}})}
{\sigma(\log(z\bar{a}))\sigma(\log\frac{1}{\bar{z}a})}\big|
-\frac{2\eta_1 \log |z|\log |a|}{\log R}.
$$
Expanding the $\sigma$-function in an infinite product gives us what we
would have obtained by the method of images (automorphization). The result is
\begin{equation}\label{eq:green-Crowdy}
G(z,a)=-\log|\frac{R(z-a)}{R^2-z\bar{a}}|-\log|\prod_{n=1}^\infty
\frac{(R^{2n}-\frac{z}{a})(R^{2n}-\frac{a}{z})}
{(R^{2n}-\frac{z\bar{a}}{R^{2n}})(R^{2n}-\frac{R^{2n}}{z\bar{a}})}|,
\end{equation}
see \cite{Courant}, \cite{Crowdy2}.

The Bergman kernel is in general
$$
K(z,\zeta)dz d\bar{\zeta} =\frac{1}{\pi}\frac{\partial
\omega_{\zeta-J({\zeta})}(z)}{\partial\bar{\zeta}}d\bar{\zeta}.
$$
Using (\ref{eq:omegazeta}), (\ref{eq:B}), this gives in the present
case
\begin{equation}\label{eq:K(t,u)}
K(t,u)=\frac{1}{\pi}(\wp(t+\bar{u})+ \frac{\eta_1}{\log R})
\end{equation}
for $t,u\in{\mathbb{C}}$, and
\begin{equation}\label{eq:K(z,a)}
K(z,a)=\frac{1}{\pi z\bar{a}}(\wp(\log (z\bar a))+
\frac{\eta_1}{\log R})
\end{equation}
in the annulus ($z,a\in A_{1,R}$). This agrees with expressions
derived by Zarankiewicz \cite {Zar}. See also \cite{Fay} (p.133) and \cite{Bergman}.

Finally, we elaborate the Schottky-Klein prime function in terms of
elliptic functions (see \cite{Crowdy2} for representations in terms
of Poincar\'{e} series). It is obtained by combining
(\ref{eq:expEEEE}) with (\ref{eq:omegazeta}) and (\ref{eq:B}). The
result of that is
$$
\exp\int_s^t
\omega_{u-v}=\frac{\sigma(t-u)\sigma(s-v)}{\sigma(t-v)\sigma(s-u)}
\exp\left[\frac{\eta_1(t-s)(u-v)}{\log R}+\frac{2\pi\I m(t-s)}{\log
R}\right],
$$
which is to be identified with
$\displaystyle\frac{E(t,a)E(s,v)}{E(t,v)E(s,u)}$. Working in the
fundamental domain $F_1$ in (\ref{eq:F1}), which allows $m=0$, this
gives
$$
E(t,u)=\sigma(t-u)\exp\left[\frac{\eta_1 tu}{\log R}\right]
f(t)f(u),
$$
for some one variable function $f$. To identify $f(t)$ we
differentiate with respect to $t$ at $t=u$:
$$
\frac{\partial E}{\partial t}|_{t=u} E(t,u)
=\sigma'(0)\exp\left[\frac{\eta_1 u^ 2}{\log R}\right]f(u)^2.
$$
Since $E(t,u)$ is to behave like $t-u$ at $t=u$ the above
derivative must be $=1$. This gives
$\displaystyle f(u)=\pm\exp\left[\frac{\eta_1 u^2}{2\log R}\right]$.
Therefore actually
$$
E(t,u)=\sigma(t-u)\exp\left[ -\frac{\eta_1 (t-u)^ 2}{2\log R}\right],
$$
and in the plane of the annulus,
$$
E(z,a)=\sigma(\log\frac{z}{a})\exp\left[ -\frac{\eta_1}{2\log R}(\log\frac{z}{a})^ 2\right].
$$


\subsection{Eisenstein series}

For further need we recall the following classical
Eisenstein series and other arithmetical functions \cite{Koblitz}.

\begin{equation}\label{E2}
E_2 (\tau)= 1-24 \sum_{n=1}^{\infty}\frac{n q^{n}}{1-q^{n}} = 1-24
\sum_{n=1}^{\infty} \sigma_1 (n)q^{n},
\end{equation}
\begin{equation}
E_4 (\tau)= 1+240 \sum_{n=1}^{\infty}\frac{n^3 q^{n}}{1-q^{n}} =
1+240 \sum_{n=1}^{\infty} \sigma_3 (n)q^{n},
\end{equation}
\begin{equation}
E_6 (\tau)= 1-504 \sum_{n=1}^{\infty}\frac{n^5 q^{n}}{1-q^{n}}   =
1-504 \sum_{n=1}^{\infty} \sigma_5 (n)q^{n},
\end{equation}
where
\begin{equation*}
\sigma_k (n)= \sum_{d|n}d^k,
\end{equation*}
and the ``nome'' $q$ always is related to $\tau$ by
\begin{equation*}
q=e^{2\pi\I \tau}.
\end{equation*}
Thus $|q|<1$.

It is well known that $E_4 (\tau)$ and $E_6 (\tau)$ are modular
forms of weights 4 and 6 for the modular group ${\G}= {\SL}$. This
is a consequence of Lipschitz' formula, which asserts that for any
integer $k\geq 2$, and with $B_{2k}$ being the $k$:th Bernoulli
number,
$$
E_{2k}(\tau)= 1- \frac{2k}{B_{2k}} \sum_{n=1}^{\infty} \frac{n^{k-1}
q^n}{1- q^n}
$$
$$
= 1- \frac{2k}{B_{2k}} \sum_{m, n=1}^{\infty}
\sigma_{k-1} (n) q^n =   1- \frac{2k}{B_{2k}} \sum_{m,n=1}^{\infty}
\frac{1}{ (m\tau+ n)^{2k}}.
$$
$E_2 (\tau)$ is not modular but
\begin{equation}\label{E_2^*}
E_2^* (\tau)= 1-24 \sum_{n=1}^{\infty}\frac{nq^{n}}{1-q^{n}}
-\frac{3}{\pi \Im\tau} q^3 -168 q^4- \cdots
\end{equation}
$$
= 1-\frac{3}{\pi \Im\tau}-24 q- 72 q^2- 96
$$
is a non-holomorphic Eisenstein series of weight 2 for ${\G}$.
Actually the series $E_2(\tau)$ satisfies
\begin{equation}\label{E_2}
E_2(\frac{a\tau+ b}{c\tau+ d})= E_2(\tau) (c\tau+ d)^2+ \frac{
6c(c\tau+d)}{\I \pi}
\end{equation}
for every $\left(\begin{array}{cc} a&b\\
c&d\end{array}\right)$ in $\G$. This is due to the fact that the
series
$$
\sum_{m,n=1}^{\infty} \frac{1}{ (m\tau+ n)^k}.
$$
is not absolutely convergent for $k= 2$. To overcome that difficulty
one considers, following Hecke, the series, defined for $s>0$,
$$
E^{*}_{2k}(\tau;s)= 1- \frac{2k}{B_{2k}} \sum_{m, n=1}^{\infty}
\frac{1}{ (m\tau+ n)^k |(m\tau+ n|^s}.
$$
and takes the limit as $s\to 0$. For $k=1$ this gives \eqref{E_2^*}.
We will use a similar idea in subsection~\ref{sec:cylinder} to
recover the Green's function via an eigenvalue problem.

We recall the classical discriminant function $\Delta$, defined as
the infinite product
$$
\Delta(\tau)= q\,\prod_{n=1}^{\infty} (1-q^n)^{24}=q- 24q^2+ 252q^3-
1472 q^4+\cdots.
$$
It is a cusp form for ${\G}$ of weight 12 and is related to the
Eisenstein series by
$$
1728 \Delta(\tau)= E_4(\tau)^3- E_6(\tau)^2.
$$
Following
Ramanujan \cite{Ram} we also introduce the functions
\begin{eqnarray}\label{Ram}
\Phi_{rs}(q)=\sum_{m=1}^{\infty}\sum_{n=1}^{\infty}m^r n^s q^{mn}.
\end{eqnarray}
Then,
\begin{eqnarray*}\label{Ram1}
\Phi_{rs}(q)&=&\Phi_{sr}(q),\\
\Phi_{rs}(q)&=&\left(q\frac{d}{dq}\right)^r \Phi_{0,s-r}(q).
\end{eqnarray*}

We may consider $\Phi_{rs}(q)$ as a function of $\tau$ and write
$\Phi_{rs}(\tau)$. For this we have an algebraic relation
\begin{eqnarray*}
\Phi_{rs}(\tau)= \sum K_{lmn} E_2(\tau)^l E_4(\tau)^m E_6(\tau)^n,
\end{eqnarray*}
for suitable coefficients $K_{lmn}$, and where the sum is taken over
all positive integers $l,m,n$ satisfying $2l+4m+6n=r+s+1$, $l\leq
\inf(r,s)+1$.

It is a fundamental fact \cite{Ram} that every modular form on
${\G}$ is uniquely expressible as a polynomial in $E_4$ and $E_6$
and that the extension ${\BC}[E_2,  E_4, E_6]$ of ${\BC}[E_4, E_6]$
is closed under differentiation, with the following dynamical system
of Ramanujan:
\begin{eqnarray} \label{S1}
\left\{ \aligned
E_2'&=\frac{1}{12}(E_2^2-E_4),\\
E_4'&= \frac{1}{3}(E_2 E_4- E_6), \\
E_6'&= \frac{1}{2}(E_2 E_6- E_4^2),
\endaligned
\right.
\end{eqnarray}
where the prime means the differentiation $\displaystyle
\frac{1}{2\pi\I}\frac{d}{d\tau}$.

The identity \eqref{E_2} can be expressed by saying that, up to a
constant factor, $E_2$ defines an affine connection, and the first
relation in \eqref{S1} similarly says that $E_4$ is the projective
connection which is the curvature of the affine connection $E_2$
(see Section~\ref{sec:connections} for the terminology). Hence the
Eisenstein series $E_2$ defines a covariant derivative sending
weight $k$ modular forms $f$ into weight $k+2$ modular forms
$\displaystyle f'-\frac{k}{12}E_2f$.

From the system \eqref{S1} we obtain
\begin{eqnarray} \label{S2}
\aligned
\begin{cases}
E_4&= E_2^2- 12 E_2', \\
E_6&= E_2 ^3- 18 E_2 E'_2+ 36 E_2''.
\end{cases}
\endaligned
\end{eqnarray}
An important consequence is that the function $E_2$ is a solution of
the {Chazy} equation
\begin{equation}\label{eq:Chazy}
E_2'''= E_2 E_2''- \frac{3}{2} {E_2'}^2.
\end{equation}
This means that we may use $E_2, E_2', E_2''$ instead of $E_2, E_4,
E_6$ to express modular forms, that is $\displaystyle \BC [E_2, E_4,
E_6] = \BC [E_2, E_2', E_2'']$.  We shall, in a later subsection,
give an illustration of this fact to determine the modulus of doubly
connected domains.


\subsection{Critical points of the Green's function and zeros
of the Bergman kernel}

The critical points of the Green's function $G(z,a)$ are by
(\ref{eq:GVJ}), (\ref{eq:upsilonV}) exactly the zeros of
$\upsilon_{a-J(a)}$. For the annulus there is exactly one critical
point, $z=z_G(a)$, and by symmetry this is located on the same
diameter as $a$. For the more detailed investigations one may pass
to the universal covering, via $t=\log z$, $u=\log a$, $-\bar{u}
=J(u)=\log J(a)$. Then, by (\ref{eq:upsilonzeta}), (\ref{eq:A}),
\begin{equation}\label{eq:upsilonuJu}
\upsilon_{u-J(u)}(t)=(\zeta(t-u)-\zeta(t+\bar{u})+\frac{2\eta_1 \Re
u}{\log R})dt.
\end{equation}
Thus the equation for the representation $t=t_G(u)$ of the critical
point is
\begin{equation}\label{eq:critptuniv}
\zeta(t_G(u)-u)-\zeta(t_G(u)+\bar{u})=-\frac{2\eta_1 \Re u}{\log R},
\end{equation}
where $0<\Re u<\log R$. In the plane of the annulus the
corrseponding equation, for $z=z_G(a)$, becomes
\begin{equation}\label{eq:critpt}
\zeta(\log \frac{z_G(a)}{a})-\zeta(\log
z_G(a)\bar{a})=-\frac{2\eta_1 \log |a|}{\log R}.
\end{equation}
These equations have been analyzed by A.~Maria \cite{Maria}, and the
results are summarized in Theorem~\ref{thm:SF} below.

The zeros of the Bergman kernel can be treated fairly explicitly. By
(\ref{eq:K(t,u)}), the zeros of $K(t,u)$ on the universal covering
surface are those points $t=t_K(u)$ for which
\begin{equation}\label{eq:tKu}
\wp(t_K(u)+\bar{u})=-\frac{\eta_1}{\log R}.
\end{equation}
Thus
$$
t_K(u)=s-\bar{u},
$$
where $s$ solves
\begin{equation}\label{eq:wpequation}
\wp(s)=-\frac{\eta_1}{\log R}.
\end{equation}
This equation has two solutions in each period parallelogram. None
of them are real because $\wp$ is positive on the real axis, while
the right member is negative. In fact, it is easy to see that every
solution has imaginary part $\pi$, modulo multiples of $2\pi$.

Let $s=s(R)$ denote that solution of (\ref{eq:wpequation}) which
satisfies $0\leq\Re s<\log R$, $\Im s=\pi$. It depends on $R$ or,
equivalently, on $\tau$ via  (\ref{eq:tau}). Define also
\begin{equation}\label{eq:rho}
\rho=\rho(R)=e^{\Re s}=|e^{s}|.
\end{equation}
According to \cite{Eichler-Zagier}, $s$ can be explicitly computed
to be
$$
s(R)= \I\pi+\frac{3\log R}{2} +\frac{3456\sqrt3\pi\log R}{5}
\int_{\tau} ^{\I\infty} \frac{\Phi_{5,6}(t)-\Delta(t)}
{\Phi_{2,3}(t)^{\frac{3}{2}}}(t- \tau) dt,
$$
where the integration is vertically upwards starting at
$\tau=\frac{\I\pi}{\log R}$.

The following theorem summarizes results in \cite{Maria} and
\cite{SF} and combines these with our findings for the zeros of the
Bergman kernel.

\begin{thm}\label{thm:SF}
With $\rho$ defined by (\ref{eq:rho}) we have

\begin{itemize}

\item[(i)]  $1<\rho<\sqrt{R}$.

\item[(ii)] The Green's function $G(z,a)$ of $A_{1,R}$ has, for any given $a\in A_{1,R}$,
a unique critical point $z=z_G(a)$. This is located on the same
diameter as $a$ but on the opposite side of the hole. More
precisely,
$$
z_G(a)=-g(|a|)\frac{a}{|a|},
$$
where $g:(1,R)\to (1,R)$ is an increasing function which maps
$(1,R)$ onto the relatively compact subinterval $(\rho, R/\rho)$. It
satisfies $g(x)g(R/x)=R$ for $1<x<R$, in particular
$g(\sqrt{R})=\sqrt{R}$, and
$$
\lim_{a\to 1} g'(a)=\lim_{a\to R} g'(a)=0.
$$

\item[(iii)] When $\rho\leq |a|\leq R/\rho$ the Bergman
kernel $K(z,a)$ has no zeros.

\item[(iv)] When $1<|a|<\rho$ or $R/\rho<|a|<R$,
$K(z,a)$ has exactly one zero, $z=z_K(a)$, and this is explicitly
given by
$$
z_K(a)= \begin{cases}
 -\frac{\rho}{\bar{a}} \quad {\rm when\,\,}
1<|a|<\rho,\\
-\frac{R^2}{\rho\bar{a}} \quad {\rm when\,\,} R/\rho<|a|<R.
\end{cases}
$$

\end{itemize}

\end{thm}

This theorem is special to the annulus but is effective.
We do not know of any such precise results in higher connectivity.

Note that the theorem confirms in a precise way the assertion of
Corollary~\ref{cor:critical} that the limiting set for the critical
points of the Green's function is exactly the zeros of the Bergman
kernel when the parameter variable is on the boundary. In fact,
choosing $a$ on the positive real axis (for simplicity of notation)
the theorem shows that
$$
\lim_{a\to 1} z_G(a)=\lim_{a\to 1} z_K(a)=-\rho,
$$
$$
\lim_{a\to R} z_G(a)=\lim_{a\to R} z_K(a)=-\frac{R}{\rho}.
$$
In addition we deduce from the theorem the following remarkable dichotomy result.

\begin{cor}
The annulus $A_{1,R}$ is the disjoint union
of the set of critical points $\{z_G(a)\}$ of the Green's function,
the set of zeros $\{z_K(a)\}$ of the Bergman kernel and the two
circles $\{|z|=\rho\}$ and $\{|z|=\frac{R}{\rho}\}$.
\end{cor}

\begin{proof}(of theorem)

By rotational symmetry we may assume that $a$ is real and positive,
namely $1<a<R$. It has been shown in \cite{Maria} that
$$
z_G(a)=-g(a)
$$
for some differentiable function $g:(1,R)\to (1,R)$, which is
increasing, more precisely $g'>0$, and satisfies
$$
0<\lim_{a\to 1} g(a)<\lim_{a\to R} g(a)<R,
$$
(the latter follows also from \cite{Solynin} and from our
Theorem~\ref{thm:critical}) and
$$
\lim_{a\to 1} g'(a)=\lim_{a\to R} g'(a)=0.
$$
Slightly more explicitly we have $g(e^u)=e^{f(u)}$, where the
function $f:(0,\log R)\to (0,\log R)$ is defined as the unique
solution of
$$
\zeta(f(u)+\I\pi-u)-\zeta(f(u)+\I\pi+u)=-\frac{2\eta_1  u}{\log R},
$$
or
\begin{equation}\label{eq:zetafu}
\frac{1}{2u}\zeta(f(u)+\I\pi-u)-\zeta(f(u)+\I\pi+u)=-\frac{\eta_1
}{\log R}.
\end{equation}

Since $g'>0$, and hence $f'>0$, $f(0)=\lim_{u\to 0} f(u)$ exists,
and letting $u\to 0$ in (\ref{eq:zetafu}) gives (in view of
$\zeta'=-\wp$)
$$
\wp (f(0)+\I\pi)=-\frac{\eta_1 }{\log R}.
$$
It follows that
$$
f(0)+\I\pi =s(R)
$$
(see (\ref{eq:wpequation})--(\ref{eq:rho})). In other words,
$$
\lim_{a\to 1} g(a)=-e^{\Re s(R)}= -\rho(R).
$$
Similarly,
$$
\lim_{a\to 1} g(a)=-e^{\Re s(R)}= -\frac{R}{\rho(R)}.
$$

We conclude that $\rho<R/\rho$, i.e., that $1<\rho<\sqrt{R}$. By
this parts (i)and (ii) of the theorem are proven.

The assertions (iii) and (iv), about the Bergman zeros, are easy
consequences of the analysis made before the statement of the
theorem. In fact, since $\wp$ is an even function, the two solutions
of (\ref{eq:wpequation}) are $\pm s$, modulo the period lattice. It
follows from (\ref{eq:K(t,u)}) that, given $u\in F_0$, the
(possible) solutions $t=t_K(u)$ of $K(t,u)=0$ are represented by
\begin{equation}\label{eq:t(u)}
t_K(u)=
\begin{cases}
s-\bar{u}, \\
-s-\bar{u},
\end{cases}
\end{equation}
modulo the period lattice. Both $u$ and $t_K(u)$ shall correspond to
points in $A_{1,R}$, i.e., $0<\Re u<\log R$ and $0<\Re t_K(u)<\log
R$. This occurs if and only if $0<\Re u<s$ or $\log R-s<\Re u<\log
R$, and then we have the same inequalities for the zero $t_K(u)$
itself: $0<\Re t_K(u)<s$ and $\log R-s<\Re t_K(u)<\log R$,
respectively.

In the plane of the annulus (\ref{eq:t(u)}) becomes
$$
\displaystyle{ z_K(a)=
\begin{cases}
\frac{\exp s}{\bar{a}},\\
\frac{\exp (-s)}{\bar{a}},
\end{cases}}
$$
and since this $z_K(a)$ is in $A_{1,R}$ if and only if $1<|a|<\rho$
or $R/\rho<|a|<R$ we have now proved statements (iii) and (iv) in
the theorem.

\end{proof}

It follows from (i) and (ii) of Theorem~\ref{thm:SF} that the
critical points of the Green's function for any annulus of modulus
$R>1$ are located in the annulus of modulus $R/\rho^2$ symmetrically
centered in the original annulus. Since $\rho=\rho(R)>1$ for all
$R>1$ and $\rho$ depends smoothly on $R$ one easily concludes the
following.

\begin{cor}
Let $R>1$. There exist a sequence $\{\rho_n\}$,
$1=\rho_0<\rho_1<\rho_2<\dots <\sqrt{R}$ with $\lim_{n\to \infty}
\rho_n= \sqrt R$ such that, on setting $A_n= \{ z\in \BC: \rho_n <
|z|< R/\rho_n \}$, the critical points for the Green's function of
$A_{n}$ are all contained in $A_{n+1}$.
\end{cor}

One might ask what would be corresponding statement in higher connectivity.


\subsection{Spectral point of view}\label{sec:cylinder}

The Green's function can also be obtained via spectral problems for
the Laplace operator. This requires a choice of a metric, or at
least a volume form. For a metric $d\sigma=\rho|dz|$, the volume
form is $\rho^2 dx dy$ and the invariant Laplacian is
$\displaystyle\frac{1}{\rho^2}\Delta$, where $\displaystyle\Delta
=\frac{\partial^2}{\partial x^2}+\frac{\partial^2}{\partial y^2}$.
Thus a natural spectral problem in a domain $\Omega$ is
$$
\begin{cases}
-\Delta u = \lambda\rho^2 u \quad {\rm in\,\,}\Omega, \\
u=0 \quad {\rm on\,\,}\partial\Omega.
\end{cases}
$$

If $\lambda_1, \lambda_2, \dots$ are the eigenvalues and
$u_1,u_2,\dots$ the eigenfunctions, normalized by
$$
\int_\Omega |u_n|^2\rho^2 dxdy =1,
$$
then the Green's function of $\Omega$ is given formally by
$$
G(z,w)=2\pi \sum_{n=1}^\infty \frac{u_n (z)u_n (w)}{\lambda_n},
$$
where however the sum may converge only in a weak sense. See for
example \cite{Courant}, \cite{Duff-Naylor}. (The factor $2\pi$
appears because  we have no factor $1/2\pi$ in front of the
singularity of the Green's function.) One way to cope with the
convergence problem is to consider the corresponding Dirichlet
series
$$
G^s(z,w)=2\pi  \sum_{n=1}^\infty \frac{u_n (z)u_n (w)}{\lambda_n^s},
$$
which is an analytic function in $s$ for $\Re s>1$. One then studies
the behaviour as $s\to 1$. To continue analytically $G^s (z,w)$ we
shall use the second Kronecker limit formula which is basically a
connection between Epstein zeta functions and Eisenstein series
\cite{Siegel}.

The eigenfunctions $u_n$ form an orthonormal set both with respect
to the $L^2$-inner product $\int_{\Omega}uv\rho^2 dxdy$ and with
respect to the Dirichlet inner product $D(u,v)=\int_{\Omega}\nabla
u\cdot\nabla v dxdy$. Note that the $\lambda_n$ and $u_n$ depend on
$\rho$, but $G(z,w)$ does not.

We shall elaborate the above approach to the Green's function in the
annulus case, $\Omega=A_{1,R}$, and with metric coming from
interpreting $A_{1,R}$ as a cylinder. This metric is equivalent to
the natural flat metric on the universal covering surface of
$A_{1,R}$. Thus we consider the lift map $z\mapsto t=\log z$, by
which $A_{1,R}$ gets identified with
$$
C_R=\{t\in{\mathbb{C}}: 0<\Re t <\log R\}/2\pi\I {\mathbb{Z}}.
$$
This can be thought of as a cylinder, with boundary $\partial C_R$
represented by the vertical lines $\{\Re t=0\}$ and $\{\Re t=\log
R\}$ and provided with the euclidean metric $ d\sigma^2=|dt|^2$.

On $A_{1,R}$ itself the metric is given by $\rho(z)={1}/{|z|}$
($z\in A_{1,R}$). We shall keep some previous notations, for example
(\ref{eq:tau}), and sometimes use $\tau$ as a parameter in place of
$R$. When working in $C_R$ directly, the eigenvalue problem becomes
(since $\rho=1$ in $C_R$)
$$
\begin{cases}
-\Delta u = \lambda u \quad {\rm in\,\,} C_R, \\
u=0 \quad {\rm on\,\,}\partial C_R.
\end{cases}
$$
This problem can be solved by standard separation of variables
techniques. The eigenvalues then come with two indices, namely
$$
\lambda_{mn}=\frac{m^2\pi^2}{(\log R)^2}+n^2= n^2-\tau^2 m^2,
$$
for $m=1,2,\dots$, $n=0,1,2,\dots$. Using also negative values of
$n$ one can list the eigenfunctions as follows (now deviating from a
previous notational convention and using the variables $x$, $y$,
$z=x+\I y$ on the universal covering surface).
\begin{eqnarray*}
u_{mn}(z)= \left\{ \aligned \frac{1}{\sqrt{\pi\log R}}\sin
\frac{m\pi x}{\log R},
\quad& (m\geq 1,\, n=0),\\
\frac{2}{\sqrt{\pi\log R}}\sin \frac{m\pi x}{\log R}\cos ny, \quad&
(m\geq 1,\, n\geq 1), \\
\frac{2}{\sqrt{\pi\log R}}\sin \frac{m\pi x}{\log R}\sin ny, \quad&
(m\geq 1,\, n\leq -1).
\endaligned
\right.
\end{eqnarray*}

On using exponentials in place of trigonometric functions the above
gives, after simplifications and still working on the universal
covering surface.
$$
G^s(z,w)=2\pi  \sum_{m\geq 1, n\in {\mathbb{Z}}}^\infty \frac{u_{mn}
(z)u_{mn} (w)}{\lambda_{mn}^s}
$$
$$
=\frac{1}{2 \log R}\sum_{(m,n)\in{\mathbb{Z}}^2\setminus \{(0,0)\}}
\frac{(e^{m\tau(x-u)}-e^{m\tau(x+u)})e^{\I
n(y-v)}}{(n^2-m^2\tau^2)^s},
$$
$$
=\frac{1}{2\log R}\sum_{(m,n)\in{\mathbb{Z}}^2\setminus \{(0,0)\}}
\frac{(e^{m\tau(x-u)}-e^{m\tau(x+u)})e^{\I n(y-v)}}{|m\tau+n|^{2s}},
$$
where $z=x+\I y$, $w=u+\I v$. To go on further we shall need
Kronecker's second limit formula, contained in the following lemma.

\begin{lem}\label{lem:Kronecker2}
Let $\tau \in \BC$, $\Im \tau >0$. Let $u, v \in \BR \setminus \BZ$
be fixed and define  for $\Re s>1$ the zeta function
$$
\zeta_{\tau} (s; u,v)= \sum_{(m,n)\in{\mathbb{Z}}^2\setminus
\{(0,0)\}} \left( \frac{\Im\tau}{\vert m\tau + n\vert ^2}\right)^s
e^{ 2\pi\I(mu+ nv)}.
$$
Then the analytic continuation at $s= 1$ of $\zeta_{\tau} (s; u,v)$
is
$$
\zeta_{\tau} (1; u,v)= 2\pi ^2v^2 \Im\tau - 2\pi \log \vert
\frac{\vartheta_{1} (u-v\tau; \tau)}{\eta(\tau)} \vert,
$$
where the theta function $\vartheta_1$ is defined (cf.
(\ref{eq:theta})) by
$$
\displaystyle \vartheta_1(z)=\vartheta_{1}(z; \tau) =\vartheta
\begin{bmatrix}
1/2\\
1/2
\end{bmatrix}(z;\tau)
= \sum_{n\in \BZ} e^{\I \pi (n+\frac{1}{2})^2 \tau + 2\pi \I
(n+\frac{1}{2})(z+ \frac{1}{2})}
$$
$$
=-2\sum_{n=0}^\infty (-1)^n q^{\frac{1}{2}(n+\frac{1}{2})^2}\sin
(2n+1)\pi z
$$
$$
= -2 q^{\frac{1}{8}}\sin\pi z \prod_{n= 1}^\infty  (1- q^n) (1-2q^n
\cos 2\pi z+ q^{2n}),
$$
and where $\eta$ is the Dedekind eta function,
$$
\eta(\tau)= q ^{\frac{1}{24}} \prod_{n \geq 1} (1-
q^n)=\Delta(\tau)^{\frac{1}{24}}.
$$
\end{lem}

We remark that for $u,v \in \BZ$, the function $\zeta_{\tau} (1;
u,v)$ reduces to the Epstein zeta function
$$
\zeta_{\tau} (s)= \sum_{(m,n)\in{\mathbb{Z}}^2\setminus
\{(0,0)\}}\left( \frac{\Im \tau}{\vert m+ n\tau\vert ^2}\right)^s,
$$
which is no longer analytic at $s= 1$ but only meromorphic.
Kronecker's first limit formula (below) gives the meromorphic
continuation \cite{Siegel}.

\begin{lem}
The Epstein zeta function $\zeta_{\tau} (s)$ has a meromorphic
continuation to all $\BC$ with a simple pole at $s= 1$, given by
$$
\zeta_{\tau} (s)= \frac{\pi}{s-1}+ 2\pi \left(\gamma-
\ln(2\sqrt{\Im\tau} \vert \eta(\tau)\vert^2)\right)+\mathcal{O}(s-1)
$$
where $\gamma$ is the Euler constant. In particular $\zeta_{\tau}
(0)= -1$.
\end{lem}

Continuing the computation of the Green's function we have, using
Lemma~\ref{lem:Kronecker2},
$$
G(z,w)=\lim_{s\to 1} G^s(z,w)
$$
$$
=\lim_{s\to 1}\frac{\Im \tau}{2\pi}
\sum_{(m,n)\in{\mathbb{Z}}^2\setminus \{(0,0)\}}
\left(\frac{e^{-(x-u)\Im\tau}e^{\I n(y-v)}}{|m\tau+n|^{2s}} -
\frac{e^{-(x+u)\Im\tau}e^{\I n(y-v)}}{|m\tau+n|^{2s}}\right)
$$
$$\displaystyle
=\lim_{s\to 1}\frac{(\Im \tau)^{1-s}}{2\pi} \left(\zeta_{\tau} (s;
-\frac{(x-u)\Im \tau}{2\pi},\frac{y-v}{2\pi})-  \zeta_{\tau} (s;
-\frac{(x+u)\Im \tau}{2\pi},\frac{y-v}{2\pi})\right)
$$
$$
\displaystyle =-\log|{\vartheta_1(-\frac{(x-u)\Im
\tau}{2\pi}-\frac{(y-v)\tau}{2\pi})}|
+\log|{\vartheta_1(\frac{-(x+u)\Im
\tau}{2\pi}-\frac{(y-v)\tau}{2\pi})}|
$$
$$
=-\log|{\vartheta_1(\frac{(z-w)\tau}{2\pi\I})}|
+\log|{\vartheta_1(\frac{(z+\bar{w})\tau}{2\pi\I})}|
$$
$$
=-\log|{\vartheta_1(\frac{z-w}{2\log R})}|
+\log|{\vartheta_1(\frac{z+\bar{w}}{2\log R})}|.
$$

This was all on the universal covering. To get back to $A_{1,R}$ one
simply substitutes $\log z$, $\log w$ for $z$, $w$:
$$
G(z,w)=-\log|{\vartheta_1(\frac{\log z-\log w}{2\log R})}|
+\log|{\vartheta_1(\frac{\log z+\log \bar{w}}{2\log R})}|.
$$
A similar formula is given in \cite{Courant}. Compare also
(\ref{eq:green-Crowdy}).


\subsection{A remark on the modulus of a doubly connected domain}

The connection between the preceeding topics and the Ramanujan partial
differential system (\ref{S1}) or the Chazy equation (\ref{eq:Chazy})
can be explored in other directions. To give an example, we
quote a result in \cite{Epstein} showing that the Bergman minimum
integrals \cite{Bergman} completely determines the modulus of a doubly connected
domain.

In any domain $\Omega$, consider the problem of minimizing
$\|f\|^2=\int_\Omega|f|^2dxdy$ among functions satisfying
$$
\begin{cases}
f(a)=\dots=f^{(n-1)}(a)=0,\\
f^{(n)}(a)=1,
\end{cases}
$$
where $a\in\Omega$ is a fixed point and $n\geq 1$. Let $f_n=f_{n,a}$
be the unique minimizer and let
$\lambda_n=\lambda_n(a)=\|f_{n,a}\|^2$ be the minimum value. Then it
is easy to see that $\{f_n/\sqrt{\lambda_n}: n=0, 1,\dots\}$ forms
an orthonormal basis in the Bergman space, and hence that
\begin{equation}\label{eq:Kexp}
K(z,w)=\sum_{n=0}^\infty \frac{f_n(z)\overline{f_n(w)}}{\lambda_n}.
\end{equation}
Here the dependence on the point $a\in\Omega$ has disappeared, even
though each individual $f_n$ and $\lambda_n$ depend on $a$.

Let $d\sigma=\sqrt{\pi K(z,z)}|dz|$ be the Bergman metric of
$\Omega$. It has Gaussian curvature $\kappa=\kappa_{\rm Gauss}<0$,
given in terms of the Bergman kernel and the Bergman minimum
integrals by
$$
\kappa(z)= -\frac{2}{\pi K(z,\bar{z})} \frac{\partial^2{\log
K(z,\bar{z})}}{\partial z\partial\bar{z}}
=-2\frac{\lambda_0^2(z)}{\lambda_1(z)}.
$$
The last expression follows from (\ref{eq:Kexp}) (see \cite{Epstein}
for details). Set also
$$
\gamma(z) =\frac{1}{\pi K(z,\bar{z})}\frac{\partial^2{\log (-\kappa
(z))}}{\partial z\partial\bar{z}}
=\frac{\lambda_0(z)\lambda_1(z)}{\lambda_2(z)}-3\frac{\lambda_0^2(z)}{\lambda_1(z)}.
$$
Both $\kappa$ and $\gamma$ are conformally invariant scalar
functions.

Now we specialize to doubly connected domains. Taking in this subsection
$\Omega=A_{r,1}$ ($0<r<1$) as a model, the conformal invariance implies that
$\kappa(z)$ and $\gamma(z)$ depend only on $|z|$ and that they are
constant on $\partial A_{r,1}$ (the same constant on both boundary
components). These constants, $\kappa_r=\kappa|_{\partial A_{r,1}}$
and $\gamma_R=\gamma|_{\partial A_{r,1}}$, depend on the modulus
$R$. It is shown in \cite{Epstein} that
\begin{eqnarray}\label{Modulus1}
\frac{ \frac{2}{3}{\gamma_r}+ 2{\kappa_r}+ 8} {\left(-{\kappa_r}- 4
\right)^{\frac{3}{2}}} = \frac{6 g_3- 14\alpha g_2+
120\alpha^3}{\left(g_2 - 12\alpha ^2\right)^{\frac{3}{2}}},
\end{eqnarray}
where
$$
\tau= \frac{\log r}{\I\pi}, \quad \alpha= -\frac{1}{12} E_2(\tau),
$$
$$
g_2= \frac{1}{12} E_4(\tau), \quad g_3= -\frac{1}{216}E_6 (\tau).
$$
The quantity on the right-hand side of \eqref{Modulus1} depends on
$r$ or, equivalently, on $\tau$. Denote it by $f(\tau)$:
$$
f(\tau)=\frac{6 g_3- 14\alpha g_2+ 120\alpha^3}{\left(g_2 - 12\alpha
^2\right)^{\frac{3}{2}}}.
$$
It is shown in \cite{Epstein} that $f(\tau)$ is monotone as a function
of $r$. Thus $f(\tau)$ determines the modulus. By using
\eqref{S2} and Chazy equation, we also find the explicit expressions
\begin{eqnarray}\label{Modulus2}
\aligned
f(\tau)= &\frac{3 {E_2}^{''}+ 2E_2 {E_2}'}{{E_2'}^{\frac{3}{2}}},\\
f'(\tau)= &\frac{-9 {E_2''}^{2}+ 4 E_2 {E_2}'{E_2}'' -5{E_2'}^{3}
}{{2 E_2'}^{\frac{5}{2}}},
\endaligned
\end{eqnarray}
where the prime denotes differentiation
$\displaystyle\frac{1}{2\pi\I}\frac{d}{d\tau}$, as before.

Thus the Eisenstein series
$E_2$, which is an affine connection, is useful to study the modulus of a doubly connected domain.
In the next section we will study affine, and projective, connections in more generality.


\section{Projective and affine connections}\label{sec:connections}

\subsection{Definitions}

We shall review some aspects of the theory of affine and projective
connections. General references here are
\cite{Schiffer-Hawley}, \cite{Gunning},
\cite{Gustafsson-Peetre}.

Let $z=f(t)$, where $z$, $t$ are complex variables and $f$
holomorphic. We introduce the nonlinear differential operators
$$
\{z,t\}_0= \log \frac{dz}{dt} = -2 \log \frac{1}{\sqrt{f'}},
$$
$$
\{z,t\}_1= \frac{d}{dt}\log \frac{dz}{dt} = -2 \sqrt{f'}\,\,
(\frac{1}{\sqrt{f'}})',
$$
$$
\{z,t\}_2= \frac{d^2}{dt^2}\log \frac{dz}{dt}
-\frac{1}{2}(\frac{d}{dt}\log \frac{dz}{dt})^2
=\frac{f'''}{f'}-\frac{3}{2}(\frac{f''}{f'})^2 = -2 \sqrt{f'}
\,\,(\frac{1}{\sqrt{f'}})''.
$$
The latter is the \emph{schwarzian derivative} of $f$. For
$\{z,t\}_0$ there is an additive indeterminancy of $2\pi \I$, so
actually only its real part is completely well-defined. With $s$ an
intermediate complex variable, so that $z=g(s)$, $s=h(t)$ for some
holomorphic functions $g$ and $h$, we have
\begin{equation}\label{composition}
\{z,t\}_k (dt)^k=\{z,s\}_k (ds)^k+\{s,t\}_k (dt)^k \quad (k=0,1,2),
\end{equation}
see \cite{Schiffer-Hawley}. In addition, one notes that
$$
\{z,t\}_0 =0 \quad {\rm iff}\quad z=t+b \quad (f {\rm \,\, is\,\,
a\,\, translation}),
$$
$$
\{z,t\}_1 =0 \quad {\rm iff}\quad z=at+b \quad (f {\rm \,\, is\,\,
an\,\,affine\,\, map}),
$$
\begin{equation}\label{eq:Mobius}
\{z,t\}_2 =0 \quad {\rm iff}\quad z=\frac{at+b}{ct+d} \quad (f {\rm
\,\, is\,\, a\,\, M\ddot{o}bius \,\, transformation}).
\end{equation}

The above differential operators can also be obtained as limits of
derivatives of logarithms of difference quotients of the holomorphic
function $f$. To be precise, let $z$, $w$ and $t$, $u$ be related by
$z=f(t)$, $w=f(u)$. Then
\begin{equation}\label{eq:logdiff}
\displaystyle
\begin{cases}
\{z,t\}_0=\lim_{u\to t} \log\frac{w-z}{u-t},\\
\{z,t\}_1=2\lim_{u\to
t} \frac{\partial}{\partial t}\log\frac{w-z}{u-t},\\
\{z,t\}_2=6\lim_{u\to t} \frac{\partial^2}{\partial t\partial
u}\log\frac{w-z}{u-t}.
\end{cases}
\end{equation}
Alternatively, for the last case, one may introduce a polarized
version of the Schwarzian derivative by
$$
[z,w;t,u]=6\frac{\partial^2}{\partial t\partial
u}\log\frac{w-z}{u-t}.
$$
Then
$$
\{z,t\}_2 = [z,z;t,t].
$$

Now, let $M$ be a Riemann surface. An {\it affine structure} on $M$
is a choice of holomorphic atlas such that all transition functions
between coordinates within this atlas are affine maps, i.e., such
that for any two coordinates $z$ and $\tilde{z}$ for which the
domains of definition overlap, the relation $\{\tilde{z},z\}_1=0$
holds. This is a quite demanding requirement, and not every Riemann
surface admits an affine structure. Of the compact Riemann surfaces
only those of genus one do.

Less demanding is a {\it projective structure}. It is a choice of
holomorphic atlas on $M$ such that all coordinate changes are
M\"obius transformations, i.e., such that $\{\tilde{z},z\}_2=0$
whenever $z$ and $\tilde{z}$ are coordinates with overlapping
domains of definition. By the uniformization theorem every Riemann
surface admits a projective structure.

Let $\phi(z)(dz)^m=\tilde{\phi}(\tilde{z})(d\tilde{z})^m$ be local
expressions for a differential of order $m$ on $M$. Thus the
coefficients transform under holomorphic change of coordinates
according to
$\phi(z)=\tilde{\phi}(\tilde{z})(\frac{d\tilde{z}}{dz})^m$. We shall
allow $m$ not only to be an integer, but also a half-integer. This
requires a consistent choice of signs in the multipliers
$(\frac{d\tilde{z}}{dz})^m$, i.e., requires a choice of a square
root of the canonical bundle (which has transition functions
$\frac{d\tilde{z}}{dz}$). Such a square root always exists, in fact
there are in general several inequivalent choices. We refer to
\cite{Gunning}, \cite{Hawley-Schiffer}, \cite{Hejhal} for details
about this. Recall however (Section~2) that in case $M$ is the
Schottky double of a plane domain then there is a canonical choice
of square root of the canonical bundle, namely obtained by choosing
the square root of the Schwarz function to be $1/T(z)$, where $T(z)$
the tangent vector of the oriented boundary. This is the choice of
square root which will be used in the sequel.

An {\it affine connection} on $M$ is an object which is represented
by local differentials $r(z)dz$,
$\tilde{r}(\tilde{z})d\tilde{z}$,\dots (one in each coordinate
variable) glued together according to the rule
\begin{equation}\label{eq:affconnection}
\tilde{r}({\tilde{z}}){d\tilde{z}}=r(z){dz}-\{\tilde{z},z\}_1\,{dz}.
\end{equation}
In the presence of an affine connection it is possible to define,
for every $k\in\frac{1}{2}{\mathbb{Z}}$, a {\it covariant derivative}
$\nabla_k$ from $k$:th order differentials to $(k+1)$:th order
differentials by $\phi(dz)^k\mapsto (\nabla_k \phi)(dz)^{k+1}$,
where
\begin{equation}\label{eq:affderivative}
\nabla_k \phi = \frac{\partial\phi}{\partial z}-kr\phi.
\end{equation}
The covariance means that if $\phi(dz)^k=\tilde{\phi}(d\tilde{z})^k$
then $\nabla_k
\phi(dz)^{k+1}=\tilde{\nabla}_k\tilde{\phi}(d\tilde{z})^{k+1}$.

Similarly, a {\it projective connection} on $M$ consists of local
quadratic differentials $q(z)(dz)^2$,
$\tilde{q}(\tilde{z})(d\tilde{z})^2$, \dots, glued together
according to
\begin{equation}\label{eq:projconnection}
\tilde{q}({\tilde{z}})({d\tilde{z}})^2=q(z)({dz})^2-\{\tilde{z},z\}_2\,({dz})^2.
\end{equation}
From (\ref{composition}) it follows that this law (as well as
(\ref{eq:affconnection})) is associative. In general, we do not
require $r$ and $q$ to be holomorphic, although our main interest is
in the holomorphic (or meromorphic) case.

In addition to the above one may consider also $0$-connections,
quantities defined up to multiples of $2\pi \I$ and which transform
according to
\begin{equation}\label{eq:zeroconnection}
\tilde p (\tilde z)= p(z) -\{\tilde z, z\}_0.
\end{equation}
This means exactly that $e^{p(z)}$ is well-defined and transforms as
differential of order one.

A projective connection is less powerful than an affine one, but it
still allows for certain covariant derivatives: for each $m=0, 1, 2,
\dots$ there is, in the presence of a projective connection $q$, a
well-defined linear differential operator $\Lambda_m$ taking
differentials of order $\frac{1-m}{2}$ to differentials of order
$\frac{1+m}{2}$: $\phi (dz)^{\frac{1-m}{2}}\mapsto (\Lambda_m\phi)
(dz)^{\frac{1+m}{2}}$. The first few look, in a local
coordinate $z$,
$$
\Lambda_0(\phi)=\phi ,
$$
$$
\Lambda_1(\phi)=\frac{\partial\phi}{\partial z},
$$
$$
\Lambda_2(\phi)=\frac{\partial^2\phi}{\partial
z^2}+\frac{1}{2}q\phi,
$$
$$
\Lambda_3(\phi)=\frac{\partial^3\phi}{\partial
z^3}+2q\frac{\partial\phi}{\partial z}+\frac{\partial q}{\partial
z}\phi,
$$
$$
\Lambda_4(\phi)=\frac{\partial^4\phi}{\partial
z^4}+10q\frac{\partial^2\phi}{\partial z^2}
+10\frac{\partial q}{\partial z}\frac{\partial \phi}{\partial z}
+(9q^2+3\frac{\partial^2 q}{\partial z^2})\phi.
$$
The covariant derivative $\Lambda_m$ will be called the $m$:th order
{\it Bol operator}, for reasons to be explained further on.

Any affine connection $r$ gives rise to a projective connection $q$
by
\begin{equation}\label{qr}
q= \frac{\partial r}{\partial z}- \frac{1}{2}r^2.
\end{equation}
This $q$ is sometimes called the ``curvature" of $r$  (cf.
\cite{Dubrovin}). In fact, its definition is analogous to that of
the curvature form in ordinary differential geometry, see
\cite{Frankel}. Slightly more generally than (\ref{qr}), any two
affine connections $r_j(z)dz$, $j=1,2$, combine into a projective
connection by
$$
q= \frac{1}{2}(\frac{\partial r_1}{\partial z} +\frac{\partial
r_2}{\partial z}- r_1 r_2).
$$

Not every projective connection is the curvature of an affine
connection. In the case that a projective connection comes from an
affine connection, as in (\ref{qr}), the corresponding covariant
derivatives are related by
\begin{equation}\label{eq:Lnablanabla}
\Lambda_1=\nabla_0, \quad \Lambda_2
=\nabla_{\frac{1}{2}}\nabla_{-\frac{1}{2}}, \quad
\Lambda_3=\nabla_{1}\nabla_{0}\nabla_{-1},\quad {\rm etc}.
\end{equation}
See \cite{Gustafsson-Peetre} for the proof.

The difference between two projective connections is a quadratic
differential, and the difference between two affine connections is
an ordinary differential. Hence, if $q(z)(dz)^2$ is one projective
connection the most general one is $q(z)(dz)^2$ plus a quadratic
differential. Similarly for affine connections.

Now, a central fact is that there is a one-to-one correspondence
between holomorphic projective connections and projective
structures: given a projective connection, represented by a
holomorphic function $q(z)$ in a general coordinate $z$, a
projective coordinate $t$ is obtained by solving the differential
equation
\begin{equation}\label{eq:SchwarzianDE}
\{t,z\}_2 =q(z).
\end{equation}
It follows from (\ref{composition}), (\ref{eq:Mobius}) that the set
of coordinates obtained in this way are related by M\"obius
transformations. In the other direction, given a projective
structure, a projective connection is obtained by simply setting
$q(t)=0$ in any projective coordinate $t$.

One way to solve (\ref{eq:SchwarzianDE}), when $q(z)$ is
holomorphic, is to consider the second order linear differential
equation
\begin{equation}\label{eq:secondorderDE}
\frac{\partial^2 u}{\partial z^2}+ \frac{1}{2}q(z)u =0,
\end{equation}
i.e., $\Lambda_2(u)=0$, for $u$ considered as a differential of
order minus one-half. The solutions to (\ref{eq:SchwarzianDE}) are
exactly the functions
$$
t(z)= \frac{u_2(z)}{u_1(z)},
$$
where $u_1$, $u_2$ are linearly independent solutions of
(\ref{eq:secondorderDE}).

In terms of a projective structure, the meaning of the covariant
derivatives $\Lambda_m$, for holomorphic $q$, is that in any
projective coordinate $t$ corresponding to the connection,
$\Lambda_m$ simply is the $k$:th order derivative with respect to
$t$: $\Lambda_m(\phi(t)(dt)^{\frac{1-m}{2}})=\frac{\partial^m
\phi(t)}{\partial t^m}(dt)^{\frac{1+m}{2}}$. The fact that the right
member here is covariant under M\"obius transformation is sometimes
called ``Bol's lemma" \cite{Bol}, \cite{Gustafsson-Peetre}. The
precise statement of this lemma is as follows.

\begin{lem}\label{lem:Bol}\cite{Bol}
Let $z=f(t)=\frac{at+b}{ct+d}$, $ad-bc =1$,
$\lambda(t)=(ct+d)^{-1}$, so that $f'(t)=\lambda(t)^2$. Then, for
any smooth function $F(z)$ and any positive integer $m$,
$$
\frac{\partial^m}{\partial t^m}(F(f(t))\lambda(t)^{1-m})
=\frac{\partial^m F}{\partial z^m}(f(t))\lambda(t)^{1+m}.
$$
\end{lem}

For $m=1$ this is the ordinary chain rule, holding for any change of
coordinate $f$, whereas for $m\geq 2$ the formula holds only for
M\"obius transformations.

In any projective coordinate $t$, a natural fundamental set (basis)
of solutions of the equation $\Lambda_m u=0$ is
$\{1,t,\dots,t^{m-1}\}$. Considering these functions as
$\frac{1-m}{2}$-forms (namely $\displaystyle (dt)^{\frac{1-m}{2}}$,
$\displaystyle t(dt)^{\frac{1-m}{2}}$ etc.) and turning to a general
coordinate $z$, this basis transforms into $\{u_1^{m-1},
u_1^{m-2}u_2,\dots,u_2^{m-1}\}$, where $u_1=\sqrt{\frac{dz}{dt}}$,
$u_2=t\sqrt{\frac{dz}{dt}}$ are $-\frac{1}{2}$-forms. The
transformation property of $\Lambda_m$ then shows that $\{u_1^{m-1},
u_1^{m-2}u_2,\dots,u_2^{m-1}\}$ is a fundamental solution set for
$\Lambda_m u=0$ when considered as a differential equation in the
$z$ variable. It follows that the operator $\Lambda_m$ agrees with
the $(m-1)$-fold symmetric product of $\Lambda_2$: $\Lambda_m=
S^{m-1}(\Lambda_2)$.

In summary, the Bol operators $\Lambda_m$, $m\geq 2$, are all generated by $\Lambda_2$,
and their solutions are generated by two solutions of $\Lambda_2 u=0$. In the last
section of the paper we shall discuss a method, via a prepotential, of finding
all solutions of $\Lambda_2 u=0$ from a single one.

Turning to affine connections there are analogous statements as for
projective connections: there is a one-to-one correspondence between
affine structures and holomorphic affine connections, the
correspondence between an affine (or ``flat") coordinate $\tau$ and
an affine connection $r(z)$ given in a general coordinate $z$ being
\begin{equation}\label{eq:firstorderDE}
\{\tau, z\}_1 =r(z).
\end{equation}
This equation can be directly integrated as
$$
\tau (z) =\int^z \exp(\int^w r(\zeta)d\zeta) dw.
$$
One may also think of (\ref{eq:firstorderDE}) as
$$
\nabla_1\nabla_0\, \tau =0,
$$
or $\nabla_1 (d\tau) =0$, cf. \cite{Dubrovin}. With $\tau$ used as
coordinate, $\nabla_m(\phi)=\frac{\partial \phi}{\partial \tau}$ for
any $m$-differential $\phi(\tau)(d\tau)^m$.

Along with $\Lambda_m$ and $\nabla_m$ it is natural to introduce
their conjugates $\bar{L}_m$ and $\overline{\nabla}_m$, defined by
replacing $\frac{\partial}{\partial z}$ by $\frac{\partial}{\partial
\bar{z}}$ and $q$, $r$ by $\bar{q}$, $\bar{r}$.

Even more powerful than an affine connection is a hermitean metric,
which we in the present section prefer to write in either of the
following two ways:
\begin{equation}\label{eq:hermitean}
d\sigma= \frac{|dz|}{\omega(z)}=e^{p(z)}|dz|.
\end{equation}
Here $\omega>0$ transforms as the coefficient of a form of bidegree
$(-\frac{1}{2}, -\frac{1}{2})$ (i.e., $\omega
(dz)^{-1/2}(d\bar{z})^{-1/2}$ is invariant), whereas $p(z)=-\log
\omega(z)$, in analogy with (\ref{eq:affconnection}),
(\ref{eq:projconnection}), transforms as the real part of a
$0$-connection:
\begin{equation}\label{eq:affmetric}
\tilde{p}(\tilde{z})=p(z)- \Re \{\tilde{z},z\}_0.
\end{equation}

Any hermitean metric gives rise to a, not necessarily holomorphic,
affine connection by
\begin{equation}\label{eq:rfrommetric}
r= -2 \frac{\partial }{\partial z}\log \omega =2 \frac{\partial
p}{\partial z}.
\end{equation}
This relationship says that the covariant derivative of the metric
vanishes: $ \nabla_{\frac{1}{2}} ({1}/{\omega})
=\overline{\nabla}_{\frac{1}{2}} ({1}/{\omega})=0$, or, which turns
out to be the same,
\begin{equation}\label{eq:nablaomega}
\nabla_{-\frac{1}{2}} \omega =\overline{\nabla}_{-\frac{1}{2}}
\omega=0.
\end{equation}
Clearly (\ref{eq:Lnablanabla}), (\ref{eq:nablaomega}) imply that
\begin{equation}\label{eq:Lomega} \Lambda_2 \omega =0,
\end{equation}
and also that $\bar{\Lambda}_2 \omega =0$. Moreover, it is easy to
check that $\Lambda_m \omega^{m-1}=\bar{\Lambda}_m \omega^{m-1}=0$
for any $m\geq 2$.

The Gaussian curvature of the metric (\ref{eq:hermitean}) is
\begin{equation}\label{eq:gaussian}
\kappa_{\rm Gauss}=-e^{-2p}\Delta p =4(\omega\, \frac{\partial^2
\omega}{\partial z {\partial}\bar{z}}-\big|\frac{\partial
\omega}{\partial z}\big|^2),
\end{equation}
which is a real and scalar quantity. In view of
(\ref{eq:rfrommetric}), it follows that $r$ is holomorphic if and
only if $\kappa_{\rm Gauss}=0$. The corresponding projective
connection is
\begin{equation}\label{eq:qfrommetric}
q= \frac{\partial r}{\partial z}- \frac{1}{2}r^2
=2(\frac{\partial^2p}{\partial z^2}- (\frac{\partial p}{\partial
z})^2) =-2\,\frac{\partial^2\omega /\partial z^2}{\omega}.
\end{equation}
From (\ref{eq:gaussian}), (\ref{eq:qfrommetric}) follows
$\frac{\partial \kappa}{\partial z}=- 2\omega^2\frac{\partial
q}{\partial\bar{z}}$ or, since $\kappa$ is real,
$$
d\kappa_{\rm Gauss}=- 2\omega^2(\frac{\partial
q}{\partial\bar{z}}dz+\frac{\partial \bar{q}}{\partial
{z}}d\bar{z}).
$$
Cf. \cite{Gustafsson-Peetre}. It follows that $q$ is holomorphic if
and only if the curvature $\kappa_{\rm Gauss}$ is constant. When
this is the case there is a corresponding projective structure, and
working in any projective coordinate $t$, equation (\ref{eq:Lomega})
and its conjugate can be directly integrated to give
$\omega(t)=at\bar{t} +bt+\bar{b}\bar{t}+c$, $a$, $c$ real, $b$
complex. The value of the curvature comes out to be $\kappa_{\rm
Gauss} =4(ac-|b|^2)$. It also follows that if $\Omega$ is complete
for the metric, meaning roughly that $\omega=0$ on $\partial\Omega$,
then in any projective coordinate $\partial\Omega$ is a circle or a
straight line.

Being complete and having constant negative curvature characterizes
the Poincar\'{e} metric up to a constant factor. Starting from the
domain, first the Poincar\'{e} metric (normalized so that
$\kappa_{\rm Gauss}=-4$), then the projective connection $q(z)$ and
finally projective coordinates, can be obtained directly by solving
differential equations in $\Omega$. In fact, $p(z)$ satisfies by
(\ref{eq:gaussian}) a Liouville equation and blows up on the
boundary:
$$
\begin{cases}
\Delta p =4e^{2p} \quad {\rm in\,\,} \Omega,\\
p=+\infty  \quad {\rm on\,\,} \partial\Omega.
\end{cases}
$$
It can be shown  that this boundary value problem, when properly
formulated, has a unique solution $p$ (see, e.g., \cite{Ahlfors},
Ch.1). From this solution, $q$ is obtained via
(\ref{eq:qfrommetric}), and finally projective coordinates $t$ are
gotten by solving (\ref{eq:SchwarzianDE}), more precisely by taking
the quotient between two solutions of (\ref{eq:secondorderDE}). The
above indicates one of the early attempts to solve the
uniformization theorem. It was proposed by H.~A.~Schwarz and later
brought to an end by Picard, Poincar\'e and Bieberbach; see
\cite{Poincare}, \cite{Bieberbach}.

This approach has recently gained new interest in some versions of
string theory, in which the above Schwarzian derivative
$q(z)=\{t,z\}_2$ of the lift map $t$ to the universal covering
surface takes the role of being the energy-momentum tensor, which
hence is a projective connection. See for example \cite{Polchinski},
\cite{Takhtajan-Teo}.


\subsection{Examples of projective connections}

If $\Omega$ is a plane domain bounded by finitely many analytic
curves, several projective structures and connections can be
naturally associated to $\Omega$.

\begin{ex} The trivial one,  $q(z)=0$, i.e., with the coordinate variable
$z$ in ${{\mathbb{C}}}$ as a projective coordinate.
\end{ex}

\begin{ex} Let $t: \Omega\to \Omega_{\rm circ}$ be a conformal map of
$\Omega$ onto a domain $\Omega_{\rm circ}$ bounded by circles. This
map is not uniquely determined, but any two such maps are related by
a M\"obius transformation. Hence it defines a unique projective
structure for which $t$ is a projective coordinate. This projective
structure on $\Omega$ extends to a projective structure on the
Schottky double $\Hat{\Omega}$, as will be seen more exactly in the
next section. The associated connection coefficient is obtained from
(\ref{eq:SchwarzianDE}).
\end{ex}

\begin{ex} The universal covering surface of $\Omega$ is conformally
equivalent to the unit disk. Let $\pi: {{\mathbb{D}}}\to \Omega$ be a
universal covering map. The inverse of $\pi$ is multivalued, unless
$\Omega$ is simply connected, but the different branches of
$\pi^{-1}$ are related by M\"obius transformations. Hence a unique
projective structure on $\Omega$ is obtained by using local branches
of $\pi^{-1}$ as projective coordinates. The connection coefficients
$q(z)$ are related to the (branches of the) multivalued liftings
$\pi^{-1}$ by (\ref{eq:SchwarzianDE}), that is
$$
\{\pi^{-1}(z),z\}_2 = q(z).
$$
\end{ex}

\begin{ex} Let $\hat{\Omega}$ be the Schottky double of $\Omega$ and let
$\pi: U\to \hat{\Omega}$ be a universal covering map for
$\hat{\Omega}$. This $U$ can be taken to be either the Riemann
sphere (if $\Omega$ is simply connected), the complex plane (if
$\Omega$ is doubly connected) or the unit disk (if $\Omega$ has at
least three boundary components). In any case, $\hat{\Omega}$, and
hence $\Omega$, is provided with a unique projective structure by
using local inverses of the uniformization map as projective
coordinates. If $\Omega$ is doubly connected, it even gets an affine
structure.

In general, the projective structures in the above examples are all
different. For example, if $\Omega$ is the annulus
$A_{1,R}$, then by straight-forward computations one
finds that the projective connections are given by
$$
q(z)=\frac{a}{z^2},
$$
where $a=0$ in the cases 1) and 2) above, $a=\frac{1}{4}$ in case 3)
and $a=\frac{1}{4}(1+\frac{\pi^2}{(\log R)^2})$ in case 4). If
$\Omega$ is replaced by a noncircular doubly connected domain also
the cases 1) and 2) will be unequal.
\end{ex}

\begin{ex}\label{ex:taylor coefficients}
Connections also come up from Taylor coefficients of regular parts
of certain domain functions. For example, for the Taylor
coefficients of the regular part of the complex Green's function,
see Section~\ref{sec:taylor coefficients}, we have the following.

\end{ex}

\begin{prop}
In the notation of Section~\ref{sec:taylor coefficients}, plus
(\ref{eq:l}), the quantities
\begin{equation}\label{eq:greenconnection0}
 p(\zeta)=-c_0(\zeta),
\end{equation}
\begin{equation}\label{eq:greenconnection1}
r(\zeta)=-2c_1(\zeta)=2\frac{\partial p}{\partial z},
\end{equation}
\begin{equation}\label{eq:greenconnection} q(\zeta)
=-6(\frac{\partial c_1(\zeta)}{\partial \zeta}-2 c_2(\zeta))=6\pi
\ell(\zeta,\zeta)
\end{equation}
transform under conformal mapping as, respectively, the real part of
a $0$-connection, an affine connection and a projective connection.
\end{prop}

\begin{proof}
The Green's function transforms under conformal mappings $f:\Omega
\to\tilde{\Omega}$, $\tilde{z}=f(z)$, as
$G(z,\zeta)=\tilde{G}(\tilde{z},\tilde{\zeta})$, hence
$$
{\mathcal H}(z,\zeta)=-\log\frac{\tilde{z}-\tilde{\zeta}}{z-\zeta}
+\tilde{{\mathcal H}}(\tilde{z},\tilde{\zeta}).
$$
Now the assertions follow by easy computations, using the formulas
(\ref{eq:logdiff}):
$$
c_0(\zeta )={\mathcal H}(\zeta,\zeta)={H}(\zeta,\zeta),
$$
\begin{equation}\label{eq:c1fromc0}
c_1(\zeta )=\frac{\partial {\mathcal H}(z,\zeta)}{\partial
z}\big|_{z=\zeta} =2\frac{\partial { H}(z,\zeta)}{\partial
z}\big|_{z=\zeta} =\frac{\partial }{\partial \zeta}{ H}(\zeta,\zeta)
=\frac{\partial }{\partial \zeta}c_0(\zeta ),
\end{equation}
$$
\frac{\partial c_1(\zeta)}{\partial \zeta}-2c_2(\zeta
)=\frac{\partial^2 {\mathcal H}(z,\zeta)}{\partial
z\partial\zeta}\big|_{z=\zeta} =2\frac{\partial^2 {
H}(z,\zeta)}{\partial z\partial\zeta}\big|_{z=\zeta} =-\pi
\ell(\zeta,\zeta).
$$
\end{proof}

We conclude from (\ref{eq:greenconnection}) that $d\sigma=
e^{-c_0(\zeta)}|d\zeta|$ is a conformally invariant metric and that
$r(z)$ is the associated affine connection. However, $q(z)$ in
(\ref{eq:greenconnection}) is in general not the same as the
projective connection associated to $r(z)$ by the general receipt
(\ref{qr}), namely
$$
Q(\zeta)=\frac{\partial
r(\zeta)}{\partial\zeta}-\frac{1}{2}r(\zeta)^2 =-2 \frac{\partial
c_1(\zeta)}{\partial\zeta}-2c_1(\zeta)^2.
$$
In the multiply connected case $Q(\zeta)$ is not holomorphic,
whereas $q(\zeta)$ in (\ref{eq:greenconnection}) is always
holomorphic.

\begin{ex}
Coefficients of linear differential equation in $\Omega$ transform
in complicated ways under changes of coordinates, and in some cases
exactly as connections. For example, it is well-known that a second
order equation always can be written on the form
(\ref{eq:secondorderDE}), that is $u''+\frac{1}{2}Q(z)u=0$. Then
$Q(z)$ works as a projective connection, and the corresponding
projective structure uniformizes the equation. This differential
equation will be further discussed in
Section~\ref{sec:prepotential}.
\end{ex}

\begin{ex} Let us spell out the relevant quantities for the unit
disk and upper half-plane provided with the Poincar\'{e} metric:

\begin{itemize}

\item[(i)] Unit disk:
$$
\omega(z) = 1-|z|^2, \quad r(z)=\frac{2\bar{z}}{1-|z|^2}, \quad
q(z)=0, \quad \kappa_{\rm Gauss} =-4.
$$

\item[(ii)] Upper half-plane:
$$
\omega(z) = 2y, \quad r(z)=\frac{\I}{y}, \quad q(z)=0, \quad
\kappa_{\rm Gauss} =-4.
$$

\end{itemize}

Note that $r(z)$ is singular on the boundary, while $q(z)$ is not.
\end{ex}


\section{Connections on the Schottky double of a plane
domain}\label{sec:connections on double}

\subsection{Generalities}

If the Riemann surface is the Schottky double of a plane domain
$\Omega$, then the connection coefficients can be described in terms
of pairs of functions on $\Omega$, in analogy with the previous
description (\ref{eq:halforderdiff}) of differentials of any
half-integer order. Recall (\ref{eq:Schwarztransition}) that the
transition function between the coordinate $\tilde z=\bar z$ on the
back-side $\tilde \Omega$ and $z$ on the front side $\Omega$ is
given by the Schwarz function. Writing the unit tangent vector as
$T(z)=\frac{dz}{ds}=e^{\I\theta}$, the curvature of $\partial\Omega$
is $\kappa=\frac{d\theta}{ds}=-\I T'(z)$, where the prime denotes
the complex derivative $\frac{\partial}{\partial z}$ for the
analytically extended $T(z)$. Then, along $\partial\Omega$,
$$
\{S(z),z\}_0=-2\log T(z)=-2\I\theta,
$$
$$
\{S(z),z\}_1 dz=-2\frac{T'(z)}{T(z)}dz = -2\I d\theta=-2\I\kappa ds,
$$
$$
\{S(z),z\}_2 dz^2=-2\frac{T''(z)}{T(z)}dz^2 =
-2\I\frac{d\kappa}{ds}ds^2.
$$

Therefore, a $0$-connection on $\Hat\Omega$ can be described as a
pair of functions, $p_1(z)$ (for the front side) and $p_2(z)$
(for the backside), on $\Omega$,
satisfying on
$\partial\Omega$ the matching condition
$$
p_1(z)=\overline{p_2(z)}-2\I \theta.
$$
Similarly, an affine connection is represented by a pair $r_1(z)$,
$r_2(z)$ satisfying the matching condition
$r_1(z)dz=\overline{r_2(z)dz}-2\I \kappa ds$, or
\begin{equation}\label{eq:affconnectdouble}
r_1(z)T(z)=\overline{r_2(z)T(z)}-2\I \kappa \quad {\rm on\,\,}
\partial\Omega,
\end{equation}
and a projective connection by a pair $q_1(z)$, $q_2(z)$ with
$$
q_1(z)T(z)^2=\overline{q_2(z)T(z)^2}-2\I \frac{d\kappa}{ds} \quad
{\rm on\,\,}
\partial\Omega.
$$
Note that if $z$ and $\tilde{z}$ are both projective coordinates
then $q_1=q_2=0$ and it follows that $\frac{d\kappa}{ds}=0$ on
$\partial\Omega$, i.e., that $\Omega$ is a circular domain.

We may notice that (\ref{eq:affconnectdouble}) is consistent with
the general fact \cite{Schiffer-Hawley} that the sum of residues of
a meromorphic affine connection on a compact Riemann surface of
genus ${\texttt{g}}$ equals $2({\texttt{g}}-1)$ (our definition of
connection differs from that in \cite{Schiffer-Hawley} by a minus
sign). In fact, in the case of the Schottky double the sum is, by
(\ref{eq:affconnectdouble}),
$$
\frac{1}{2\pi \I}\int_{\partial \Omega} r_1 (z)dz + \frac{1}{2\pi
\I}\int_{-\partial \Omega} \overline{{r}_2 (z)}d\bar{z} =
-\frac{1}{\pi}\int_{\partial\Omega}{\kappa}ds=2({\texttt{g}}-1)
$$
since there is one outer component and ${\texttt{g}}$ inner ones.

Naturally, one may be particularly interested in symmetric, or
``real", connections, namely those for which $p_1=p_2$, $r_1=r_2$ or
$q_1=q_2$ (respectively). The above matching conditions then become,
for the single representatives $p(z)$, $r(z)$, $q(z)$ on $\Omega$,
$$
\Im p(z)=- \theta \quad {\rm on\,\,}
\partial\Omega,
$$
\begin{equation}\label{eq:symmaffconnectdouble}
\Im (r(z)T(z))=-\kappa \quad {\rm on\,\,}
\partial\Omega,
\end{equation}
$$
\Im (q(z)T(z)^2)=- \frac{d\kappa}{ds} \quad {\rm on\,\,}
\partial\Omega.
$$


\subsection{Formulas for the curvature of a curve}

The differential parameter $\{z,w\}_1$ turns out to be a useful tool
for summarizing various formulas for the curvature of a curve in the
complex plane.

The definition of  $\{z,w\}_1$, with $z=f(w)$ analytic, can be
written in differential form as
$$
\{z,w\}_1 dw= d\log \frac{dz}{dw}=d(\log dz-\log dw).
$$
Hence
$$
\Im (\{z,w\}_1 dw)=d\arg dz-d\arg dw={}^*d(\log |dz|-\log |dw|).
$$
To interpret this formula one should let $w$ run along a curve
$\Gamma$, say with parametrization $t\mapsto w(t)$,
$t\in{\mathbb{R}}$. Then $z(t)=f(w(t))$ runs along $f(\Gamma)$ and one
just replaces $dz$ and $dw$  by $z'(t)$ and $w'(t)$, respectively.

Denoting the curvature of $\Gamma$ by $\kappa_w$ and the curvature
of $f(\Gamma)$ by $\kappa_z$ we thus have the following formula.

\begin{prop}
Under a conformal mapping $f:z\mapsto w$, the curvature of a curve
$\Gamma$ and its image curve $f(\Gamma)$ are related by
\begin{equation}\label{eq:kzkw}
\Im (\{z,w\}_1 dw)=\kappa_z ds_z-\kappa_w ds_w.
\end{equation}
Here $ds_z=|dz|$, $ds_w=|dw|$.
\end{prop}

\begin{ex} (The curvature in terms of a real parameter.)
With $\Gamma={\mathbb{R}}$ and $z=f(w)$, so that $z=f(t)$ with $t=\Re
w$ parametrizes $f(\Gamma)$, we have $\kappa_w=0$, which gives the
well-known formula
$$
\kappa_z= \Im(\frac{d}{dw}(\log \frac{dz}{dw}) \frac{dw}{ds_z})
+\kappa_w\frac{ds_w}{ds_z}
$$
$$
=\Im(\frac{f''(w)}{f'(w)} \frac{dw}{|f'(w)||dw|})
=\frac{1}{|f'(t)|}\Im\frac{f''(t)}{f'(t)}.
$$
\end{ex}

\begin{ex} (The curvature in terms of an angular parameter; essentially Study's formula
\cite{Polya}, p.125.) With $\Gamma=\partial{\mathbb{D}}$, $w=e^{\I
t}$, $dw=\I wd t$ we have,
$$
\kappa_z=\Im (\{z,w\}_1 \frac{dw}{ds_z})+\kappa_w\frac{ds_w}{ds_z}
=\Im(\frac{f''(w)}{f'(w)} \frac{\I w dt}{|f'(w)||dt|})
+\frac{1}{|f'(w)|}
$$
$$
=\frac{1}{|f'(w)|}(\Re\frac{wf''(w)}{f'(w)}+1).
$$
\end{ex}

\begin{ex} (The curvature of a curve given as a
level line of a harmonic function.) Let $\Gamma$ be any level line
of a harmonic function $u$ in the $z$-plane. Assume that the curve
is nonsingular, so that $\nabla u\ne 0$ on $\Gamma$. With $v$ any
harmonic conjugate of $u$, $w=u+iv=g(z)$ maps $\Gamma$ into a
vertical line in the $w$-plane. Thus $\kappa_w=0$ and the curvature
of $\Gamma$ is obtained from
$$
\kappa_z ds_z =\Im (\{z,w\}_1 dw)+ \kappa_w ds_w =-\Im (\{w,z\}_1
dz)
$$
$$
=\Im (\frac{\partial}{\partial z}(\log \frac{\partial w}{\partial
z}) dz)
=\Im (2\frac{\partial}{\partial z}\log (|2\frac{\partial u}{\partial
z}|) dz)
$$
$$
=-\frac{\partial}{\partial x}\log |\nabla u| dy +
\frac{\partial}{\partial y}\log |\nabla u| dx
=-\frac{\partial}{\partial n}\log |\nabla u| ds_z.
$$
We may may also extract from the above the formula (cf.
\cite{Jerrard})
$$
\kappa_z=|g'(z)|\Re \frac{g''(z)}{g'(z)^2}.
$$

In summary, for a level curve $\Gamma$ of a harmonic function $u$
the curvature is given by
\begin{equation}\label{eq:kz}
\kappa_z =-\frac{\partial}{\partial n}\log |\nabla u| =-\frac{d}{d
n}\log |\frac{du}{dn}| =-{\frac{d^2u}{dn^2}}/{\frac{du}{dn}}.
\end{equation}
In the last expressions we used $\frac{d}{d n}$ in place of
$\frac{\partial}{\partial n}$ because the latter becomes obscure
when applied twice ($\frac{\partial}{\partial n}$ is not really a
partial derivative since $n$ is not a coordinate of a coordinate
system; $\frac{d}{d n}$ should be interpreted as a directional
derivative along the straight line in the normal direction).

In higher dimensions, (\ref{eq:kz}) gives the mean curvature of a
hypersurface given as the level surface of a harmonic function. This
can easily be deduced from the more general (and well-known, cf.
\cite {Frankel}) formula $\kappa_{\rm mean}={\rm div\,}\frac{\nabla
u}{|\nabla u|}$, for the mean curvature of a level surface of any
smooth function $u$.
\end{ex}


\subsection{Geodesics}

The equation for geodesic curves $z=z(t)$ for the metric
(\ref{eq:hermitean}) is
\begin{equation}\label{eq:geodesic}
\frac{d^2 z}{dt^2} + r(z) (\frac{dz}{dt})^2=0,
\end{equation}
where $r(z)$ is the corresponding affine connection
(\ref{eq:rfrommetric}) and the parameter $t$ measures arc-length
with respect to the metric. This equation is the usual geodesic
equation in differential geometry \cite{Frankel}, just written in
complex analytic language. The classical Christoffel symbols
$\Gamma^i_{jk}$ ($i,j,k=1,2$) turn out to coincide with the
components $\pm\Re r(z), \pm\Im r(z)$. An intuitive direct
derivation of (\ref{eq:geodesic}) goes as follows.

The tangent vector along the curve is $\frac{dz}{dt}$ and the
geodesic equation is supposed to say that this propagates by
parallel transport, i.e., has covariant derivative zero along the
curve. Since a vector in one complex variable can be thought of as a
differential of order minus one the relevant covariant derivative
(\ref{eq:affderivative}) will be
$\nabla_{-1}=\frac{\partial}{\partial z}+r(z)$. This is the
covariant version of $\frac{\partial}{\partial z}$, and the
covariant version of $\frac{d}{d t}$ then is
$\frac{\partial}{\partial t}+r(z)\frac{dz}{d t}$. Applying this to
$\frac{dz}{dt}$ gives (\ref{eq:geodesic}).

It is convenient to write (\ref{eq:geodesic}) on the form
\begin{equation}\label{eq:geodesiclog}
\frac{d}{dt}\log \frac{dz}{dt} + r(z) \frac{dz}{dt}=0
\end{equation}
and to decompose it into real and imaginary parts. The real part
just contains internal information about how the curve is
parametrized, namely saying that $t$ measures arc-length with
respect to the metric. The information about the geometry of the
curve (that it is a geodesic) is entirely contained in the imaginary
part,
\begin{equation}\label{eq:geodesicarg}
\frac{d}{dt}\arg \frac{dz}{dt} + \Im (r(z) \frac{dz}{dt})=0.
\end{equation}

In equation (\ref{eq:geodesicarg}), $t$ can be taken to be any
parameter, for example euclidean arc-length $s$. Then the first term
equals the euclidean curvature of the geodesic and $\frac{dz}{ds}$
is the unit tangent vector. Thus we have
\begin{prop}
The curvature $\kappa=\kappa(z, \frac{dz}{ds})$ of the geodesic
passing through a point $z$ and having direction $\frac{dz}{ds}$ (a
unit vector) is given by
\begin{equation}\label{eq:curvature1}
\kappa (z, \frac{dz}{ds}) =-\Im(r(z)
\frac{dz}{ds})=-\Im(2\frac{\partial p}{\partial{z}} \frac{dz}{ds})
=- \frac{\partial p}{\partial n},
\end{equation}
where $\frac{\partial p}{\partial n}$ denotes the derivative in the
rightward normal direction of the curve. In particular, the sharp
bound
\begin{equation}\label{eq:curvature2}
|\kappa (z,\frac{dz}{ds})|\leq  |r(z)|
\end{equation}
holds for all geodesics passing through $z$, and equality is
attained if and only $\frac{dz}{ds}$ is tangent to the level line of
$p$ at $z$.
\end{prop}

The above estimate (\ref{eq:curvature2}) can be combined with
geometric estimates of $r(z)$. For example, for the coefficients of
the Green's function, as in Section~\ref{sec:Taylor} and
Example~\ref{ex:taylor coefficients}, we have by
(\ref{eq:greenconnection0}), (\ref{eq:greenconnection1}) that
$r(z)=-2c_1(z)$ is the affine connection for the metric
$d\sigma=e^{-c_0}|dz|$. For $c_1(z)$ we have the estimate in
Lemma~\ref{lem:estcn}, so the above proposition shows that
$$
|r(z)|\leq\frac{2}{d(z,\partial\Omega)}
$$
for this connection. When $\Omega$ is simply connected the metric in
question is the Poincar\'{e} metric. Combining with
(\ref{eq:curvature2}) we therefore have the following.

\begin{cor}
For a simply connected domain $\Omega$ provided with its
Poincar\'{e} metric, the curvature for any geodesics through a point
$z$ is subject to the estimate
$$
|\kappa (z,\frac{dz}{ds})|\leq\frac{2}{d(z,\partial\Omega)}.
$$
\end{cor}

It is allowed here that $\Omega\subset{\mathbb{P}}$ contains the
point of infinity, and at least in this generality the corollary is
sharp. Indeed, with $\Omega
={\mathbb{P}}\setminus\overline{{\mathbb{D}}(0,\epsilon)}$,
$\epsilon>0$ small, the circle with center at any finite point
$c\in\Omega$, and having radius $|c|$, is almost a geodesic for the
Poincar\'{e} metric in $\Omega$. The curvature for this circle is
$1/|c|$, hence the bound in the corollary is essentially attained at
the point $2c$ on the circle.


\subsection{Some applications in physics}

Projective and affine connections come up naturally in both
classical and modern physics. As to modern physics, e.g., conformal
field theory and string theory, we have already mentioned the
example with the energy-momentum tensor as a projective connection.
For further examples, see \cite{Dubrovin}, \cite{Takhtajan-Teo} and
references therein.

For classical physics, we shall briefly mention one example from
vortex dynamics. Consider in $\Omega$ an incompressible inviscid
fluid which is irrotational except for a point vortex of unit
strength at a point $z_0$. If $\Omega$ is simply connected this
makes the flow uniquely determined. The stream function $\psi$ will,
at each instant of time,  coincide with the Green's function,
$\psi(z)=G(z,z_0)$. However, the flow will not be stationary because
the vortex will move with the speed obtained by subtracting off,
from the general flow, the rotationally symmetric singular part
corresponding to $-\log|z-z_0|$ in the stream function. Thus in fact
$\psi=\psi(z,t)=G(z,z_0(t))$ and one deduces, in the notation of
Example~\ref{ex:taylor coefficients}, that the vortex moves along a
level line of $c_0(z)$ and that the velocity is given by
\begin{equation}\label{eq:vortex}
\frac{dz_0(t)}{dt}=-\I\overline{{c}_1(z_0(t))}=-\I\frac{\partial
c_0}{\partial \bar{z}}(z_0(t)).
\end{equation}
This means that $c_0(z)$ is a kind of stream function for the vortex
motion, called the {\it Routh stream function}, \cite{Lin},
\cite{Newton}. It also follows that the vortex motion is a
hamiltonian motion, determined by the symplectic form $dx\wedge dy$
(area $2$-form) and hamiltonian function $H=\frac{1}{2}c_0$. Note
that (\ref{eq:vortex}) is a special case of (\ref{eq:hamiltonian}).

If $\Omega$ is multiply connected one need to prescribe the
circulations $\gamma_j$ around the holes to make the flow uniquely
determined. These circulations will be preserved in time (Kelvin's
theorem) and everything will be as in the simply connected case but
with the ordinary Green's function replaced by the hydrodynamic
Green's function $G_{\gamma}(z,z_0)$, to be discussed in
Section~\ref{sec:Neumann}.

A related physical application comes from electrostatics. Think of
the complement  $K={\mathbb{C}}\setminus\Omega$ as a perfect conductor
and let it be grounded, so that its potential is zero. Then consider
a unit charge located at $z_0\in\Omega$. This will induce charges of
the opposite sign on $\partial K$, namely distributed so that the
density with respect to arc-length is given by the normal derivative
$\frac{\partial G(\cdot,z_0)}{\partial n_z}$. These charges will
exert a force on the charge at $z_0$, and this force is
$$
F(z_0)=-\frac{1}{2\pi}\overline{{c}_1(z_0(t))}.
$$
If $\Omega$ is multiply connected and the components of $K$ are not
grounded to a common zero, but are isolated from each other, then
the hydrodynamic Green function shall be used in place of the
ordinary one, with the $\gamma_j$ proportional to the respective
total charges isolated on the different components of $K$.


\section{On Neumann functions and the hydrodynamic Green's
function}\label{sec:Neumann}

\subsection{Definitions}

Recall that the Bergman kernel and the Schiffer kernel are given in
terms of the ordinary Green's function by (\ref{eq:K}),
(\ref{eq:L}). Decomposing the Green's function as in (\ref{eq:GH})
we also have, for the ``$\ell$-kernel'' (\ref{eq:l}),
$$
\ell(z,\zeta)=\frac{2}{\pi}\frac{\partial^2 H(z,\zeta)}{\partial
z\partial{\zeta}}.
$$
The reduced Bergman kernel (see
Proposition~\ref{prop:reducedBergman}) and its adjoint have several
similar representations in terms of Neumann type functions and
hydrodynamic Green's functions. Here we shall briefly review these
matters. See \cite{Sario-Oikawa}, \cite{Schiffer-Hawley} for more
details.

By a {\it Neumann function} we mean a domain function $N_a(z,\zeta)$
with a logarithmic singularity at a given point $\zeta\in\Omega$ and
satisfying Neumann boundary data given by a boundary function $a$
subject to
\begin{equation}\label{eq:neumann0}
\int_{\partial\Omega} ads =2\pi.
\end{equation}
The requirements on $N_a(z,\zeta)$ are, more exactly,
\begin{equation}\label{eq:neumann1}
N_a(z,\zeta)= -\log|z-\zeta|+ {\rm harmonic}
\quad {\rm in\,\,}\Omega,
\end{equation}
\begin{equation}\label{eq:neumann2}
-\frac{\partial N_a(\cdot, \zeta)}{\partial n}= a \quad {\rm
on\,\,}\partial\Omega,
\end{equation}
\begin{equation}\label{eq:neumann3}
\int_{\partial\Omega} N_a(\cdot,\zeta)ads
=0.
\end{equation}
The final condition (\ref{eq:neumann3}) can also be written
$$
\int_{\partial\Omega} N_a(\cdot,\zeta)^*dN_a(\cdot,z)=0,
$$
and is a normalization which guarantees that $N_a$ is symmetric:
$$
N_a(z,\zeta)=N_a(\zeta,z).
$$

A {\it hydrodynamic Green's function} (or ``modified Green's
function'') \cite{Lin}, \cite{Flucher}, \cite{Crowdy2} is defined in
terms of ${\texttt{g}}+1$ prescribed circulations,
$\gamma_0,\dots\gamma_{{\texttt{g}}}$ subject to the consistency
condition
\begin{equation}\label{eq:hydrogreen0}
\gamma_0+\dots+\gamma_{{\texttt{g}}}=2\pi.
\end{equation}
Let $\gamma=(\gamma_0,\dots\gamma_{{\texttt{g}}})$ denote the entire
vector of periods. The defining properties of the hydrodynamic Green
function, denoted $G_{\gamma}(z,\zeta)$, are
\begin{equation}\label{eq:hydrogreen1}
G_{\gamma}(z,\zeta)= -\log|z-\zeta|+ {\rm harmonic} \quad {\rm
in\,\,}\Omega,
\end{equation}
\begin{equation}\label{eq:hydrogreen2}
G_{\gamma}(z,\zeta)=b_j(\zeta)\quad {\rm for\,\,}z\in\Gamma_j,
\end{equation}
\begin{equation}\label{eq:hydrogreen3}
-\int_{\Gamma_j}\frac{\partial G_{\gamma}(z,\zeta)}{\partial
n}\,ds(z)= \gamma_j \quad (j=0,\dots,{\texttt{g}}),
\end{equation}
\begin{equation}\label{eq:hydrogreen4}
\sum_{j=0}^{{\texttt{g}}}\gamma_j\beta_j(\zeta)=0.
\end{equation}

Here $b_j(\zeta)$ denote ``floating constants'' (they cannot be
preassigned), i.e., (\ref{eq:hydrogreen2}) really means
$$
dG_{\gamma}(\cdot,\zeta)=0 \quad {\rm along\,\,}\partial\Omega.
$$
Condition (\ref{eq:hydrogreen4}) is a normalization which can be
written
$$
\int_{\partial\Omega}G_{\gamma}(\cdot,\zeta)\,
^*dG_{\gamma}(\cdot,z)=0
$$
and which guarantees the symmetry,
$$
G_{\gamma}(z,\zeta)=G_{\gamma}(\zeta,z).
$$

The hydrodynamic Green's function can be constructed from the
ordinary Green's function by
$$
G_{\gamma}(z,\zeta)=G(z,\zeta)+\sum_{k,j=0}^{{\texttt{g}}}
c_{kj}u_k(z) u_j(\zeta),
$$
where $u_k$, $k=0,\dots , {\texttt{g}}$ are the harmonic measures.
With the above `Ansatz',  the requirements (\ref{eq:hydrogreen1})
and (\ref{eq:hydrogreen2}) are automatically satisfied, and
(\ref{eq:hydrogreen3}), (\ref{eq:hydrogreen4}) give a system of
equations which determine the coefficients $c_{kj}$ uniquely. The
matrix $(c_{kj})$ is symmetric and positive semidefinite.


\subsection{Reproducing kernels for Dirichlet and Bergman spaces}

The Neumann and (hydrodynamic) Green's functions have logarithmic
singularities, hence cannot themselves be reproducing kernels for
any Hilbert spaces of harmonic functions. However, the singularities
disappear when subtracting them, and also after application of the
the differential operator $\frac{\partial^2 }{\partial
z\partial{\bar{\zeta}}}$. In these cases we do obtain reproducing
kernels for important spaces. Below we elaborate on these matters,
slightly extending the analysis in \cite{Bergman} and
\cite{Schiffer-Hawley}.

Let $D(u,v)$ denote the Dirichlet inner product:
$$
D(u,v)=\int_\Omega \nabla u\cdot \nabla v \,dxdy=\int_\Omega
du\wedge^*dv,
$$
let $a$ be boundary data as above, satisfying (\ref{eq:neumann0}),
and define the period vector $\gamma$ by
\begin{equation}\label{eq:gammaa}
\gamma_k = \int_{\Gamma_k} ads.
\end{equation}
Then (\ref{eq:hydrogreen0}) holds, by (\ref{eq:neumann0}). Define
$H(\Omega)$ to be the Hilbert space of all harmonic functions $u$ in
$\Omega$ which satisfy $D(u,u)<\infty$ and
\begin{equation}\label{eq:harm1}
\int_{\partial\Omega} u \,ads =0,
\end{equation}
and let $H_e(\Omega)$ be the subspace consisting of those functions
which in addition satisfy
\begin{equation}\label{eq:harm2}
\int_{\Gamma_j}  \frac{\partial u}{\partial n}ds =0\quad
(j=0,\dots,{\texttt{g}}).
\end{equation}
Note that (\ref{eq:harm1}) just fixes the additive constant in $u$
which the inner product leaves free. Except for this additive
constant $H(\Omega)$ does not depend on $a$. The following is a
slight extension of results in Ch.V:3 of \cite{Bergman}.

\begin{prop}
The reproducing kernels for $H(\Omega)$ and $H_e(\Omega)$ are,
respectively,
$$
k(z,\zeta)=\frac{1}{2\pi}(N_{a}(z,\zeta)-G(z,\zeta)),
$$
$$
k_e(z,\zeta)=\frac{1}{2\pi}(N_{a}(z,\zeta)-G_{\gamma}(z,\zeta)).
$$
In other words, $k(\cdot,\zeta)\in H(\Omega)$ and
\begin{equation}\label{eq:ukzeta}
u(\zeta)= D(u,k(\cdot, \zeta)) \quad (u\in H(\Omega)),
\end{equation}
and similarly for $k_e(\cdot,\zeta)$.
\end{prop}

\begin{proof}
All verifications are straightforward. Let us just show, for
example, that (\ref{eq:ukzeta}) holds for $k_e(\cdot,\zeta)$ and
$u\in H_e(\Omega))$. We may assume that $u$ is smooth up to the
boundary. The proof amounts to standard applications of Green's
formula to functions with a singularity. For the Neumann function we
get
$$
D(u, N_a(\cdot, \zeta))= \int_\Omega du\wedge ^*dN_a(\cdot, \zeta)
$$
$$
=\lim_{\epsilon\to 0} \int_{\Omega\setminus
{{\mathbb{D}}}(\zeta,\epsilon)} d(u\wedge ^*dN_a(\cdot, \zeta))
=-\int_{\partial\Omega} u\, ads + 2\pi u(\zeta)=2\pi u(\zeta),
$$
and for the hydrodynamic Green's function,
$$
D(u, G_{\gamma}(\cdot, \zeta))=D(G_{\gamma}(\cdot, \zeta),u)
=\lim_{\epsilon\to 0} \int_{\Omega\setminus
{{\mathbb{D}}}(\zeta,\epsilon)} d(G_{\gamma}(\cdot,
\zeta)\wedge\,^*du)
$$
$$
=\sum_{j=0}^{{\texttt{g}}}b_j(\zeta)\int_{\Gamma_j}\,\frac{\partial
u}{\partial n}ds -\lim_{\epsilon\to
0}\int_{{\partial{\mathbb{D}}}(\zeta,\epsilon)}G_{\gamma}(\cdot,
\zeta)\frac{\partial u}{\partial n}ds=0.
$$
Now (\ref{eq:ukzeta}) follows.
\end{proof}

By definition of $H_e(\Omega)$, all functions $u\in H_e(\Omega)$
have single-valued harmonic conjugates. Thus we can form the space
of analytic functions $f=u+\I u^*$ with $u\in H_e(\Omega)$ and with
also $u^*$ normalized by (\ref{eq:harm1}). This is simply the
ordinary {\it Dirichlet space} with normalization (\ref{eq:harm1}),
i.e., the Hilbert space $\mathcal{A}(\Omega)$ of analytic functions
in $\Omega$ provided with the hermitean inner product
$$
(f,g)_{-1}=\int_\Omega f' \bar{g'} \,dxdy =-\frac{1}{2\I}
\int_\Omega df\wedge d\bar{g},
$$
and subject to the normalization
\begin{equation}\label{eq:anal1}
\int_{\partial\Omega} f \,ads =0.
\end{equation}
The reason for indexing the inner product by $-1$ is that
$\mathcal{A}(\Omega)$ shortly will be identified as one space in a
sequence of weighted Bergman spaces with inner products in general
denoted $(f,g)_{m}$, $m\in{{\mathbb{Z}}}$. We note that, with $u=\Re
f$, $v=\Re g$,
$$
D(u,v)=\int_\Omega du\wedge \,^*dv =-\Re\frac{1}{2\I} \int_\Omega
df\wedge d\bar{g}= \Re(f,g)_{-1}.
$$

The reproducing kernel for $\mathcal{A}(\Omega)$ is the analytic
completion ${\mathcal{K}} (z,\zeta)$ of $k_e(z,\zeta)$, normalized
by (\ref{eq:anal1}). In terms of the multivalued analytic
completions ${\mathcal{N}}_{a}(z,\zeta)$,
${{\mathcal{G}}}_{\gamma}(z,\zeta)$ of ${N}_{a}(z,\zeta)$ and
${G}_{\gamma}(z,\zeta)$ we therefore have
$$
{\mathcal{K}} (z,\zeta)=\frac{1}{2\pi}({\mathcal{N}}_{a}(z,\zeta)
-{{\mathcal{G}}}_{\gamma}(z,\zeta)).
$$
An  adjoint kernel may be introduced as
$$
{\mathcal{L}}(z,\zeta)=\frac{1}{2\pi}({\mathcal{N}}_{a}(z,\zeta)
+{{\mathcal{G}}}_{\gamma}(z,\zeta)).
$$

\begin{rem}
${\mathcal{N}}_{a}(z,\zeta)$, ${{\mathcal{G}}}_{\gamma}(z,\zeta)$ are
not symmetric with respect to $z$, $\zeta$ and are not analytic or
antianalytic with respect to $\zeta$. However,
${\mathcal{K}}(z,\zeta)$ is (of course) hermitean symmetric and is
antianalytic in $\zeta$. The adjoint kernel,
${\mathcal{L}}(z,\zeta)$ is multivalued analytic in both $z$ and
$\zeta$. See \cite{Schiffer-Hawley} for more details.
\end{rem}

\begin{ex}\label{ex:kernelsunitdisk}
For the unit disk, $\Omega={\mathbb{D}}$, with $a=1$ and (necessarily)
$\gamma=(\gamma_0)=(2\pi)$  we have
$$
N_a(z,\zeta)=-\log|z-\zeta| -\log|1-z\bar{\zeta}|,
$$
$$
G_{\gamma}(z,\zeta)=-\log|z-\zeta| +\log|1-z\bar{\zeta}|,
$$
$$
k(z,\zeta) =-\frac{1}{\pi}\log|1-z\bar{\zeta}|,
$$
$$
{\mathcal{K}}(z,\zeta) =-\frac{1}{\pi}\log(1-z\bar{\zeta}),
$$
$$
{\mathcal{L}}(z,\zeta) =-\frac{1}{\pi}\log(z-{\zeta})
$$
$$
{\mathcal{N}}_a(z,\zeta)=-\log(z-\zeta) -\log(1-z\bar{\zeta}),
$$
$$
{{\mathcal{G}}}_{\gamma}(z,\zeta)=-\log(z-\zeta)
+\log(1-z\bar{\zeta}),
$$
Of course, $G_{\gamma}(z,\zeta)=G(z,\zeta)$,
${{\mathcal{G}}}_{\gamma}(z,\zeta)={{\mathcal{G}}}(z,\zeta)$ in the
simply connected case.
\end{ex}

We shall denote the ordinary Bergman space by $B(\Omega)$. It is the
the Hilbert space of analytic functions in $\Omega$ with
$(f,f)_1<\infty$, where
$$
(f,g)_1=\int_\Omega f\bar{g} dxdy
$$
is the inner product. The exact Bergman space, $B_e(\Omega)$, is the
subspace consisting of those functions of the form $F'$, where $F$
is (single-valued) analytic in $\Omega$. The Dirichlet space is
related to the exact Bergman space by differentiation: $\Lambda_1:
f\mapsto f'$ (or $f\mapsto df$) is an isometric isomorphism
$\Lambda_1:{\mathcal{A}}(\Omega)\to B_e(\Omega)$.

\begin{prop}\label{prop:reducedBergman}
The reduced Bergman kernel, i.e., the reproducing kernel for
$B_e(\Omega)$, is given by
$$
K_e(z,\zeta)= \frac{2}{\pi}\frac{\partial^2 N_a(z,\zeta)}{\partial
z\partial\bar{\zeta}} =-\frac{2}{\pi}\frac{\partial^2
G_{\gamma}(z,\zeta)}{\partial z\partial\bar{\zeta}}
=\frac{\partial^2 {\mathcal{K}}(z,\zeta)}{\partial
z\partial\bar{\zeta}}
$$
for any choices of $a$ and $\gamma$ as above. Similarly, the
corresponding adjoint kernel is
$$
L_e(z,\zeta)= \frac{2}{\pi}\frac{\partial^2 N_a(z,\zeta)}{\partial
z\partial{\zeta}} =-\frac{2}{\pi}\frac{\partial^2
G_{\gamma}(z,\zeta)}{\partial z\partial{\zeta}} =\frac{\partial^2
{\mathcal{L}}(z,\zeta)}{\partial z\partial{\zeta}}.
$$
\end{prop}

\begin{proof}
The proof is essentially well-known and straight-forward. Cf.
\cite{Bergman}, \cite{Schiffer-Spencer}.
\end{proof}

\begin{rem}\label{rem:span}
A beautiful argument, due to M.~Schiffer \cite{Schiffer2}, shows
that the function $\displaystyle F(z)=\int^z L_e(t,\zeta)dt$ is
univalent and maps $\Omega$ onto a domain $D\subset{\mathbb{P}}$ with
$\infty\in D$ and such that each component of ${\mathbb{P}}\setminus
D$ is convex. See \cite{Sario-Oikawa} for
further information and several related issues.

A particular consequence is that $L_e(z,\zeta)$ has no zeros in
$\Omega$, hence that $K_e(z,\zeta)$ has $2{\texttt{g}}$ zeros,
$m_1(\zeta),\dots,m_{2{\texttt{g}}}(\zeta)$, in $\Omega$. In terms of
these the following generalization \cite{Schiffer-Hawley},
\cite{Hejhal} of (\ref{curvature7}) to the multiply connected case
holds
\begin{equation}\label{curvature11}
\frac{\partial^2 }{\partial z \partial \zeta}\log K_e(z,\zeta) =
2\pi K_e(z,\zeta)+ \pi \sum_{k=1}^{2{\texttt{g}}} L_e\left(z,
\overline{m_k(\zeta)}\right)\overline{m_k'(\zeta)}.
\end{equation}
For the full Bergman kernel $K(z,\zeta)$ there are similar results
\cite{Hejhal2}, but they are slightly more complicated because the
distribution of the $2{\texttt{g}}$ zeros between $K(z,\zeta)$ and
$L(z,\zeta)$ is less clear in this case. For example,
Theorem~\ref{thm:SF} shows that the number of zeros of $K(z,\zeta)$
in general depends on the location of $\zeta\in\Omega$.

\end{rem}


\subsection{Behavior of Neumann function under conformal mapping}

To describe how an object (e.g., the Green's function) on a domain
$\Omega$ transforms under conformal mapping is equivalent to telling
what kind of object it is (function, differential, connection etc.)
when the domain is considered as a Riemann surface. In fact, a
conformal map onto another domain can be considered simply as a
change of holomorphic coordinate on the Riemann surface. Many
objects associated to a domain actually extend to the Schottky
double, and since this Riemann surface is described by an atlas with
two charts overlapping on the boundary of the domain, complete
information of the conformal behavior of the object then is
contained in the transition formula on the boundary.

The ordinary Green's function $G(z,\zeta)$ extends directly to the
Schottky double as an odd function in each variable, see
(\ref{eq:GVJ}), and its differential is the real part of normalized
abelian differentials of the third kind on the double. Precisely, by
(\ref{eq:GVJ}), (\ref{eq:upsilonV}), Lemma~\ref{lem:omegaupsilon},
\begin{equation}\label{eq:dGupsilon}
dG(\cdot,\zeta)=-\Re\upsilon_{\zeta-J(\zeta)}=-\Re\omega_{\zeta-J(\zeta)}.
\end{equation}
When trying to extend the hydrodynamic Green function as an odd
function on the double, there appear jumps (=$2b_j$) on
$\partial\Omega$. However, the differential $dG_\gamma(\cdot,
\zeta)$ extends perfectly well, along with the conjugate
differential $^*dG_\gamma(\cdot, \zeta)$. Therefore $dG_\gamma
(\cdot,\zeta)+\I\,^*dG_\gamma (\cdot,\zeta)$ is an abelian
differential of the third kind with poles at $\zeta$ and $J(\zeta)$.
Computing the periods gives
$$
\int_{\alpha_j} \,(dG_\gamma (z,\zeta)+\I\,^*dG_\gamma (z,\zeta))
=2(b_j(\zeta)-b_0(\zeta)),
$$
$$
\int_{\beta_j}\,(dG_\gamma (z,\zeta)+\I\,^*dG_\gamma (z,\zeta))=
\I\gamma_j,
$$
$j=1,\dots, {\texttt{g}}$.

If we in particular choose the period vector to be
$$
\gamma=(2\pi, 0,\dots,0)
$$
we see that the $\beta_j$-periods ($j=1,\dots,{\texttt{g}}$) of
$dG_\gamma (\cdot,\zeta)+\I\,^*dG_\gamma (\cdot,\zeta)$ vanish.
Therefore (in this case)
$$
dG_\gamma (\cdot,\zeta)+\I\,^*dG_\gamma (\cdot,\zeta) =
-\tilde{\omega}_{\zeta-J(\zeta)},
$$
where $\tilde{\omega}_{a-b}$ in general denotes the abelian
differential of the third kind with the same singularities as
${\omega}_{a-b}$ (and ${\upsilon}_{a-b}$), but normalized so that
the $\beta_j$-periods vanish. Thus, in analogy with
(\ref{eq:dGupsilon}) we have
$$
dG_\gamma(\cdot,\zeta)=-\Re\tilde{\omega}_{\zeta-J(\zeta)}
$$
for $\gamma$ as above.

For the Neumann function the situation is slightly more complicated
than for the Green's functions. We shall consider $N_a(z,\zeta)$ and
${\mathcal{N}}_a(z,\zeta)=N_a(z,\zeta)+\I N_a^* (z,\zeta)$ together
with the meromorphic differential
$$
\Gamma_a (z,\zeta)dz=2\frac{\partial N_a(z,\zeta)}{\partial z}dz=
\frac{\partial {\mathcal{N}}_a(z,\zeta)}{\partial z}dz.
$$
Along the boundary (with respect to $z$),
$$
\Im (\Gamma_a (z,\zeta)dz)= \Im ( d{\mathcal{N}}_a (z,\zeta)dz)
=^*dN_a(z,\zeta) =-ads.
$$

In order to discuss questions of conformal invariance the boundary
function $a$ has to be linked to the domain $\Omega$. The most naive
choice, $a={\rm constant}$, turns out not to give any good behavior
under conformal mapping.  More promising is to relate $a$ to the
curvature $\kappa$ of the boundary. If we simply take $a=\kappa$
then (\ref{eq:symmaffconnectdouble}) shows that $\Gamma_a
(z,\zeta)dz$ extends to the Schottky double as a symmetric
meromorphic affine connection. Moreover, on integrating we see that
$$
\Im {\mathcal{N}}_{a}(z,\zeta) =-\theta \quad (+\,{\rm local \,\,constant}),
$$
i.e., that the analytic completion of the Neumann function behaves
essentially as a symmetric $0$-connection.

Unfortunately, the choice $a=\kappa$ is allowed only for simply
connected domains because in the multiply connected case it violates
(\ref{eq:neumann0}). In fact,
$$
\int_{\partial\Omega}\kappa ds= 2\pi (1-{\texttt{g}}).
$$
On the other hand, we are allowed to take
$a=\frac{\kappa}{1-{\texttt{g}}}$ for any ${\texttt{g}}\ne 1$ and then
$(1-{{\texttt{g}}})\Gamma_{a} (z,\zeta)dz$ becomes a symmetric
meromorphic affine connection on $\hat{\Omega}$:
\begin{equation}\label{eq:gammacurvature}
\Im ((1-{{\texttt{g}}})\Gamma_{a} (z,\zeta)dz) =-\kappa ds
\end{equation}
along $\partial\Omega$. Moreover, it is an affine connection with
simplest possible pole structure: simple poles with residue
$1-{\texttt{g}}$ at each of the points $z=\zeta$ and
$z=\tilde{\zeta}$. For the Neumann function we get
\begin{equation}\label{eq:gammacurvature1}
\Im(1-{{\texttt{g}}}){\mathcal{N}}_{a}(z,\zeta) =-\theta \quad
(+\,{\rm local\,\,constant})
\end{equation}
on $\partial\Omega$.

For the case ${\texttt{g}}= 1$ the situation is actually better,
because in this case there exists a holomorphic affine connection
(in agreement with the fact that in the genus one case the Riemann
surface admits an affine structure). Just let $v$ be the regular
harmonic function in $\Omega$ with Neumann boundary data
$\frac{\partial v}{\partial n}=-\kappa$. This is consistent since
$\int_{\partial\Omega} \kappa ds=0$ when ${\texttt{g}}= 1$, and $v$ is
determined up to an additive constant. Now,
$\Gamma(z)=2\frac{\partial v}{\partial z}$ is the required
holomorphic affine connection.

A slightly different point of view on $N_a(z,\zeta)$ and
$\Gamma_a(z,\zeta)$ is taken in \cite{Schiffer-Hawley}, where
exterior domains are considered, i.e., domains $\Omega\subset
{\mathbb{P}}$ with $\infty\in\Omega$. In this case
$$
\int_{\partial\Omega}\kappa ds = -2\pi (1+{\texttt{g}})
$$
and hence it is possible, for any ${\texttt{g}}$, to choose
$$
a=-\frac{\kappa}{1+{\texttt{g}}}
$$
in the definition of the Neumann function. In place of
(\ref{eq:gammacurvature}), (\ref{eq:gammacurvature1}) one then gets
\begin{equation}\label{eq:gammacurvature2}
\Im ((1+{{\texttt{g}}})\Gamma_{a} (z,\zeta)dz) =\kappa ds,
\end{equation}
\begin{equation}\label{eq:gammacurvature3}
\Im(1+{{\texttt{g}}}){\mathcal{N}}_{a}(z,\zeta) =\theta \quad (+\,{\rm
constant}).
\end{equation}
Thus, $-(1+{\texttt{g}})\Gamma_a(z,\zeta)dz$ is now a symmetric affine
connection on $\Hat{\Omega}$. (In \cite{Schiffer-Hawley} this
connection is denoted $\Gamma (z,\zeta)dz$.) However, it has not
simplest possible pole structure. Besides the pole with residue
$1+{\texttt{g}}$ at $z=\zeta$ it has on $\Omega$ also a pole with
residue $-2$ at $z=\infty$. Similarly on the backside
$\tilde{\Omega}$.

To discover the pole at infinity one has to introduce a regular
(holomorphic) coordinate there, for example $w=\frac{1}{z}$. Then
$$
\{w,z\}_1 dz= \frac{d}{dz}\log \frac{dw}{dz}dz= -2\frac{dz}{z}
=2\frac{dw}{w}.
$$
By (\ref{eq:affconnection}) the representative of
$-(1+{\texttt{g}})\Gamma_a(z,\zeta)dz$ with respect to $w$ is
$$
-(1+{\texttt{g}})\Gamma_a(\frac{1}{w},\zeta)d(\frac{1}{w})-2\frac{dw}{w},
$$
which has residue $-2$ at $w=0$ (the first term is regular at
$w=0$).

As an example, consider the exterior disk: $\Omega
=\{z\in{{\mathbb{C}}}: |z|>1\}\cup\{\infty\}$. For this we have, with
$a=-\frac{\kappa}{1+{\texttt{g}}}=1$,
$$
N_a(z,\zeta)= -\log|z-\zeta|-\log|1-z\bar{\zeta}|+ 2\log |z|
+2\log|\zeta|,
$$
($1<|z|,|\zeta|<\infty$) which gives, for the affine connection
$-(1+{\texttt{g}})\Gamma_a(z,\zeta)dz$,
$$
-\Gamma_a(z,\zeta)dz=\frac{dz}{z-\zeta}+\frac{dz}{z-\frac{1}{\bar{\zeta}}}
-2\frac{dz}{z}.
$$
In terms of $w=\frac{1}{z}$ the same connection is given by
$$
-\Gamma_a(\frac{1}{w},\zeta)d(\frac{1}{w}) -2\frac{dw}{w}
=\frac{dw}{w-\frac{1}{{\zeta}}}+\frac{dw}{w-\bar{\zeta}}-2\frac{dw}{w},
$$
valid actually in all $\Omega$ ($|w|<1$). Here one sees clearly the
pole of residue $-2$ at $w=0$ (i.e., $z=\infty$), along with the
pole with residue one at $w=\frac{1}{\zeta}$ (i.e., $z=\zeta$).


\section{Weighted Bergman spaces}

We have discussed so far the Dirichlet space ${\mathcal A}(\Omega)$
and the Bergman spaces $B(\Omega)$ and $B_{e}(\Omega)$, and remarked
that ${\mathcal A}(\Omega)$ and $B_{e}(\Omega)$ are connected via
the simplest Bol operator, namely $\Lambda_1: F\mapsto dF$. In this
section we shall extend the discussion to a sequence of weighted
Dirichlet spaces ${\mathcal A}_m(\Omega)$ and Bergman spaces
$B_{m}(\Omega)$ and $B_{e,m}(\Omega)$.

Let
$$
d\sigma =\frac{|dz|}{\omega(z)}
$$
be the Poincar\'{e} metric of $\Omega$, characterized by its
constant curvature and being complete:
$$
\begin{cases}
\kappa_{\rm Gauss}=-4,\\ \omega=0 \quad {\rm on\,\,} \partial\Omega.
\end{cases}
$$
As discussed previously it gives rise to a holomorphic projective
connection $q(z)$ and a projective structure, namely that for which
the projective coordinates are the coordinate on the universal
covering surface of $\Omega$ when this is taken to be a disk or a
half-plane. Let $\Lambda_m$ be the corresponding Bol operator, given
in a projective coordinate $t$ by $\Lambda_m
=\frac{\partial^m}{\partial t^m}$.

For $\alpha\in {\mathbb{C}}$ with $\Re \alpha>0$, and $f$, $g$
analytic in $\Omega$, smooth up to $\partial\Omega$, define
\begin{equation}\label{eq:fgalpha}
(f,g)_\alpha=\int_\Omega  f \bar{g}\omega^{\alpha-1}dxdy.
\end{equation}
When $\alpha>0$ is real this is an inner product on a space of
analytic functions, a weighted Bergman space (taken to be complete,
hence a Hilbert space). We denote this space $B_\alpha (\Omega)$. It
will mainly be used for $\alpha=m$ an integer. Clearly,
$B_1(\Omega)$ is the ordinary Bergman space, and it is well-known
that as $\alpha\to 0$ one gets the Hardy space, with the Szeg\"o
inner product given by
$$
\lim_{\alpha\to 0}{\alpha}
(f,g)_\alpha=\frac{1}{\pi}\int_{\partial\Omega}f\bar{g}|dz|
$$

\begin{rem}
Our notations for weighted Bergman spaces deviate from some standard
notations used in the case of the unit disk. In for example
\cite{Hedenmalm-Korenblum-Zhu} the notation $A_{\alpha}^2$ is used
for the Bergman space with weight $\frac{\alpha
+1}{\pi}\omega^\alpha dxdy$, in place of our
$\omega^{\alpha-1}dxdy$, in the definition (\ref{eq:fgalpha}) of the
inner product. In the context of automorphic forms the inner product
$(f,g)_\alpha$ is known as the Petersson inner product.
\end{rem}

If $F$, $G$ are analytic in a neighborhood of $\overline{\Omega}$
and are kept fixed, then the inner product, regarded as a function
of $\alpha$,
$$
\alpha\mapsto (F,G)_\alpha,
$$
is analytic for $\Re \alpha >0$ and has a meromorphic extension to
all ${\mathbb{C}}$, with simple poles at $\alpha=0,-1,-2,\dots$. The
proof of this is implicit in the proof of
Proposition~\ref{prop:Stokes} below. We shall then extend the inner
product $(F,G)_\alpha$ to such values of $\alpha$ by setting
$$
(F,G)_{-m} ={\rm Res\,}_{\alpha=-m}\int_\Omega F
\bar{G}\omega^{\alpha-1}dxdy,
$$
$m=0,1,2,\dots$. For example, for $m=0$ we retrieve the Szeg\"o
inner product.

Now, the main issue is that, for any $m=0,1,2,\dots$, $(F,G)_{-m}$
is an inner product on a generalized Dirichlet space ${\mathcal
A}_m(\Omega)$ of analytic functions and that the Bol operator is an
isometric isomorphism of ${\mathcal A}_m(\Omega)$ onto the subspace
$B_{e,m}(\Omega)$ of `exact' differentials in $B_{m}(\Omega)$):
$$
\Lambda_m: {\mathcal A}_m(\Omega)\to B_{e,m}(\Omega).
$$

In order to elaborate the details of the above we need certain (in
principle known) integral formulas for the $\Lambda_m$.

\begin{prop}\label{prop:Stokes}
With $F$, $G$ (regarded as $\frac{1-m}{2}$-forms) and $g$ (regarded
as a $\frac{1+m}{2}$-form) analytic in $\Omega$, smooth up to
$\partial\Omega$, we have
\begin{equation}\label{eq:Stokes1}
\int_\Omega (\Lambda_m F) \bar{g}\omega^{m-1} dzd\bar{z}
=\I^{m-1}(m-1)!
\int_{\partial\Omega}F\bar{g}(dz)^{\frac{1-m}{2}}(d\bar{z})^{\frac{1+m}{2}},
\end{equation}
\begin{equation}\label{eq:Stokes2}
{\rm Res\,}_{\alpha=-m}\int_\Omega  F \bar{G}\omega^{\alpha-1} dz
d\bar{z}
=\frac{(-\I)^{m+1}}{m!}\int_{\partial\Omega}F\overline{\Lambda_m
G}(dz)^{\frac{1-m}{2}}(d\bar{z})^{\frac{1+m}{2}}.
\end{equation}
In other words,
\begin{equation}\label{eq:Stokes3}
(\Lambda_m F,g)_{m}=\frac{\I^{m}m!}{2}
\int_{\partial\Omega}F\bar{g}(dz)^{\frac{1-m}{2}}(d\bar{z})^{\frac{1+m}{2}},
\end{equation}
\begin{equation}\label{eq:Stokes4}
(F,G)_{-m}=\frac{(-\I)^{m}}{2m!}
\int_{\partial\Omega}F\overline{\Lambda_m
G}(dz)^{\frac{1-m}{2}}(d\bar{z})^{\frac{1+m}{2}}.
\end{equation}
Equation (\ref{eq:Stokes1}) is valid only for $m=1,2,\dots$, while
(\ref{eq:Stokes2}), (\ref{eq:Stokes3}), (\ref{eq:Stokes4}) hold also
for $m=0$ (with $0!=1$).
\end{prop}

By choosing $g=\Lambda_mG$ in (\ref{eq:Stokes3}) we obtain

\begin{cor}
For $m=0,1,2,\dots$,
\begin{equation}\label{eq:Lambdaiso}
(F,G)_{-m}=(\Lambda_m F,\Lambda_m G)_{m},
\end{equation}
hence $(F,F)_{-m}\geq 0$ with equality if and only if $\Lambda_m
F=0$.
\end{cor}

\begin{proof}
The proof of the proposition is based on ideas of Jaak Peetre, and
parts of it have previously been outlined in \cite{Gustafsson-Peetre
2}.

Let $\{\varphi_j\}$ be a smooth partition of unity (i.e., $\sum_j
\varphi_j =1$) on $\overline{\Omega}$ such that each individual
$\varphi_j$ has support within the domain of definition of a
projective coordinate $t_j$. In such a coordinate $\omega$ is of
the form $\omega_j = a |t_j|^2 + bt_j + \bar{b}\bar{t_j} + c$, with
$a$, $c$ real and $|b|^2=ac +1$ (to have curvature $\kappa =-4$).
After an additional M\"obius transformation we may even assume that
$\omega_j = 2\,\Im t_j$ and that $t_j$ takes values in the closed
upper half plane. Note then that $\partial\Omega$ necessarily is
mapped into the real line (since $\omega=0$ on $\partial\Omega$).

In the coordinates $t_j$, the coefficients $F$, $G$, $g$ correspond
to (say) $F_j$, $G_j$, $g_j$ (i.e.,
$F(dz)^{\frac{1-m}{2}}=F_j(dt_j)^{\frac{1-m}{2}}$ etc.) and
$\Lambda_m$ becomes $\frac{\partial^m}{\partial t_j^m}$. On setting
$\psi_j=\varphi_j \circ t_j^{-1}$ we have
$$
\int_\Omega \Lambda_m F\bar{g}\omega^{m-1}dzd\bar{z}
=\sum_j\int_\Omega \Lambda_m F\varphi_j
\bar{g}\omega^{m-1}dzd\bar{z}
$$
$$
= \sum_j\int_{\Im t_j>0} \frac{\partial^m F_j}{\partial t_j^m}\psi_j
\bar{g_j}(2\Im t_j)^{m-1}dt_jd\bar{t_j}.
$$
After an $m$-fold partial integration this becomes
$$
\sum_j(-1)^m\int F_j \frac{\partial^m }{\partial t_j^m}(\psi_j
\bar{g_j}(2\Im t_j)^{m-1})dt_j d\bar{t_j}
$$
$$
+ \sum_j(-1)^m\sum_{k=0}^{m-1}(-1)^k\int_{{\mathbb{R}}}
\frac{\partial^{m-k-1}F_j}{\partial t^{m-1-k}} \frac{\partial^k
}{\partial t_j^k}(\psi_j \bar{g_j}(2\Im t_j)^{m-1})dt_jd\bar{t_j}.
$$
Here all the boundary terms except the last one (with $k=m-1$), and
with $\frac{\partial^k }{\partial t_j^k}=\frac{\partial^{m-1}
}{\partial t_j^{m-1}}$ acting only on $(2\Im t_j)^{m-1}$, vanish
because $\Im t_j=0$ on ${{\mathbb{R}}}$. Moving back to the $z$-plane
and taking covariance (Lemma~\ref{lem:Bol}) into account, the above
expression becomes
$$
\sum_j(-1)^m\int_{\Omega} F \Lambda_m (\varphi_j
\bar{g}\omega^{m-1}) dzd\bar{z} +\I^{m-1}{(m-1)!}\sum_j
\int_{\partial\Omega}F\varphi_j \bar{g} (dz)^{\frac{1-m}{2}}
(d\bar{z})^{\frac{1+m}{2}}
$$
$$
=(-1)^m\int_{\Omega} F \bar{g} \Lambda_m ( \omega^{m-1}) dzd\bar{z}
+\I^{m-1}{(m-1)!} \int_{\partial\Omega}F \bar{g}
(dz)^{\frac{1-m}{2}} (d\bar{z})^{\frac{1+m}{2}}
$$
$$
={\I^{m-1}}{{(m-1)}!}\int_{\partial\Omega}
F\bar{g}(dz)^{\frac{1-m}{2}} (d\bar{z})^{\frac{1+m}{2}}.
$$
This proves (\ref{eq:Stokes1}). A different proof is given in
\cite{Gustafsson-Peetre}.

Now we turn to (\ref{eq:Stokes2}). In order to show that
$\int_\Omega  F \bar{G}\omega^{\alpha-1} dz d\bar{z}$ has a
meromorphic extension and to compute ${\rm
Res\,}_{\alpha=-m}\int_\Omega F \bar{G}\omega^{\alpha-1} dz
d\bar{z}$ we first observe that, with partition of unity and
notations as above,
$$
d(\sum_{k=0}^m(-1)^k\frac{\partial^k}{\partial\bar{t_j}^k}(F_j(2\Im
t_j)^{\alpha +m})
\frac{\partial^{m-k}}{\partial\bar{t_j}^{m-k}}(\psi_j
\bar{G_j})dt_j)
$$
$$
=F_j \frac{\partial^{m+1}}{\partial\bar{t_j}^{m+1}}(\psi
\bar{G_j})(2\Im t_j)^{\alpha +m}d\bar{t_j}dt_j
+(-1)^m\frac{\partial^{m+1}}{\partial\bar{t_j}^{m+1}}(F_j(2\Im
t_j)^{\alpha +m}) \psi_j \bar{G_j} d\bar{t_j}dt_j.
$$
Therefore,
$$
\sum_{k=0}^{m}(-1)^k\int_{{\mathbb{R}}}
\frac{\partial^k}{\partial\bar{t}^k}(F_j (2\Im t_j)^{\alpha +m})
\frac{\partial^{m-k}}{\partial\bar{t}^{m-k}}(\psi_j \bar{G}_j)dt_j
$$
$$
=\int_{\Im t_j>0}F_j
\frac{\partial^{m+1}}{\partial\bar{t}_j^{m+1}}(\psi_j
\bar{G}_j)(2\Im t_j)^{\alpha +m}d\bar{t}_jdt_j
$$
$$
-{(-\I)^{m+1}}{(\alpha +m)\cdots (\alpha+1) \alpha} \int_{\Im t_j>0}
F_j \bar{G}_j \psi_j (2\Im t_j)^{\alpha-1}  d\bar{t}_j dt_j.
$$
The boundary terms vanish whenever $\Re\alpha>0$ and it follows that
for such $\alpha$,
$$
 \int_{\Im t_j>0} F_j
\bar{G}_j \psi_j (2\Im t_j)^{\alpha-1} d\bar{t}_j dt_j
$$
$$
=\frac{\I^{m+1}}{(\alpha +m)\cdots (\alpha+1)\alpha} \int_{\Im
t_j>0} F_j \frac{\partial^{m+1}}{\partial\bar{t}_j^{m+1}}(\psi_j
\bar{G}_j)(2\Im t_j)^{\alpha +m}d\bar{t}_jdt_j.
$$
Here the integral in the right member is an analytic function in
$\alpha$ for $\Re\alpha>-m-1$, hence the left member has a
meromorphic extension to this range, with the residue at $\alpha
=-m$ given by
$$
{\rm Res\,}_{\alpha=-m} \int F_j \bar{G}_j \psi_j (2\Im
t_j)^{\alpha-1} d\bar{t}_j dt_j
=-\frac{(-\I)^{m+1}}{m!} \int F
\frac{\partial^{m+1}}{\partial\bar{t}_j^{m+1}}(\psi_j
\bar{G}_j)d\bar{t}_j dt_j.
$$

Putting the pieces together it follows that also ${\rm
Res\,}_{\alpha=-m}\int_\Omega F \bar{G}\omega^{\alpha-1} dz
d\bar{z}$  has a meromorphic continuation to $\Re\alpha>-m-1$ with
residue at $\alpha =-m$ given by
$$
{\rm Res\,}_{\alpha=-m}\int_\Omega  F \bar{G}\omega^{\alpha-1} dz
d\bar{z}
$$
$$
= -\sum_j {\rm Res\,}_{\alpha=-m}\int_{\Im t_j >0}  F_j
\bar{G_j}\psi_j (2\Im t_j)^{\alpha-1}  d\bar{t}_j dt_j
$$
$$
=\frac{(-\I)^{m+1}}{m!} \sum_j \int_{\Im t_j>0}F_j
\frac{\partial^{m+1}}{\partial\bar{t}_j^{m+1}}(\psi_j
\bar{G}_j)d\bar{t}_jdt_j
$$
$$
=\frac{(-\I)^{m+1}}{m!} \sum_j \int_{\Im t_j>0} d(F_j
\frac{\partial^{m}}{\partial\bar{t}_j^{m}}(\psi_j \bar{G}_j)dt_j)
$$
$$
=\frac{(-\I)^{m+1}}{m!} \sum_j \int_{{\mathbb{R}}} F_j
\frac{\partial^{m}}{\partial\bar{t}_j^{m}}(\psi_j \bar{G}_j)d{t}_j
$$
$$
=\frac{(-\I)^{m+1}}{m!} \sum_j \int_{{\mathbb{R}}} F_j
\frac{\partial^{m}}{\partial\bar{t}_j^{m}}(\psi_j
\bar{G}_j)(dt_j)^{\frac{1-m}{2}} (d\bar{t}_j)^{\frac{1+m}{2}}
$$
$$
=\frac{(-\I)^{m+1}}{m!} \sum_j \int_{\partial\Omega} F
\bar{\Lambda}_m(\varphi_j \bar{G})(dz)^{\frac{1-m}{2}}
(d\bar{z})^{\frac{1+m}{2}}
$$
$$=\frac{(-\I)^{m+1}}{m!}
 \int_{\partial\Omega} F
\bar{\Lambda}_m\bar{G}(dz)^{\frac{1-m}{2}}
(d\bar{z})^{\frac{1+m}{2}}.
$$
This finishes the proof of the proposition.
\end{proof}

The corollary shows that the hermitean form $(F,G)_{-m}$ is positive
semidefinite. Set
$$
A_m(\Omega)=\{F {\rm \,\,analytic \,\,in\,\,}\Omega:
(F,F)_{-m}<\infty \}.
$$
The spaces $B_m(\Omega)$, with inner product $(f,g)_m$, were defined
after (\ref{eq:fgalpha}), and by (\ref{eq:Lambdaiso}) the Bol
operator $\Lambda_m$ is an isometry
$$
\Lambda_m: A_m(\Omega)\to B_m(\Omega).
$$
However, this map is neither injective nor surjective when $m>0$. We
shall say something both about its kernel and its cokernel.

As to the kernel, recall that $\Lambda_m$ in terms of any projective
coordinate $t$ is simply $\Lambda_m =\frac{\partial^m}{\partial
t^m}$. With $F(z)(dz)^{\frac{1-m}{2}}$ a holomorphic differential of
order $\frac{1-m}{2}$, it follows that $\Lambda_m F=0$ if and only
if $F$ becomes a polynomial of degree $\leq m-1$ when expressed in
any projective coordinate. Accordingly, we denote by
$$
P_m (\Omega)= \{F\in A_m(\Omega): F \,\,{\rm holomorphic \,\,in}
\,\, \Omega, \,\, \Lambda_m F=0\}
$$
the kernel of $\Lambda_m$. The image of $\Lambda_m$ is by definition
$$
B_{m,e} (\Omega)= \{\Lambda_m F\in B_m(\Omega): F\in A_m(\Omega)\}.
$$
Then we have the exact sequence
$$
0\to P_m(\Omega)\to A_m(\Omega)\stackrel{\Lambda_m}{\rightarrow}
B_{m,e}(\Omega)\to 0.
$$
In other words,
$$
A_m(\Omega)/P_m(\Omega)\cong B_{m,e}(\Omega)
$$
and $(\cdot,\cdot)_{-m}$ is positive definite on the quotient space
$A_m(\Omega)/P_m(\Omega)$.

Instead of having $(\cdot,\cdot)_{-m}$ defined on a quotient space
of $A_m(\Omega)$ it might be desirable to have it defined on a
subspace of $A_m(\Omega)$. One advantage then is that the elements
in the space become functions, hence it will be possible to discuss
reproducing kernels. There seems to be no canonical choice of such a
subspace, but in principle it is obtained by imposing normalization
conditions. To this purpose, choose a linear operator
$$
{\texttt{a}}: A_m(\Omega)\to {{\mathbb{C}}}^m
$$
which does not degenerate on $P_m(\Omega)$, i.e., such that
${\texttt{a}}(P_m(\Omega))={{\mathbb{C}}}^m$. Then we define
$$
{\mathcal{A}}_m (\Omega)=\{F\in A_m(\Omega): \texttt{a}(F)=0\}
$$
(the dependence on $\texttt{a}$ is suppressed in the notation). Then
${\mathcal{A}}_m (\Omega)\cong A_m(\Omega)/P_m(\Omega)$ and
$\Lambda_m$ is an isometric isomorphism
$$
\Lambda_m: {\mathcal A}_m(\Omega)\to B_{e,m}(\Omega),
$$
as desired.

In the special case $m=1$ we may choose $\texttt{a}$ on the form
$$
\texttt{a} (F) = \int_{\partial\Omega}F\,ads
$$
for some boundary function $a$ with $\int_{\partial\Omega}\,ads\ne0$
and then we retrieve the previously discussed (see after
(\ref{eq:anal1})) Dirichlet space, i.e., ${\mathcal
A}_1(\Omega)={\mathcal A}(\Omega)$.

The cokernel $B_{m}(\Omega)/B_{m,e}(\Omega)$ of $\Lambda_m$ can be
identified with the orthogonal complement of $B_{m,e}(\Omega)$ in
$B_{m}(\Omega)$, which we write simply as $B_{m,e}(\Omega)^\perp$.
And for this finite dimensional space (see \cite{Gustafsson-Peetre
2} for the dimension) we have the following description.

\begin{prop}
$B_{m,e}(\Omega)^\perp$ consists of those elements in
$B_{m}(\Omega)$ which extend to the Schottky double $\Hat{\Omega}$
as holomorphic differentials of order $\frac{1+m}{2}$.
\end{prop}

\begin{proof}
By definition, $g\in B_{m,e}(\Omega)^\perp$ if and only if
$$
(\Lambda_m F,g)_m=0 \quad {\rm for\,\,all}\quad F\in A_m(\Omega).
$$
Using (\ref{eq:Stokes3}) this becomes
$$
\int_{\partial\Omega}F\bar{g}(dz)^{\frac{1-m}{2}}(d\bar{z})^{\frac{1+m}{2}}=0
$$
for all $F\in A_m(\Omega)$. In terms of the function
\begin{equation}\label{eq:hgT}
h(z)=\overline{g(z)}(dz)^{\frac{-1-m}{2}}(d\bar{z})^{\frac{1+m}{2}}
=\overline{g(z)} T(z)^{-1-m},
\end{equation}
defined on $\partial\Omega$, this becomes
$$
\int_{\partial\Omega}F hdz =0
$$
for all $F\in A_m(\Omega)$, which implies that $h$ has a holomorphic
extension to $\Omega$. This proves the proposition because $h$
represents the continuation of $g$ to the back-side of the Schottky
double. Note that (\ref{eq:hgT}) is an instance of
(\ref{eq:halforderdiff}).
\end{proof}

\begin{ex}
The case $m=1$ is the well-known \cite{Schiffer-Spencer} fact that $B_{1,e}(\Omega)^\perp$ consists of
the abelian differentials of the first kind.
\end{ex}


\section{Reproducing kernels}

\subsection{Reproducing kernels for weighted Bergman spaces}

Each of the weighted Bergman spaces $B_m(\Omega)$ and
$B_{m,e}(\Omega)$, $m\geq 0$, have reproducing kernels
$K_m(z,\zeta)$, $K_{m,e}(z,\zeta)$, which extend to the Schottky
double as differentials of order $\frac{1-m}{2}$ in each variable.
The continuations to the backside are represented by the adjoint
kernels $L_m(z,\zeta)$, $L_{m,e}(z,\zeta)$, which have singularities
$$
L_m(z,\zeta)=\frac{(-\I)^{m-1}}{\pi(z-{\zeta})^{m+1}} +{\rm less\,\,
singular\,\, terms}
$$
(similarly for $L_{m,e}(z,\zeta)$). See \cite{Gustafsson-Peetre 2}
for details and proofs. Recall also that $K_1(z,\zeta)$ is the
ordinary Bergman kernel and $K_0(z,\zeta)$ the Szeg\"o kernel.

For the spaces $A_m(\Omega)$ and ${\mathcal{A}}_m(\Omega)$  the
situation is not quite that good. But at least
${\mathcal{A}}_m(\Omega)$ is a Hilbert space of functions for which
all point evaluations are continuous linear functionals, and hence
it has a reproducing kernel, which we denote
${\mathcal{K}}_m(z,\zeta)$. This should be thought of as a
differential of order $\frac{1-m}{2}$ in each of $z$ and
$\bar{\zeta}$, but its extension to the Schottky double is
cumbersome because of appearance of branch points and
multi-valuedness on the backside, as have already been observed in
the case of Dirichlet space, $m=-1$,
Example~\ref{ex:kernelsunitdisk}. Still it is possible to define
indirectly a kind of (multi-valued) adjoint kernel, which we denote
${\mathcal{L}}_m(z,\zeta)$. In fact, we simply define
${\mathcal{L}}_m(z,\zeta)$ to be any solution of equation
(\ref{eq:LambdaLL}) below.

\begin{thm}
Let ${\mathcal{K}}_m(z,\zeta)$ denote the reproducing kernel for
${\mathcal{A}}_m(\Omega)$ and ${\mathcal{L}}_m(z,\zeta)$ the adjoint
kernel. Then,
\begin{equation}\label{eq:LambdaKK}
\Lambda_m\bar{\Lambda}_m {\mathcal{K}}_m(z,\zeta)=K_{m,e}(z,\zeta),
\end{equation}
\begin{equation}\label{eq:LambdaLL}
\Lambda_m {\Lambda}_m {\mathcal{L}}_m(z,\zeta)=L_{m,e}(z,\zeta),
\end{equation}
where the first $\Lambda_m$ acts on the $z$-variable and the second
one on $\zeta$. The leading term in the singularity of
${\mathcal{L}}_m$ is given by
$$
{\mathcal{L}}_m(z,\zeta) = -\frac{\I^{m-1}}{\pi m!(m-1)!}
(z-{\zeta})^{m-1}\log (z-{\zeta}) + {\rm less\,\, singular\,\,
terms}.
$$
\end{thm}

\begin{proof}
The formula (\ref{eq:LambdaKK}) follows immediately by letting
$\Lambda_m$ act on both members, as functions of $\zeta$, in the
defining equation (reproducing property)
$$
F(\zeta)=(F,{\mathcal{K}}_m(\cdot,\zeta))_{-m} \quad
(F\in{\mathcal{A}}_m(\Omega))
$$
for  ${\mathcal{K}}_m(z,\zeta)$ and then applying
(\ref{eq:Lambdaiso}).

As (\ref{eq:LambdaLL}) was simply taken as the definition of
${\mathcal{L}}_m(z,\zeta)$ it just remains to prove the form of the
singularity. This is a matter of computation, which we omit. One has
to check that the claimed singularity for ${\mathcal{L}}_m(z,\zeta)$
match with that of ${{L}}_{m,e}(z,\zeta)$ (which is the same as that
of ${{L}}_m(z,\zeta)$) under (\ref{eq:LambdaLL}). The computation
may be performed in a projective coordinate because the assertion
only concerns the leading term of the singularity, which is the same
in any coordinate.
\end{proof}

\begin{ex}\label{ex:kernels}
For the unit disk we have
$$
{\mathcal{K}}_m(z,\bar{\zeta}) = -\frac{\I^{m-1}}{\pi m!(m-1)!}
(1-z{\bar{\zeta}})^{m-1}\log (1-z{\zeta}) +\sum_{k=0}^{m-1}
a_k(z)\bar{\zeta}^k +\sum_{k=0}^{m-1} \overline{a_k(\zeta)}{z}^k,
$$
$$
{\mathcal{L}}_m(z,\zeta) = -\frac{\I^{m-1}}{\pi m!(m-1)!}
(z-{\zeta})^{m-1}\log (z-{\zeta}) +\sum_{k=0}^{m-1} a_k(z){\zeta}^k
+\sum_{k=0}^{m-1} {a_k(\zeta)}{z}^k,
$$
where the analytic functions $a_k$ depend on the normalization
chosen in the definition of ${\mathcal{A}}_m(\Omega)$. If the
normalization for example is that $F(0)=\dots =F_{m-1}(0)=0$ for
$F\in{\mathcal{A}}_m(\Omega)$, then all the $a_j$ are zero and the
kernels consist of just the first term.
\end{ex}


\subsection{The reproducing kernel as a resolvant}

Interchanging the roles of $F$ and $G$  in (\ref{eq:Stokes4}) gives,
after conjugation,
$$
(F,G)_{-m}=-\frac{(-\I)^{m}}{2m!} \int_{\partial\Omega}(\Lambda_m
F)\overline{G}(dz)^{\frac{1+m}{2}}(d\bar{z})^{\frac{1-m}{2}}.
$$
With $F\in{\mathcal{A}}_m(\Omega)$ and $ G={\mathcal{K}}_m(z,\zeta)$
we get
\begin{equation}\label{eq:resolvent}
F(\zeta)=-\frac{(-\I)^{m}}{2m!} \int_{\partial\Omega}\Lambda_m
F(z)\overline{{\mathcal{K}}_{m}(z,\zeta)}(dz)^{\frac{1+m}{2}}
(d\bar{z})^{\frac{1-m}{2}}.
\end{equation}
This formula says that ${{\mathcal{K}}}_{m}$ represents the inverse,
or resolvant, of $\Lambda_m$, in terms of the boundary values. In
the case $m=2$, (\ref{eq:resolvent}) becomes
\begin{equation}\label{eq:resolvent2}
F(\zeta)=\frac{1}{4} \int_{\partial\Omega}\Lambda_2
F(z)\overline{{\mathcal{K}}_{2}(z,\zeta)}(dz)^{\frac{3}{2}}
(d\bar{z})^{-\frac{1}{2}} \quad (\zeta\in\Omega),
\end{equation}
which resembles a formula of R.~J.~V.~Jackson \cite{Jackson}.

\begin{ex}
For the unit disk (\ref{eq:resolvent2}) becomes
$$
F(\zeta)=\frac{1}{2\pi\I} \int_{\partial{\mathbb{D}}}
F''(z)(1-\bar{z}\zeta)\log(1-\bar{z}\zeta)zdz,
$$
if ${\mathcal{A}}_2({\mathbb{D}})$ is normalized by $F(0)=F'(0)=0$,
(i.e., $\texttt{a} (F)=(F(0), F'(0))$).
\end{ex}

Equation (\ref{eq:resolvent2}) is valid for $\zeta\in\Omega$. For
$\zeta\ne\overline\Omega$ it does not make sense because, as is seen
clearly in Examples~\ref{ex:kernelsunitdisk} and \ref{ex:kernels},
even though the kernel extends analytically across $\partial\Omega$,
the extended kernel is not single-valued in any neighborhood of
$\overline\Omega$. However, it is possible to let $\zeta\in\Omega$
approach $\partial\Omega$ from inside and get a sensible version of
(\ref{eq:resolvent2}) for $\zeta\in\partial\Omega$. This is what
Jackson \cite{Jackson} does. The equation (\ref{eq:resolvent2}) then
solely deals with objects defined on $\partial\Omega$ and it is
appropriate to consider the kernel ${\mathcal{K}}_{2}(z,\zeta)$ as a
resolvent of $\Lambda_2$. This kernel is not smooth on the boundary,
as is seen in the example of the unit disk, there is a discontinuity
in the first derivative caused by the jump in the imaginary part of
$\log (1-\bar{z}\zeta)$ as $z$ passes $\zeta$.

\begin{rem}
Since ${{\mathcal{K}}_{2}(z,\zeta)}$ is to be considered as a form
of degree $-1/2$ in each of $z$ and $\bar\zeta$ it is natural to
compare it with $1/{K_0(z,\zeta)}$, $1/\sqrt{K_1(z,\zeta)}$ and, on
the diagonal, compare all these with $\omega(z)$. Certainly, $1/{\pi
K_0(z,z)}=1/\sqrt{\pi K_1(z,z)}=\omega(z)$ in the simply connected
case. However, ${\mathcal{K}}_{2}(z,\zeta)$ is of a different nature
because its continuation ${\mathcal{L}}_{2}(z,\zeta)$ to the
back-side of the Schottky double is not single-valued.

Similar remarks as above apply to the equation (\ref{eq:resolvent})
for all values of $m$.
\end{rem}


\section{The prepotential of a second order linear
DE}\label{sec:prepotential}

\subsection{The method of Faraggi and Matone}

Here we shall discuss a further topic related to projective
structures. Consider a differential equation of the form
(\ref{eq:secondorderDE}) in general. We write it as
\begin{equation}\label{eq:secondorderDE1}
u''+\frac{1}{2}Q(z)u=0,
\end{equation}
where $Q(z)$ is a holomorphic function, say in a neighborhood of
$z=0$. Note that the Wronskian
\begin{equation}\label{eq:Wronskian}
W=W(u_1,u_2)= u_1 u'_2 - u'_1 u_2
\end{equation}
of any pair $u_1$, $u_2$ of solutions is constant. If $u_1(0)\ne 0$,
then any other solution $u_2$ is obtained from $u_1$ by
\begin{equation}\label{eq:u1tou2}
u_2(z)= u_1(z)(\int_0^z \frac{W d\zeta}{u_1(\zeta)^2}+C)
\end{equation}
for suitable constants $W$ (which then becomes the Wronskian) and
$C$.

An interesting approach to the problem of producing further
solutions from a given one has been considered for instance by A.~E.~Faraggi and
M.~Matone \cite{Faraggi-Matone}. The authors introduce a
``prepotential" ${\mathcal F}(u)={\mathcal F}_{u_1}(u)$, a function of
a complex variable $u$, which functionally depends also on the first
solution $u_1$ (alternatively, depends on $Q$ if initial conditions
for $u_1$ are specified), to the effect that a second solution $u_2$
is obtained by taking the derivative with respect to $u$:
\begin{equation}\label{eq:prepotential}
u_2 (z) = \frac{d{{\mathcal F}}_{u_1}(u)}{d u}\big|_{u=u_1(z)}.
\end{equation}
As noted in \cite{Faraggi-Matone}, ${\mathcal F}(u)$ is
essentially a Legendre transform, namely of the independent
variable $z$ considered as a function of any projective coordinate
$t$ associated to the projective structure given by $Q(z)$. We
proceed to explain briefly these issues, going slightly beyond
\cite{Faraggi-Matone}.

Fixing a value $W\ne 0$ of the Wronskian, consider pairs
$u_1=u_1(z)$, $u_2=u_2(z)$ of solutions, holomorphic in a
neighborhood of $z=0$ and subject to (\ref{eq:Wronskian}), i.e.,
\begin{equation}\label{eq:Wronskian1}
u_1 du_2 - u_2 du_1 =W dz.
\end{equation}
We assume that $u_1(0)\ne 0$ and set
\begin{equation}\label{eq:st}
s=u_1^2, \quad t =\frac{u_2}{u_1}.
\end{equation}
Then
$$
\frac{dt}{dz}=\frac{W}{s}, \quad \{t, z\}_2=Q(z),
$$
and also, for example, $\{z,t\}_1=\frac{1}{W}\frac{ds}{dz}$. In
terms of $s$ and $t$, (\ref{eq:Wronskian1}) becomes
\begin{equation}\label{eq:Wronskian2}
s\,dt=Wdz.
\end{equation}
We note also that, in terms of $s$ and $t$, the $(m-1)$-fold
symmetric product $S^{m-1}(L)$ of the operator
$L=\frac{d^2}{dz^2}+\frac{1}{2}Q(z)$, mentioned briefly after
Lemma~\ref{lem:Bol}, will have solutions generated by
$s^{\frac{m-1}{2}}t^k$, $k=0,1,\dots, m-1$.

In a neighborhood of $z=0$ we can invert $t=t(z)$ to consider $z$ as
a function of $t$: $z=z(t)$. Then $\frac{dz}{dt}=\frac{s}{W}$.
Assuming for a moment that $\frac{d^2z}{dt^2}\ne 0$ we can form the
Legendre transform of $Wz(t)$. It is
\begin{equation}\label{eq:Legendre}
{\mathcal L} (s) = st - W z(t) \quad {\rm with\,\,}t\,\,{\rm
chosen\,\, so\,\,that}\,\, s=W\frac{dz(t)}{dt}.
\end{equation}
Note that the final equation assures that the variable $t$ keeps its
meaning as $t=\frac{u_2}{u_1}$ when $s=u_1^2$.

By a direct computation, or by using that the Legendre transform is
involutive, one realizes that
$$
\frac{d{\mathcal L}(s)}{ds}=t,
$$
with $s$, $t$ related to $u_1$, $u_2$ as above. This shows that
$$
u_2= u_1 \cdot t =
\frac{1}{2}\frac{ds}{du_1}\frac{d{\mathcal L}(s)}{ds}
=\frac{1}{2}\frac{d{\mathcal L}(u_1^2)}{du_1},
$$
hence establishes the Legendre transform as essentially the desired
prepotential:
$$
{\mathcal F}(u) = \frac{1}{2} {\mathcal L}(u^2).
$$

An alternative and slightly more general approach, which makes sense
also if $\frac{d^2 z}{dt^2}= 0$, is to consider $u_1$, $u_2$ and $z$
as independent variables, or coordinates, in a three dimensional
space. Then (\ref{eq:st}) is simply a coordinate transformation (in
two of the variables) and (\ref{eq:Wronskian1}),
(\ref{eq:Wronskian2}) should be thought of as defining a contact
structure (see, e.g., \cite{Arnold}, Appendix~4). With
\begin{equation}\label{eq:Legendre1}
{\mathcal L}=st-Wz= u_1 u_2 -Wz,
\end{equation}
now considered as a function in the three dimensional space, we have
$$
d{\mathcal L} = t ds = 2u_2 du_1
$$
when the contact structure is taken into account. This gives again
$$t=\frac{d{\mathcal L}}{d s}\big|_{sdt=Wdz}, \quad
u_2=\frac{1}{2}\frac{d{\mathcal L}}{d u_1}\big|_{u_1 du_2-u_2
du_1=Wdz}.
$$

To derive explicit formulas for ${\mathcal F}(u)$ it is convenient
to assume that the first solution $u_1$ is invertible near $z=0$,
i.e., to assume that $u_1'(0)\ne 0$. Since the independent variable
$u$ in ${\mathcal F}(u)$ will finally, in (\ref{eq:prepotential}),
be assigned to have the value $u_1(z)$, it is natural to invert this
relation to $z=u_1^{-1}(u)$. Substituting into (\ref{eq:u1tou2}),
(\ref{eq:Legendre1}) and recalling that $\mathcal{F}$ is half of
${\mathcal L}$ one arrives at
\begin{equation}\label{eq:prepotential2}
{\mathcal F}_{u_1}(u)= \frac{1}{2}(u_1u_2 -Wz) =\frac{1}{2}
(u^2\int_0^{u_1^{-1}(u)}\frac{W
d\zeta}{u_1(\zeta)^2}+Cu^2-Wu_1^{-1}(u)).
\end{equation}
After a partial integration and a change of variable this gives the
formula presented in \cite{Faraggi-Matone}:
\begin{equation}\label{eq:prepotential3}
{\mathcal F}_{u_1}(u)=u^2(\int_{u_1(0)}^{u}
\frac{Wu_1^{-1}(\eta)d\eta}{\eta^3}+\frac{C}{2}) =Wu^2\int
\frac{u_1^{-1}(u)du}{u^3}.
\end{equation}
Here the final integral is an ``indefinite integral''. One readily
verifies that (\ref{eq:prepotential}) indeed holds.

As a final issue one may notice (as in \cite{Faraggi-Matone})
that the prepotential satisfies a third order differential equation.
This is obtained by considering $u_2$ as a function of $u_1$ (in
place of $z$). This renders (\ref{eq:secondorderDE1}) of the form
$$
\frac{d^2 u_2}{du_1^2}= \frac{Q}{2W^2}(u_1 \frac{du_2}{du_1}-u_2)^3.
$$
Replacing $u_1$ by $u$ and $u_2$ by $\frac{d{\mathcal F}(u)}{d u}$
gives
$$
\frac{d^3{\mathcal F}(u)}{d u^3}
=\frac{Q}{2W^2}(u\frac{d^2{\mathcal F}(u)}{d
u^2}-\frac{d{\mathcal F}(u)}{d u})^3.
$$


\subsection{Example}

Consider the differential equation
$$
u''+u =0,
$$
i.e., (\ref{eq:secondorderDE1}) with $Q=2$, and choose
$$
u_1 (z)=\cos z.
$$
Then
$$
u_2 (z)= \cos z (\int_0^z \frac{Wd\zeta}{\cos^2 \zeta} +C)= W\sin
z+C\cos z.
$$
Using (\ref{eq:prepotential2}) the prepotential becomes
$$
{\mathcal F}_{u_1}(u)= \frac{1}{2}(u^2 \int_0^{\arccos u}
\frac{Wd\zeta}{\cos^2 \zeta}+Cu^2-W\arccos u) =\frac{W}{2}
(u\sqrt{1-u^2}-\arccos u)+\frac{C}{2}u^2.
$$

Somewhat surprisingly perhaps, we start with a very simple
differential equation and arrive at a rather complicated
prepotential. Note that $u_1 (z)=\cos z$ violates the assumption of
being invertible at $z=0$. This causes ${\mathcal F}_{u_1}(u)$ to
have a branch point at $u=u_1(0)=1$.


\section{Glossary of notations}

\begin{itemize}

\item ${\mathbb{D}}(a,r)=\{z\in{\mathbb{C}}:|z-a|<r\}$,
${\mathbb{D}}={\mathbb{D}}(0,1)$.

\item ${\mathbb{P}}={\mathbb{C}}\cup\{\infty\}$, the Riemann sphere.

\item $A_{r,R}=\{z\in{{\mathbb{C}}}: r<|z|<R\}$.

\item $\Omega$ usually denotes a finitely connected domain
in ${\mathbb{P}}$, with boundary components denoted
$\Gamma_0,\dots\Gamma_{{\texttt{g}}}$, ${\texttt{g}}\geq 0$ (each
$\Gamma_j$ consisting of more than one point).

\item $\Hat{\Omega}=\Omega \cup \partial \Omega\cup\Tilde{\Omega}$,
the Schottky double of $\Omega$, a symmetric compact Riemann surface
of genus ${\texttt{g}}$.

\item $J:\Hat{\Omega}\to\Hat{\Omega}$, the anticonformal involution
on $\Hat{\Omega}$.

\item $\alpha_j$, $\beta_j$ ($j=1,\dots,{\texttt{g}}$), canonical
homology basis on a compact Riemann surface of genus ${\texttt{g}}$.

\item ${d}(z,A)={\rm dist\,}(z,A)$, distance from a point $z$ to a set $A$;
${d}(z)={d}(z,\partial\Omega)$ if $z\in\Omega$.

\item $^*\omega$, the Hodge star of a differential form $\omega$, for example
$^*(adx+bdy)=-bdx+ady$, $^*dz=-\I dz$. If $u$ is a harmonic function
then $^*du=d(u^*)$ where $u^*$ is a harmonic conjugate of $u$.

\item $V(z,w;a,b)$, a fundamental potential on a compact Riemann surface.

\item $E(z,\zeta)$, the Schottky-Klein prime function.

\item $G(z,\zeta)$, the Green's function of a domain $\Omega\subset{\mathbb{P}}$.

\item $H(z,\zeta)$, the regular part of the Green's function

\item ${\mathcal G}(z,\zeta) =G(z,\zeta)+\I G^*(z,\zeta)$,
the analytic completion of $G(z,\zeta)$ with respect to $z$.
Similarly for ${\mathcal H}(z,\zeta)$.

\item $P(z,\zeta)=-\frac{\partial G(z,\zeta)}{\partial n_z}$,
the Poisson kernel ($\frac{\partial }{\partial n}$ outward normal
derivative).

\item $G_\gamma(z,\zeta)$, the hydrodynamic Green function with circulations
$\gamma$.

\item ${\mathcal G}_\gamma(z,\zeta)$,
the analytic completion of $G_\gamma(z,\zeta)$ with respect to $z$.

\item $N_a(z,\zeta)$, the Neumann function with Neumann data $-a$
on $\partial\Omega$.

\item ${\mathcal{N}}_a(z,\zeta)$, the analytic completion of $N_a(z,\zeta)$
with respect to $z$

\item $\upsilon_{a-b}$, the abelian differential of the third kind
with poles at $a$, $b$ and having purely imaginary periods.

\item $\omega_{a-b}$, the abelian differential of the third kind
with poles at $a$, $b$ and with vanishing $\alpha_j$-periods.

\item $\tilde{\omega}_{a-b}$, the abelian differential of the third kind
with poles at $a$, $b$ and with vanishing $\beta_j$-periods.

\item ${A}_m(\Omega)$, ${A}(\Omega)={A}_1(\Omega)$:
spaces of analytic functions in $\Omega$ ($m=0,1,2,\dots$).
Essentially weighted Bergman spaces of negative index.
${A}_0(\Omega)$ is Hardy space.

\item ${\mathcal{A}}_m(\Omega)$,
${\mathcal{A}}(\Omega)={\mathcal{A}}_1(\Omega)$: subspace of
${A}_m(\Omega)$ defined by a normalization (to make the inner
product positive definite). ${\mathcal{A}}(\Omega)$ is Dirichlet
space.

\item $B_m(\Omega)$, $B(\Omega)=B_1(\Omega)$: Weighted Bergman
spaces of positive index  $m=1,2,\dots$. $B(\Omega)$ ordinary
Bergman space.

\item $B_e(\Omega)$, $B_{m,e}(\Omega)$ ($m=1,2,\dots$), the subspaces of
$B(\Omega)$, $B_m(\Omega)$ consisting of ``exact'' differentials
with respect to $\Lambda_m$.

\item $P_m(\Omega)=\{F\in A_m(\Omega): \Lambda_m F =0\}
=\{F\in A_m(\Omega): (f,f)_{-m}=0\}$, the space of $\frac{1-m}{2}$:s
order differentials in $\Omega$ which expressed in any projective
coordinate are polynomials of degree $\leq m-1$ ($m=1,2,\dots$).

\item $H(\Omega)$, $H_e(\Omega)$: spaces of harmonic functions.

\item $(f,g)_m$, the inner product on $B_m(\Omega)$ for $m\geq 1$,
on ${\mathcal{A}}_{-m}(\Omega)$ for $m\leq 0$.

\item $D(f,g)$, Dirichlet inner product.

\item ${\mathcal K}_m(z,\zeta)$, reproducing kernel for ${\mathcal
A}_m(\Omega)$; ${\mathcal K}(z,\zeta)={\mathcal K}_1(z,\zeta)$.

\item ${\mathcal L}_m(z,\zeta)$, adjoint kernel for ${\mathcal
A}_m(\Omega)$; ${\mathcal L}(z,\zeta)={\mathcal L}_1(z,\zeta)$.

\item $K(z,\zeta)=-\frac{2}{\pi}
\frac{\partial^2 G(z,\zeta)}{\partial z\partial \bar{\zeta}}$, the
Bergman kernel, reproducing kernel for $B(\Omega)$.

\item $K_e(z,\zeta)=-\frac{2}{\pi}
\frac{\partial^2 G_\gamma(z,\zeta)}{\partial z\partial
\bar{\zeta}}$, the reduced Bergman kernel, reproducing kernel for
$B_e(\Omega)$.

\item $K_m(z,\zeta)$, the reproducing kernel for $B_m(\Omega)$.

\item $K_{m,e}(z,\zeta)$, the reproducing kernel for $B_{m,e}(\Omega)$.

\item $L(z,\zeta)=\frac{2}{\pi}\frac{\partial^2 G(z,\zeta)}{\partial z\partial
{\zeta}}$, the Schiffer kernel, or adjoint Bergman kernel.
(Similarly for $L_m(z,\zeta)$, $L_{m,e}(z,\zeta)$.)

\item $\ell(z,\zeta)$, the regular part of $L(z,\zeta)$.

\item $k(z,\zeta)$, $k_e(z,\zeta)$, reproducing kernels for $H(\Omega)$ and
$H_e(\Omega)$ respectively.

\item $M(z,\zeta)$, the ordinary Martin kernel.

\item $F(z,\zeta)$, the Martin kernel for the gradient structure.

\item $ds$, arc-length differental along a curve.

\item $\kappa$, the curvature of a curve in the
complex plane (and sometimes short for $\kappa_{\rm Gauss}$).

\item $d\sigma = \rho(z)|dz|=\frac{|dz|}{\omega(z)}=e^{p(z)}|dz|$,
a hermitean metric, e.g., the Poincar\'e metric.


\item $\kappa_{\rm Gauss}$, the Gaussian curvature of a hermitean metric.

\item $S(z)$, the Schwarz function of an analytic curve ($S(z)=\bar{z}$
on the curve, $S(z)$ analytic in a full neighborhood of the curve).

\item $T(z)=\frac{dz}{ds}$, the unit tangent vector on a curve (and its
analytic extension to a neighborhood of the curve, if the curve is
analytic).

\item $\{z,t\}_k$, a differential expression appearing in the definition of a
$k$-connection ($k=0,1,2$). (The Schwarzian derivative if $k=2$.)

\item $\nabla_k$, the covariant derivative for an affine connection,
when acting on $k$:th order differentials.

\item $\Lambda_m$, the $m$:th order Bol operator, the covariant derivative
for a projective connection, acting on differentials of order
$\frac{1-m}{2}$.

\item $\nabla$, the gradient.

\item ${\mathcal{L}}$, the Legendre transform.

\item $S^m$, the $m$-fold symmetric product (of a differential operator).

\item $\partial=\partial_z=\frac{\partial}{\partial z}$,
$\bar{\partial}=\bar{\partial}_z=\frac{\partial}{\partial \bar{z}}$.

\end{itemize}



\end{document}